\documentclass[english,reqno]{amsart}

\usepackage{amsmath, appendix, ulem}
\usepackage[alphabetic,nobysame]{amsrefs}
\usepackage{amssymb, color}
\usepackage[margin=1in]{geometry}

\usepackage{mathrsfs}
\usepackage{graphicx}
\usepackage{subfig}
\usepackage{float}
\usepackage{epsf}
\usepackage{hyperref}
\usepackage{titletoc}

\graphicspath{{../Figures/}}

\numberwithin{equation}{section}

\newtheorem{lem}{Lemma}[section]
\newtheorem{thm}{Theorem}[section]
\newtheorem{prop}[thm]{Proposition}

\theoremstyle{remark}
\newtheorem{rmk}{Remark}[section]

\renewcommand{\tilde}{\widetilde}
\renewcommand{\hat}{\widehat}
\renewcommand{\bar}{\overline}

\newcommand{\nn}{\nonumber}
\newcommand{\R}{{\mathbb R}}

\newcommand{\del}{\partial}
\newcommand{\Denote}{\stackrel{\Delta}{=}}
\newcommand{\dx}{ \, {\rm d} x}
\newcommand{\dv}{ \, {\rm d} v}

\newcommand{\dmu}{\, {\rm d} \mu}

\newcommand{\One}{\boldsymbol{1}}
\newcommand{\La}{\left\langle}
\newcommand{\Ra}{\right\rangle}

\newcommand{\ii}{I}

\newcommand{\CalB}{{\mathcal{B}}}
\newcommand{\CalD}{{\mathcal{D}}}

\newcommand{\CalK}{{\mathcal{K}}}
\newcommand{\CalL}{{\mathcal{L}}}
\newcommand{\CalP}{{\mathcal{P}}}
\newcommand{\CalR}{{\mathcal{R}}}

\newcommand{\Vecf}{\vec f}
\newcommand{\Vech}{\vec h}
\newcommand{\Vecg}{\vec g}
\newcommand{\Vf}{f}

\newcommand{\VU}{\vec U}
\newcommand{\VecL}{{\CalL}}

\newcommand{\VecX}{\vec X}
\newcommand{\Vecphi}{{{\vec \phi}}}

\newcommand{\PK}{{\bf PK}}
\newcommand{\PL}{{\bf PL}}

\newcommand{\BA}{{\bf A}}

\newcommand{\Be}{{\bf e}}

\newcommand{\NullL} {{\rm Null} \, \VecL}

\newcommand{\domega}{\, {\rm d}\omega}
\newcommand{\dsigma}{\, {\rm d}\sigma}

\newcommand{\I}{{\rm I}}
\newcommand{\Vd}{\bar v}

\newcommand{\mc}[1]{\mathcal{#1}}

\newcommand{\TT}{\mathrm{T}}

\newcommand{\VV}{\mathbb{V}}

\newcommand{\rd}{\mathrm{d}}

\newcommand{\abs}[1]{\left\lvert#1\right\rvert}
\newcommand{\norm}[1]{\left\lVert#1 \, \right\rVert}
\newcommand{\vint}[1]{\left\langle#1\right\rangle}
\newcommand{\viint}[2]{\left\langle#1, \, #2 \,\right\rangle}
\newcommand{\vpran}[1]{\left(#1\right)}

\DeclareMathOperator{\Span}{span}

\begin{document}

\title{Half-space Kinetic Equations with General Boundary Conditions}

\author{Qin Li} 
\address{Mathematics Department, University of Wisconsin-Madison, 480 Lincoln Dr., Madison, WI 53705 USA.}
\email{qinli@math.wisc.edu}
\author{Jianfeng Lu}
\address{Departments of Mathematics, Physics, and Chemistry, Duke University, Box 90320, Durham, NC 27708 USA.}
\email{jianfeng@math.duke.edu}
\author{Weiran Sun}
\address{Department of Mathematics, Simon Fraser University, 8888 University Dr., Burnaby, BC V5A 1S6, Canada}
\email{weirans@sfu.ca}

\date{\today}

\thanks{We would like to express our gratitude to the support from the
  NSF research network grant RNMS11-07444 (KI-Net). The research of Q.L.~was supported in part by the AFOSR MURI grant FA9550-09-1-0613 and the National Science Foundation under award DMS-1318377. The research of J.L.~was supported in part by the Alfred P.~Sloan Foundation and the National Science Foundation under award DMS-1312659 and DMS-1454939. The research of W.S.~was supported in part by the Simon Fraser University President's Research Start-up Grant PRSG-877723 and NSERC Discovery Individual Grant \#611626.}

\begin{abstract}
  We study half-space linear kinetic equations with general boundary
  conditions that consist of both given incoming data and various type
  of reflections, extending our previous work \cite{LiLuSun} on
  half-space equations with incoming boundary conditions. As in
  \cite{LiLuSun}, the main technique is a damping adding-removing
  procedure. We establish the well-posedness of linear (or linearized)
  half-space equations with general boundary conditions and
  quasi-optimality of the numerical scheme. The numerical method is
  validated by examples including a two-species transport equation, a
  multi-frequency transport equation, and the linearized BGK equation
  in 2D velocity space.
\end{abstract}

\maketitle

\section{Introduction}
In this paper we propose an efficient numerical method for linear
half-space kinetic equations with general boundary conditions
\begin{equation}
  \begin{aligned} \label{eq:kinetic-1}
    & \mu \del_x \Vecf + \CalL \Vecf  = 0 \,, && \text{in } (0, \infty)\times \mathbb{V} \,,
    \\
    & \Vecf \big|_{\mu > 0} = \Vech(\mu) + \CalK \vpran{\Vecf \big|_{\mu < 0}} \,, && \text{on } x = 0 \,,
\end{aligned}
\end{equation}
where the density function $\Vecf(x, v) \in \R^m$ with $m \geq 1$ for
$x \in [0, \infty)$ and
$v = (\mu, \bar v) = (\mu, v_2, \cdots, v_d) \in \VV$. Typical
examples for the velocity space $\VV$ include the whole space $\R^d$, as in the case of the Boltzmann equation, and $\VV = [-1,
1]$ as in the case of the transport equation. 
By allowing higher-dimensions in $\Vecf$ and $v$, we include
multi-species models and models with multi-dimensional velocity
variables such as the linearized Boltzmann and linearized BGK equations. The setup also
includes the multi-frequency case where the frequency variable can be
treated 
as an index for
multi-species after discretization.

The operator $\CalL$ in~\eqref{eq:kinetic-1} is a linear operator,
examples of which include the scattering operator in the linear
transport equations, the collision operator in the linearized
Boltzmann equations and the linearized BGK equation. The operator
$\CalK$ is the boundary operator which characterizes various types of
reflections at the boundary.
Two classical examples for the reflections are the diffuse and
specular reflections. 
It will be
discussed in details in Section~\ref{sec:K} that our method applies to a
general class of boundary operators 
including Maxwell boundary condition (linear combination of the
diffuse and specular reflection), bounce-back reflection, and also the
more general (linearized) Cercignani-Lampis boundary condition.

It is well known that to ensure the well-posedness of
equation~\eqref{eq:kinetic-1}, one needs to prescribe suitable
boundary conditions at $x=\infty$. The precise conditions were first
formulated in \cite{CoronGolseSulem:88} for the linearized Boltzmann
equations with prescribed incoming data. This type of well-posedness
result has been extended to general linear/linearized half-space
equations and weakly nonlinear half-space equations with both incoming
and Maxwell boundary conditions (see e.g., \cites{Golse:08, ST2011,
  UkaiYangYu:03}) and also to discrete Boltzmann equation with general boundary conditions \cites{Bernhoff:08, Bernhoff:10}.  This is also the setting that we use for
developing numerical methods. Now we briefly explain the details of
the formulation of the boundary condition at infinity. Denote the null
space of $\CalL$ as $\NullL$ which is assumed to be
finite-dimensional.  Let $\CalP$ be the $L^2$-projection operator onto
$\NullL$ and $\CalP^\perp$ as the corresponding orthogonal projection
operator such that
\begin{align*}
    \CalP : L^2(\dsigma) \to \NullL \,,
\qquad
    \CalP^\perp = \I - \CalP \,.
\end{align*}
Define the operator $\CalP_1: \NullL \to \NullL$ as 
\begin{align*}
    \CalP_1 \Vecf = \CalP(\mu \Vecf \,)
\qquad
   \text{for $\Vecf \in \NullL$.}
\end{align*}
It is clear that $\CalP_1$ is a symmetric operator on a finite-dimensional space, and hence all its eigenvalues are real. Denote the eigenspaces of $\CalP_1$ associated with positive, negative, and zero eigenvalues as $H^+, H^-, H^0$ respectively. Then $\NullL$ is decomposed as
\begin{align*}
    \NullL = H^+ \oplus H^- \oplus H^0 \,.
\end{align*}
Using these notations, we prescribe the boundary conditions at $x = \infty$ in a similar way as in \cite{CoronGolseSulem:88} such that
\begin{align*}
  \lim_{x \to \infty} \Vecf \in H^+ \oplus H^0\,.
\end{align*}
The complete form of the kinetic equation considered in this paper reads
\begin{equation}
\begin{aligned} \label{eq:kinetic}
     & \mu \del_x \Vecf + \CalL \Vecf  = 0 \,,
\\
     & \Vecf \big|_{\mu>0} = \Vech(\mu) + \CalK (\Vecf \big|_{\mu<0})\,, && x = 0 \,,
\\
     & \lim_{x \to \infty} \Vecf \in H^+ \oplus H^0 \,, 
\end{aligned}
\end{equation}
More specific assumptions regarding $\CalL$ and $\CalK$ to guarantee the well-posedness of \eqref{eq:kinetic} will be discussed in Section~\ref{sec:2}. 

Half-space equations with general boundary conditions are frequently
encountered in electric propulsion for
satellites~\cite{goebel_fundamentals_2008} and photon transport in
solid state
devices~\cite{hua_analytical_2014,hadjiconstantinou_variance-reduced_2010},
among many other applications. The standard treatment of this type of
equations is the Monte Carlo
method~\cite{hadjiconstantinou_variance-reduced_2010}. There are also
special cases where analytical solutions are
possible~\cite{hua_analytical_2014}.  

In \cite{LiLuSun} we developed a
direct systematic method to solve half-space equations in the case of pure incoming boundary condition (when $\CalK = 0$).  There are also other direct numerical approaches for this case proposed in
\cites{Coron:90, GolseKlar:95}. Compared with our approach, these methods
suffer from severe Gibbs phenomena and lack of error analysis or
systematic strategy to reduce numerical errors. 
We also note that the method for linearized discrete equations in
\cite{Bernhoff:08} can be applied to solve the continuous half-space equation by approximating it using discrete velocity models. 
Unlike \cite{Bernhoff:08} which focuses on the analysis of the discrete
model, our goal here is to approximate the solutions to the continuous
half-space equation using a spectral type method with convergence analysis.

The present work extends our previous method to the case when various
reflections are involved. The main difficulties that we need to
overcome are the degeneracy of $\CalL$, the derivation of a proper
weak formulation involving $\CalK$, and the fact that the boundary
conditions at $x=\infty$ are part of the solution instead of being
prescribed.  To this end, we apply similar procedure proposed in
\cite{LiLuSun}, which combines and extends the ideas of even-odd
decomposition \cite{EggerSchlottbom:12} and a damping adding-removing
procedure \cites{UkaiYangYu:03, Golse:08}. More specifically, we first
modify $\CalL$ by adding damping terms to it (see
Section~\ref{assump:L}). This will remove the degeneracy of $\CalL$
and ensure that the end-state of the damped solution at $x = \infty$
is zero. Both analysis and numerical schemes are then performed on the
weak formulation of the damped equation, which is derived by applying
an even-odd decomposition with mixed regularity
\cite{EggerSchlottbom:12} of the (damped) solution $\Vecf$.  One
important advantage of the even-odd decomposition is that it leads to
a natural way of constructing a family of basis functions that
captures the possible jump discontinuity of the solution at $\mu = 0$,
by the odd extension of the basis functions constructed for positive
$\mu$.  This discretization of the velocity space based on even-odd
decomposition turns out to be equivalent to the double-$P_n$ method
developed in the literature of solving neutron transport equations,
see e.g.,~\cite{thompson_theory_1963}. We also comment that the
appearance of the boundary operator $\CalK$ introduces extra
difficulty into formulating the weak form of the half-space
equation. The difficulty comes from the fact that only the even part
of the solution $\Vecf^+$ has enough regularity to define a trace on
the boundary. Our main idea here is to use the properties assumed for
$\CalK$ in Section 2 to represent the odd part $\Vecf^-$ on the
boundary in terms of $\Vecf^+$.

Our numerical method is spectral in nature: we apply Galerkin approximations to the weak formulation and use  Babu\v{s}ka-Aziz lemma to show that the damped equation is well-posed and the finite-dimensional approximation is quasi-optimal. It will be clear that the damping plays a crucial role here. 
Finally, we make use of the linearity and use proper superposition of
certain special solutions to the damped equation to recover the original undamped solution. 

A by-product of the above procedure is that we obtain a unified proof for the
well-posedness of the half-space equations with general boundary
conditions.  This well-posedness theory is general enough to include
multi-species and multi-dimensional (in velocity) half-space
equations.

The layout of the paper is as follows. In Section 2 we explain all the assumptions for the linear operator $\CalL$ and the boundary operator $\CalK$. In Section 3 we prove the well-posedness of the half-space equation using the damping adding-removing procedure. In Section 4 we show three numerical examples which cover the three cases of multi-species, multi-frequency transport equations and a multi-dimensional (in velocity) linearized BGK equation. 

\section{Main Assumptions for $\CalL$ and $\CalK$}  \label{sec:2}
In this section we collect the conditions on the linear operator $\CalL$ and the boundary operator $\CalK$. 

\smallskip

\noindent{\em Notation.} In this paper we denote
\begin{align*}
   \viint{\Vecf}{\Vecg}_v = \int_{\VV} \Vecf \cdot \Vecg \dsigma \,,
\quad \text{and} \quad 
   \viint{\Vecf}{\Vecg}_{x,v} = \int_{\R}\int_{\VV} \Vecf \cdot \Vecg \dsigma \dx \,,
\end{align*}
where $\dsigma$ is a measure in the velocity space. Throughout this
paper we assume that the measure $\dsigma$ is symmetric with respect
to $\mu$.

\subsection{Main Assumptions for $\CalL$}  \label{assump:L}
In this subsection we state the general assumptions for the collision operator $\CalL$. 
First, define the weight function (attenuation coefficient)
\begin{equation}\label{def-a}
     a(v) = (1 + |v|)^{\kappa_0}
\end{equation}
for some $0 \leq \kappa_0 \leq 1$. 
The first four basic assumptions for the linear operator $\CalL$ are as follows:
\begin{itemize}
\item[({\bf \PL1})] $\VecL: \CalD(\VecL) \to (L^2(\dsigma))^m$ is
  self-adjoint with its domain $\CalD(\CalL)$ given by
  \begin{equation*}
    \CalD(\VecL) = \{\Vecf \in (L^2(\dsigma))^m \big| \, 
    a(v) \Vecf \in (L^2(\dsigma))^m\} \subseteq (L^2(\dsigma))^m \,,
  \end{equation*}
  where $a(v)$ is defined in \eqref{def-a}. Such space arises
  naturally for linear/linearized collision operator since in many
  cases $\CalL$ has the structure as
  \begin{align*}
    \CalL = a(v) {\mathcal{I}} + \CalL_1 \,,
  \end{align*}
  where $\CalL_1$ is a bounded or even compact operator. 
\item[({\bf \PL2})] $\VecL: (L^2(a\dsigma))^m \to
  (L^2(\tfrac{1}{a}\dsigma))^m$ is bounded, that is, there exists a
  constant $C_0 > 0$ such that
  \begin{equation*}
    \norm{\VecL \Vecf }_{(L^2(\tfrac{1}{a}\dsigma))^m} 
    \leq  C_0 \norm{\Vecf}_{(L^2(a\dsigma))^m} \,.
  \end{equation*}
\item[({\bf \PL3})] $\NullL$ is finite dimensional and
  $\NullL\subseteq (L^p(\dsigma))^m$ for all $p \in [1, \infty)$.

\item[({\bf \PL4})] $\CalL$ is nonnegative: for any $\Vecf \in (L^2(a\dsigma))^m$, 
\begin{align} \label{cond:positive-L}
     \int_{\VV} \Vecf \cdot \CalL \Vecf \dsigma
\geq 0 \,.
\end{align}
\end{itemize}

Assumptions (\PL1)-(\PL4) are general enough to include many classical
models such as the linearized Boltzmann operators (around Maxwellians)
with hard-potentials, the linearized BGK operator, and linear
transport operators for single- or multi-species.
In fact, these classical operators satisfy an even stronger coerciveness property:
\begin{align} \label{coerv-1} 
  \int_{\VV} \Vecf \cdot \CalL \Vecf
  \dsigma \geq c_0 \norm{\CalP^\perp \Vecf}_{(L^2(a\dsigma))^m}^2 \,,
\end{align}
where recall that $\CalP^\perp = \I - \CalP$ and $\CalP: (L^2(\dsigma))^m \to \NullL$ is the projection onto $\NullL$.

We need one last essential assumptions on the coercivity of a damped version of  $\CalL$ on the whole $(L^2(\dsigma))^m$ but not just $(\NullL)^{\perp}$. 
To properly explain this assumption, we introduce several definitions
related to the null space of $\CalL$. Recall that $\CalP_1: \NullL \to
\NullL$ is the operator given by
\begin{equation}
       \CalP_1 (f) = \CalP(\mu f) \qquad \text{for any $f \in \NullL$} \,. \nn
\end{equation}
Note that $\CalP_1$ is a symmetric operator on the finite dimension
space $\NullL$. Therefore, its eigenfunctions form a complete 
basis of $\NullL$. Denote $H^+, H^-, H^0$ as the eigenspaces of $\CalP_1$
corresponding to positive, negative, and zero eigenvalues respectively
and denote their dimensions as
\begin{equation*}
  \nu_+ = \dim H^+, \qquad  \nu_- = \dim H^-, \qquad \nu_0 = \dim H^0 \,.  
\end{equation*}
Let $X_{+,i}, X_{-,j}, X_{0,k}$ be the associated orthornormal
eigenfunctions with $1 \leq i \leq \nu_+$, $1 \leq j \leq \nu_-$, and
$1 \leq k \leq \nu_0$. Note that if any of $\nu_\pm, \nu_0$ is equal
to zero, we simply do not have any eigenfunction associated with the
corresponding eigenspace.  By definition, these eigenfunctions satisfy
\begin{equation*} 
\begin{aligned}
 \La X_{\tau, \gamma}, X_{\tau', \gamma'} \Ra_v = \delta_{\tau\tau'} \delta_{\gamma\gamma'} \,, \qquad
 \La \mu X_{\tau, \gamma}, \,\, X_{\tau', \gamma'} \Ra_v =  0 \,\, \text{if $\tau \neq \tau'$ or $\gamma \neq \gamma'$} \,, \\
 \La \mu X_{0, j}, \,\, X_{0, k} \Ra_v = 0 \,, \quad 
 \La \mu X_{+, j}, \,\, X_{+, i} \Ra_v > 0 \,, \qquad
 \La \mu X_{-, j},  \,\, X_{-, j} \Ra_v < 0 \,, 
\end{aligned}
\end{equation*}
where $\tau \in \{+, -, 0\}$, $\gamma \in \{i,j,k\}$, $1 \leq i \leq \nu_+$, $1 \leq j \leq \nu_-$, and $1 \leq k \leq \nu_0$. 

Our method relies on full coercivity of the collision/scattering operator on $(L^2(a\dsigma))^m$ instead of the partial one in \eqref{coerv-1} on $(\NullL)^\perp$. Hence, instead of working directly with $\CalL$, we add in the damping terms on the modes in $\NullL$ and define the damped linear operator $\CalL_d$ as
\begin{equation}
\begin{aligned} \label{def:L-d-1}
  \CalL_d \Vecf 
 =& \CalL \Vecf + \alpha  \sum_{k=1}^{\nu_+} 
         \mu \VecX_{+,k} \viint{\mu \VecX_{+,k}}{\Vecf}_v
     + \alpha \sum_{k=1}^{\nu_-} 
           \mu \VecX_{-,k} \viint{\mu \VecX_{-,k}}{\Vecf}_v
\\
  &  + \alpha \sum_{k=1}^{\nu_0} 
           \mu \VecX_{0,k} \viint{\mu \VecX_{0,k}}{\Vecf}_v                          
       + \alpha \sum_{k=1}^{\nu_0}
              \mu \CalL^{-1}(\mu \VecX_{0,k})
             \viint{\mu \CalL^{-1}(\mu \VecX_{0,k})}{\Vecf}_v \,,
\end{aligned}
\end{equation}
where $\alpha > 0$ is some constant damping coefficient to be
determined later. The motivation of defining $\CalL_d$ in such a form
is as follows: the operator $\CalL$ normally will provide bounds for
the orthogonal component of $\Vecf$ in $(\NullL)^\perp$. With the
added damping terms to dissipate the modes in $\NullL$, we expect that
$\CalL_d$ will satisfy certain full coercivity condition on
$(L^2(\dsigma))^m$. On the other hand, this added damping effect can
be eventually removed using linearity of the equations.  The precise
assumption of $\CalL$ regarding its coercivity states
\begin{itemize}
\item[({\bf \PL5})] There exist two constants $\alpha, \sigma_0> 0$ such that the damped operator $\CalL_d$ satisfies
\begin{align} \label{cond:coercivity-L}
     \int_{\VV} \Vecf \cdot \CalL_d \Vecf \dsigma
\geq \sigma_0 \norm{\Vecf}_{(L^2(a\dsigma))^m}^2
\end{align}
for any $\Vecf \in (L^2(a\dsigma))^m$. 
\end{itemize}
It will be shown in Lemma~\ref{lem:assump-L-1} that the coercivity
condition in \eqref{coerv-1} combined with the form of $\CalL_d$ in
\eqref{def:L-d-1} implies (\PL5), and hence (\PL5) is a natural
assumption for many examples.

\subsection{Main Assumptions for $\CalK$} \label{sec:K} In this part
we specify conditions for the boundary operator $\CalK$.  These
conditions are stated in rather general forms and are satisfied by a
large class of boundary operators.  Recall that we have denoted
$v = (\mu, \bar v) = (\mu, v_2, \cdots, v_d)$. Denote the incoming and
outgoing parts of the velocity space as
\begin{align*}
   \VV_+ = \{v = (\mu, \bar v) | \, \mu > 0\} \,,
   \quad \text{and} \quad
   \VV_- = \{v = (\mu, \bar v) | \, \mu < 0\} \,.
\end{align*}

We consider the general case where the boundary operator $\CalK$ consists of various types of reflections in the sense that there exists a coefficient
  $\alpha_r \in [0, 1)$ and a scattering kernel $k_r$ (which is a positive measure) such that 
  \begin{align}
    \CalK = \alpha_r \CalK_r \,,
 \qquad
   [\CalK_r \Vecf \,](v) 
    = \int_{\mu'<0} k_r(v, v') \Vecf(v') \dsigma(v')
    \qquad \text{for } v  \in \VV_+ \,. \label{def:K-form} 
  \end{align}

The main assumption for such $\CalK$ is
\begin{itemize}
\item[(\PK)] The reflection operator $\CalK_r$ satisfies that
  \begin{align} \label{bound:K-i}
    \int_{\mu > 0} \mu \abs{\CalK_r \Vecf \,}^2 \dsigma
    \leq 
    \int_{\mu < 0} \abs{\mu} \abs{\Vecf \,}^2 \dsigma \,.
  \end{align}
\end{itemize}

There is a large family of reflection boundary operators $\CalK_r$ that satisfy (\PK). In the literature, the reflection boundary operator for nonlinear kinetic equations of a single species is usually written as
\begin{align*}
   [\mathscr{K}_r F] (v) 
   = \frac{1}{\mu} \int_{\mu'<0} \mu' R(v, v') F(v') \dv' \,,
\qquad
   v \in \VV^+ 
\end{align*}
for some scattering kernel $R$. If we consider the linearization around the equilibrium state such that 
\begin{align*}
      F = M + Mf \,,
\end{align*}
then the linearized version has the form
\begin{align}
   [\CalK_r f] (v) 
  = \frac{M^{-1}(v)}{\mu} \int_{\mu'<0} \mu' R(v, v') f(v') M(v') \dv'
\label{def:K-d-example}
\end{align}
We show in the following lemma that as long as $\CalK_r$ satisfies the classical normalization and reciprocity conditions, then the main assumption (\PK) holds:
\begin{lem} \label{lem:suff-cond-PK} Suppose $M$ is a scalar
  equilibrium state and $\dsigma = M \dv$ where $\dv$ is the Lebesgue
  measure. Suppose $\CalK_r$ has the form as
  in~\eqref{def:K-d-example}.
If $R$ satisfies the normalization and reciprocity conditions:
\begin{align} 
   |\mu'| M(v') R(v, v') &= |\mu| M(v) R(-v', -v) \,, 
\quad
   v \in \VV^+ \,, v' \in \VV^- \,, \label{cond:normal}
\\
    \int_{\mu > 0} R(v, v') & \dv = 1 \,, \label{cond:recip}
\qquad
   v' \in \VV^- \,,
\end{align}
then (\PK) holds. 
\end{lem}
\begin{proof}
Note that an immediate variation of~\eqref{cond:recip} is 
\begin{align}
    \int_{\mu'<0} R(-v', -v) & \dv' = 1 \,,
\qquad
    v \in \VV^+ \,. \label{cond:recip-1}
\end{align}
By the definition of $\CalK_r$, we have
\begin{align*}
   \int_{\mu>0} \mu \abs{\CalK_r f}^2 M(v) \dv
& = \int_{\mu>0} \frac{M^{-1}(v)}{\mu}
      \vpran{\int_{\mu'<0} \mu' R(v, v') f(v') M(v') \dv'}^2 \dv
\\
& \hspace{-3cm}\leq 
    \int_{\mu>0}
       \vpran{\int_{\mu'<0} |\mu'| R(v, v') f^2(v') M(v') \dv'}
       \vpran{\int_{\mu'<0} \frac{|\mu'|}{\mu} R(v, v') M(v') M^{-1}(v) \dv'} \dv
\\
& \hspace{-3cm}
   \stackrel{\eqref{cond:normal}}{=}
    \int_{\mu>0}
       \vpran{\int_{\mu'<0} |\mu'| R(v, v') f^2(v') M(v') \dv'}
       \vpran{\int_{\mu'<0} R(-v', -v) \dv'} \dv
\\
& \hspace{-3cm}
   \stackrel{\eqref{cond:recip-1}}{=}
    \int_{\mu>0}
       \int_{\mu'<0} |\mu'| R(v, v') f^2(v') M(v') \dv'\dv
\\
& \hspace{-3cm}
   = \int_{\mu'<0} \vpran{|\mu'|  f^2(v') M(v')}
       \vpran{\int_{\mu'<0}R(v, v') \dv} \dv'
    \stackrel{\eqref{cond:recip}}{=}
    \int_{\mu'<0} |\mu'|  f^2(v') M(v') \dv' \,.
\end{align*}
The condition (\PK) follows as $\dsigma = M \dv$ in this case. 
\end{proof}

Examples that satisfy \eqref{cond:normal} and \eqref{cond:recip} include 
\begin{itemize}
\item the specular reflection condition where $R(v, v') = \delta(\mu + \mu') \delta(\bar v - \bar v')$;

\item the bounce-back condition where $R(v, v') = \delta(v+v')$;

\item the pure diffuse condition for BGK or linearized Boltzmann equation where $R(v, v') = \frac{\mu}{(2\pi)^{\frac{d-1}{2}}} e^{-\frac{|v|^2}{2}}$;

\item convex combinations of the above three; and more generally, 

\item the (linearized) Cercignani-Lampis collision operator with $R$
  given by 
  \begin{align*}
    R(v, v') = \frac{1}{2\pi \alpha_n \alpha_t (2-\alpha_t)}
    \exp\vpran{-\frac{\mu^2 + (1 - \alpha_n) (\mu')^2}{2\alpha_n}}
    \exp\vpran{-\frac{|\bar v - (1-\alpha_t) \bar v'
        |^2}{2\alpha_t(2-\alpha_t)}} J_0\vpran{\frac{\sqrt{1-\alpha_n}
        \mu \mu'}{\alpha_n}} \,,
  \end{align*}
  where $0 < \alpha_n < 1$, $0 < \alpha_t < 2$, and
  \begin{align*}
    J_0(x) = \frac{1}{2\pi}\int_0^{2\pi} e^{x \cos\phi} {\, \rm d}\phi  \,.
  \end{align*}
\end{itemize}
Hence our method applies to all of these classical cases for single
species.

\begin{rmk} \label{rmk:Maxwell-BC}
In all of our numerical examples in Section~\ref{section:numerics}, we use either the Dirichlet boundary condition with given incoming data or the classical Maxwell boundary condition where
\begin{align*}
    \CalK = \alpha_d \CalK_d + \alpha_s \CalK_s \,
\end{align*}
with the accommodation coefficients $\alpha_d, \alpha_s$ satisfying 
\begin{align*}
   \alpha_d, \alpha_s \geq 0  \,,
\qquad
   0 \leq \alpha_d + \alpha_s < 1 \,.
\end{align*}
The two operators $\CalK_d, \CalK_s$ are the diffuse and specular reflection operators respectively. In terms of the notation in~\eqref{def:K-form}, we can choose in this case
\begin{align*}
    \alpha_r = \alpha_d + \alpha_s \,,
\qquad
   \CalK_r 
   = \frac{\alpha_d}{\alpha_r} \CalK_d  
      + \frac{\alpha_s}{\alpha_r} \CalK_s   \,.
\end{align*}
Since $\CalK_s$ automatically satisfies (\PK) with an equal sign, we only need to check in each numerical example that $\CalK_d$ satisfies (\PK) as well. 
\end{rmk}

Below we state two essential consequences of assumption ({\PK}), which will guarantee the well-posedness of the half-space equation and provide the foundation for the numerical scheme.
\begin{lem} \label{lem:uniqueness} Suppose $\CalK$ satisfies  ({\PK}). Then the half-space equation~\eqref{eq:kinetic} has at
  most one solution $\Vecf \in C([0, \infty); (L^2(|\mu|
  \dsigma))^m)$. 
\end{lem}
\begin{proof}
By the linearity of the equation we only need to prove that if $h = 0$ in the boundary condition of~\eqref{eq:kinetic}, then the only solution to~\eqref{eq:kinetic} is zero. By the non-negativity of $\CalL$, we have $\int_\VV \mu \big|\Vecf \, \big|^2(x, \cdot) \dsigma$ is decreasing in~$x$. Since $\Vecf_\infty \in H^+ \oplus H^0$, we have $\int_\VV \mu \, \big|\Vecf_\infty \big|^2 \dsigma \geq 0$. Hence 
$\int_\VV \mu \, \big|\Vecf\big|^2(x, \cdot) \dsigma \geq 0$ for all $x \geq 0$. In particular this shows 
\begin{equation}\label{eq:nonneg}
  \int_\VV \mu \big|\Vecf\big|^2(0, \cdot) \dsigma \geq 0.
\end{equation} 

By assumption (\PK),
at the boundary $x=0$ we have
\begin{align*}
   \int_{\mu > 0} \mu \abs{\Vecf \, }^2 \dsigma 
&   = \int_{\mu > 0} \mu \abs{\alpha_r \CalK_r(\Vecf\vert_{\mu<0})\,}^2 \dsigma 
\leq
  \alpha_r^2
  \int_{\mu < 0} \abs{\mu} \abs{\Vecf \,}^2 \dsigma \,.
\end{align*} 
Therefore, 
\begin{align*}
  \int_\VV \mu \abs{\Vecf \,}^2(0, \cdot) \dsigma
  \leq 
  - \vpran{1 - \alpha_r^2}
  \int_{\mu < 0} \abs{\mu} \abs{\Vecf \,}^2(0, \cdot) \dsigma
  \leq 0 \,.
\end{align*}
By \eqref{eq:nonneg} and that
$1 - \alpha_r^2 > 0$, we deduce that
\begin{align*}
     \int_{\mu < 0} \mu \abs{\Vecf \,}^2(0, \cdot) \dsigma = 0 \,,
\quad \text{and hence} \quad
     \int_{\mu > 0} \mu \abs{\Vecf \,}^2(0, \cdot) \dsigma = 0 \,.
\end{align*}
Therefore, at $x=0$ we have $f(0, \cdot) \equiv 0$. By the uniqueness of
solutions to~\eqref{eq:kinetic} with only the incoming data
\cite{CoronGolseSulem:88} (that is, $\alpha_r = 0$), we
have that~\eqref{eq:kinetic} has at most one solution.
\end{proof}

\begin{rmk}
The assumption that $\alpha_r < 1$ in~\eqref{def:K-form} is necessary for the uniqueness of the solution. For example, if $\alpha_r = 1$ or $\alpha_d + \alpha_s = 1$ in the boundary operator $\CalK$ for the linear transport equation in \eqref{def:K} considered in our numerical examples, then any multiple of $X_0$ is a solution to the half-space equation with zero incoming data. 
\end{rmk}

The second consequence of assumption (\PK) is 
\begin{lem} \label{lem:bar-K-invertible}
Suppose the measure $\dsigma$ in the velocity space is symmetric with respect to $\mu$. Define the operator $\bar\CalK: (L^2(\mu \One_{\mu > 0} \dsigma))^m \to (L^2(\mu \One_{\mu > 0} \dsigma))^m$ such that
\begin{align} \label{def:bar-K}
    \bar\CalK = \alpha_r \bar\CalK_r \,,
\end{align}
where 
$\bar\CalK_r$ is defined as 
\begin{align*}
    (\bar\CalK_r \Vecf)(\mu, \bar v) = \int_{\mu' > 0} k_r\bigl((\mu, \bar v'), (-\mu', \bar{v}')\bigr) \Vecf(\mu', \bar v) \dsigma(\mu', \bar v') \,,
\qquad \mu > 0 \,,
\end{align*}
where $k_r$ is the reflection kernel of $\CalK_r$. Note that we have
reflected the $\mu'$ component of $v'$ in the
kernel. 
Then
\begin{itemize}

\item[(a)] $\I + \bar\CalK$ is invertible on $(L^2(\mu \One_{\mu > 0} \dsigma))^m$.  \smallskip

\item[(b)] There exists a constant $\beta_1 > 0$ such that the operator $(\I + \bar\CalK)^{-1}(\I - \bar\CalK)$ satisfies that 
\begin{align} \label{posit-bdry}
   \viint{\mu \Vecf}{(\I + \bar\CalK)^{-1} (\I - \bar\CalK) \Vecf}_{\mu>0}
\geq \beta_1 \viint{\mu \Vecf}{\Vecf}_{\mu>0} 
\end{align}
for any $\Vecf \in (L^2(\mu \One_{\mu>0} \dsigma))^m$.

\end{itemize}
\end{lem}
\begin{proof} 

\noindent (a) Denote $\Vecg_1(\mu', \bar v') = \Vecf(-\mu', \bar v')$. 
Then $\bar\CalK_r \Vecf = \CalK_r \Vecg_1$ by the symmetry of $\dsigma$ with respect to $\mu$. Hence, 
\begin{align*}
    \int_{\mu > 0} \mu \abs{\bar\CalK_r \Vecf \,}^2 \dsigma
= \int_{\mu > 0} \mu \abs{\CalK_r \Vecg_1 \,}^2 \dsigma
\leq 
   \int_{\mu < 0} \abs{\mu} \abs{\Vecg_1 \,}^2 \dsigma
= \int_{\mu > 0} \mu \abs{\Vecf \,}^2 \dsigma \,.
\end{align*}
Therefore,
\begin{align}\label{bound:I-bar-K-inv}
    \norm{\bar\CalK}_{\mc{L}((L^2(\mu \One_{\mu > 0} \dsigma))^m)} 
\leq \alpha_r
< 1 \,.
\end{align}
This shows $\I + \bar\CalK$ is invertible on $(L^2(\mu \One_{\mu > 0} \dsigma))^m$. Furthermore, we have the bound 
\begin{equation*}
  \norm{(\I + \bar\CalK)^{-1}}_{\mc{L}((L^2(\mu \One_{\mu > 0} \dsigma))^m)} \leq \bigl(1 - \norm{\bar\CalK}_{\mc{L}((L^2(\mu \One_{\mu > 0} \dsigma))^m)} \bigr)^{-1} 
  \stackrel{\eqref{bound:I-bar-K-inv}}{\leq} (1 - \alpha_r)^{-1}.
\end{equation*}

\smallskip

\noindent (c) Denote $\Vecg_2 = (\I + \bar\CalK)^{-1} \Vecf \in (L^2(\mu \One_{\mu>0} \dsigma))^m$.  Then
\begin{equation*}
  \begin{aligned}
  \viint{\mu \Vecf}{(\I + \bar\CalK)^{-1} (\I - \bar\CalK) \Vecf}_{\mu>0}
& = \viint{\mu (\I + \bar\CalK)\Vecg_2}{(\I - \bar\CalK) \Vecg_2}_{\mu>0}
= \viint{\mu \Vecg_2}{\Vecg_2}_{\mu > 0}
   -  \viint{\mu \bar\CalK \Vecg_2}{\bar\CalK \Vecg_2}_{\mu>0} \\
   & = \norm{\Vecg_2}^2_{(L^2(\mu \One_{\mu > 0} \dsigma))^m} - \norm{\bar\CalK \Vecg_2}^2_{(L^2(\mu \One_{\mu > 0} \dsigma))^m} \\ &\stackrel{\eqref{bound:I-bar-K-inv}}{\geq} \bigl( 1 - \alpha_r^2 \bigr) \norm{\Vecg_2}^2_{(L^2(\mu \One_{\mu > 0} \dsigma))^m}. 
\end{aligned}
\end{equation*}
Observe that 
\begin{equation*}
  \norm{\Vecf}_{(L^2(\mu \One_{\mu > 0} \dsigma))^m} \leq \norm{(\I + \bar\CalK)}_{\mc{L}((L^2(\mu \One_{\mu > 0} \dsigma))^m)} \norm{\Vecg_2}_{(L^2(\mu \One_{\mu > 0} \dsigma))^m} \stackrel{\eqref{bound:I-bar-K-inv}}{\leq}
(1 + \alpha_r) \norm{\Vecg_2}_{(L^2(\mu \One_{\mu > 0} \dsigma))^m}. 
\end{equation*}
We conclude by combining the previous two estimates such that 
\begin{equation*}
  \viint{\mu \Vecf}{(\I + \bar\CalK)^{-1} (\I - \bar\CalK) \Vecf}_{\mu>0}
  \geq \bigl( 1 - \alpha_r^2 \bigr) (1 + \alpha_r)^{-2} 
  \norm{\Vecf}_{(L^2(\mu \One_{\mu > 0} \dsigma))^m}.  
\end{equation*}
Hence \eqref{posit-bdry} holds with
$\beta_1 = \bigl( 1 - \alpha_r^2 \bigr) (1 + \alpha_r)^{-2}$.
\end{proof}

\section{Well-posedness} \label{sec:well-posedness}
In this section we establish the well-posedness of equation~\eqref{eq:kinetic} based on the assumptions for $\CalL$ and $\CalK$ in the previous section. The framework is similar to \cite{LiLuSun}: first we add damping terms to $\CalL$  and show that the damped equation has a unique solution. This will be achieved by using the Babu\v{s}ka-Aziz lemma. Then we show how to recover the solution to the original kinetic equation using suitable superpositions with special solutions.  

The damped kinetic equation has the form
\begin{equation}
\begin{aligned} \label{eq:damped}
     \mu &\del_x \Vecf + \CalL_d \Vecf  = 0 \,,
\\
     \Vecf \big|_{\mu>0} & =\, \Vech + \CalK (\Vecf \big|_{\mu<0}) \,, \qquad &&\mu > 0 \,,
\\
     &\Vecf \to 0 \,, &&\text{as $x \to \infty$},
\end{aligned}
\end{equation}
where the damped operator $\CalL_d$ is defined in~\eqref{def:L-d-1}.

\subsection{Weak Formulation} \label{sec:weak-form} In order to show the well-posedness of~\eqref{eq:damped}, we consider the weak formulation of the equation using the even-odd decomposition. Recall that we have denoted $\Vd  = (v_2, \cdots, v_d)$. For any scalar function $g(\mu, \bar v)$, let $g^+, g^-$ be its even and odd parts (with respect to $\mu$) respectively such that
\begin{align*}
   g^+(\mu, \Vd) = \frac{g(\mu, \Vd) + g(-\mu, \Vd)}{2} \,,
\qquad
   g^-(\mu, \Vd) = \frac{g(\mu, \Vd) - g(-\mu, \Vd)}{2} \,,
\end{align*}
Therefore we have
\begin{align*}
    g^+(\mu, \Vd) + g^-(\mu, \Vd) = g(\mu, \Vd) \,, 
\qquad 
     g^+(\mu, \Vd) - g^-(\mu, \Vd) = g(-\mu, \Vd) \,.
\end{align*}
For the vector-valued function $\Vecf$, denote
\begin{align*}
   \Vecf^+ = (f_1^+, \, f_2^+, \, \cdots, \, f_d^+)^T \,,
\qquad
   \Vecf^- = (f_1^-, \, f_2^-, \, \cdots, \, f_d^-)^T \,.
\end{align*}
The solution space for~\eqref{eq:kinetic} and~\eqref{eq:damped} is 
\begin{align*}
    \Gamma =  \{ \Vecf \in (L^2(\dsigma\dx))^m \; \big| \; \mu \del_x \Vecf^+ \in (L^2(\dsigma_1\dx))^m \} \,. 
\end{align*}
for some $\dsigma_1$ such that the term $\int_{\VV} \Vecf \cdot \mu \del_x \Vecf^+$ is well-defined. The norm in $\Gamma$ is defined as 
\begin{align} \label{def:norm-Gamma}
    \norm{\Vecf}_\Gamma
    = \norm{\Vecf}_{(L^2(\dsigma\dx))^m}
       + \norm{\mu \del_x \Vecf^+}_{(L^2(\dsigma_1\dx))^m}\,.
\end{align}
One example of $\dsigma, \dsigma_1$ is that $\dsigma = a(v) \dv$ and $\dsigma_1 = \frac{1}{a} \dv$ where $a(v)$ is the attenuation coefficient defined in \eqref{def-a}. For the operator defined in \eqref{def:L-multi-freq}, we have $\dsigma_1 = \dsigma$.  

This type of solution space $\Gamma$ with mixed regularity is introduced in \cite{EggerSchlottbom:12}. For a general function $\Vecf \in \Gamma$, the trace of $\mu \Vecf^+$ at $x=0$ is well-defined while the trace of $\mu \Vecf^-$ may not.
Due to this lack of regularity for $\Vecf^-$, when deriving the weak formulation we will represent $\Vecf^-$ in terms of $\Vecf^+$ on the boundary.  Recall that the boundary condition is given by 
\begin{align*}
     \Vecf \big|_{\mu > 0} 
    = \Vech + \CalK \vpran{\Vecf \big|_{\mu < 0}}  \,.
\end{align*}
Using the even-odd decomposition, we have 
\begin{align*}
   \vpran{\Vecf^+ + \Vecf^-} \big|_{\mu > 0}
& = \Vech
    + \CalK(\Vecf^+ \big|_{\mu<0})
    + \CalK(\Vecf^- \big|_{\mu<0})
\\
& = \Vech
    + \alpha_r \int_{\mu'<0} k_r((\mu, \bar v), (\mu', \bar v')) \Vecf^+(\mu') \dsigma
    + \alpha_r \int_{\mu'<0} k_r((\mu, \bar v), (\mu', \bar v')) \Vecf^-(\mu') \dsigma
\\
& = \Vech
    + \alpha_r \int_{\mu'>0} k_r((\mu, \bar v), (-\mu', \bar v')) \Vecf^+(\mu') \dsigma
    - \alpha_r \int_{\mu'>0} k_r((\mu, \bar v), (-\mu', \bar v')) \Vecf^-(\mu') \dsigma
\\
& = \Vech
      + \bar\CalK \Vecf^+
      - \bar\CalK \Vecf^- \,, 
\end{align*}
where $\bar\CalK$ is defined in~\eqref{def:bar-K}. Note that in order to get the third line  we have used that $\dsigma$ is symmetric with respect to $\mu$.
Hence, the boundary condition has been reformulated as
\begin{align*}
    (\I + \bar\CalK) \Vecf^- = \Vech - (\I - \bar\CalK) \Vecf^+ \,,
\qquad  \mu > 0 \,.
\end{align*}
For any operator $\CalK$ that satisfies assumption (\PK), we have shown in Lemma~\ref{lem:bar-K-invertible} that $\I + \bar\CalK$ is invertible on $(L^2(\mu \One_{\mu > 0} \dsigma))^m$. Thus, $\Vecf^-$ is related to $\Vecf^+$ as
\begin{align} \label{eq:f-neg-bdry}
   \Vecf^- \big|_{\mu>0} 
   = (\I + \bar\CalK)^{-1} \Vech
      - (\I + \bar\CalK)^{-1} (\I - \bar\CalK) \Vecf^+ \big|_{\mu>0} \,.
\end{align}
Hence when deriving the weak formulation of the half-space, the boundary term at $x=0$ becomes
\begin{align*}
   \viint{\mu \Vecf}{\Vecphi}_v
= 2 \viint{\mu \Vecphi^+}{\Vecf^-}_{\mu > 0}
= 2\viint{\mu \Vecphi^+}{(\I + \bar\CalK)^{-1} \Vech}_{\mu > 0}
   - 2\viint{\mu \Vecphi^+}{(\I + \bar\CalK)^{-1} (\I - \bar\CalK) \Vecf^+}_{\mu > 0} \,.
\end{align*}

Define the bilinear form 
\begin{equation}
\begin{gathered} \label{def:B}
    \CalB(\Vecf, \Vecphi \,) 
  = - \viint{\Vecf^-}{\mu \del_x \Vecphi^+}_{x,v}
       + \viint{\Vecphi^-}{\mu \del_x \Vecf^+}_{x, v} 
       + \viint{\Vecphi}{\CalL_d \Vecf}_{x, v}
\\
       + 2\viint{\mu \Vecphi^+}{(\I + \bar\CalK)^{-1} (\I - \bar\CalK) \Vecf^+}_{\mu>0} \Big\vert_{x=0} \,.
\end{gathered}
\end{equation}
and let $l$ be the linear functional on $(L^2(\mu \One_{\mu > 0} \dsigma))^m$ such that
\begin{align} \label{def:l}
    l(\Vecphi) = 2\viint{\mu \Vecphi^+}{(\I + \bar\CalK)^{-1} \Vech}_{\mu>0} \,.
\end{align}
The previous calculations then show that the weak formulation of
equation~\eqref{eq:damped} has the form
\begin{align} \label{eq:weak-form-transport}
     \CalB(\Vecf, \Vecphi) = l(\Vecphi)  \, \qquad \text{for any $\Vecphi \in \Gamma$} \,.
\end{align}

The main tool that we use to show well-posedness and quasi-optimality is the Babu\v{s}ka-Aziz lemma which we recall below:
\begin{thm}[Babu\v{s}ka-Aziz]\label{thm-BA}
Suppose $\Gamma$ is a Hilbert space and $\CalB: \Gamma \times \Gamma \to \R$ is a bilinear operator on $\Gamma$.  Let $l: \Gamma \to \R$ be a bounded linear functional on $\Gamma$. 
\smallskip

\noindent (a) If $\CalB$ satisfies the boundedness and inf-sup conditions on $\Gamma$ such that
\begin{itemize}
\item  there exists a constant $c_0 > 0$ such that $|\CalB (f, g)| \leq c_0 \|f\|_\Gamma \|g\|_\Gamma$ for all $f, g \in \Gamma$;  \smallskip

\item there exists a constant $\delta_0 > 0$ such that 
\begin{equation}
    \label{cond:inf-sup}
\begin{aligned}
     \sup_{\|\Vf\|_{\Gamma}=1} \CalB(\Vf, \psi) \geq \delta_0 \|\psi\|_\Gamma \,, \qquad &\text{for any $\psi \in \Gamma$} \,, 
\\
     \sup_{\|\psi\|_{\Gamma}=1} \CalB(\Vf, \psi) \geq \delta_0 \|\Vf\|_\Gamma \,,  \qquad &\text{for any $f \in \Gamma$}
\end{aligned}       
\end{equation}
for some constant $\delta_0 > 0$,
\end{itemize}
then there exists a unique $f \in \Gamma$ which satisfies 
\begin{equation*}
     \CalB (f, \psi) = l(\psi) \,,
\qquad \text{for any $\psi \in \Gamma$} \,.      
\end{equation*}

\noindent (b) Suppose $\Gamma_N$ is a finite-dimensional subspace of $\Gamma$. If in addition $\CalB: \Gamma_N \times \Gamma_N \to \R$ satisfies the inf-sup condition on $\Gamma_N$, then there exists a unique solution $f_N$ such that
\begin{equation*}
     \CalB (f_N, \psi_N) = l(\psi_N) \,,
\qquad \text{for any $\psi_N \in \Gamma_N$} \,. 
\end{equation*} 
Moreover, $f_N$ gives a quasi-optimal approximation to the solution $f$ in (a), that is, there exists a constant $\kappa_1$ such that
\begin{equation*}
    \|f - f_N\|_\Gamma \leq \kappa_1 \inf_{w \in \Gamma_N} \|f - w\|_{\Gamma} \,. 
\end{equation*}
\end{thm}

Now we verify that $B(\cdot, \cdot)$ and $l(\cdot)$ defined in \eqref{def:B} and \eqref{def:l} satisfy the conditions in Theorem \ref{thm-BA}.
\begin{prop} 
Suppose the measure $\dsigma$ in the velocity space is symmetric with respect to $\mu$. Suppose the linear operators $\CalL$ satisfies the assumptions (\PL1)-(\PL5) and 
the boundary operator $\CalK$ satisfies assumption (\PK). 
Then 
\begin{itemize}
\item[(a)] the bilinear form $\CalB: \Gamma \times \Gamma \to \R$ satisfies the boundedness and inf-sup conditions and the linear functional $l$ is bounded on $\Gamma$.  Therefore, equation~\eqref{eq:weak-form-transport} has a unique solution $\Vecf \in \Gamma$.  \smallskip

\item[(b)] Moreover, $\mu \del_x \Vecf \in (L^2(\frac{1}{a}\dsigma\dx))^m$.
Thus $\Vecf$ is a strong solution to the damped half-space equation~\eqref{eq:damped}.

\end{itemize}
\end{prop}
\begin{proof}
For each $\Vecf \in \Gamma$ the proof is done by finding an appropriate test function $\Vecphi_f \in \Gamma$ such that $\CalB(\Vecphi, \Vecf)$ satisfies the inf-sup condition: 
\begin{align*}
    \CalB(\Vecphi_f, \Vecf) \geq \hat c_0 \norm{\Vecf \,}_\Gamma^2 \,,
\qquad
   \norm{\Vecphi_f}_\Gamma \leq \hat c_1 \norm{\Vecf \,}_\Gamma \,.
\end{align*}
The particular choice of $\Vecphi_f$ is the same as in \cite{LiLuSun} such that
$\Vecphi_f = \delta_1 \Vecphi_1 + \Vecphi_2$ with $\delta_1 > 0$ large enough and
\begin{align*}
   \Vecphi_1 = \Vecf\,,
\qquad
   \Vecphi_2 = \frac{1}{(1 + |v|)^{\kappa_0}} \mu \del_x \Vecf^+  \,.
\end{align*}
Using such $\Vecphi_f$ together with the coercivity of the damped
operator $\CalL_d$ in (\PL5), we have identical estimates for the
interior terms in $\CalB(\Vecphi_f, \Vecf)$ as in the proof of
\cite{LiLuSun}*{Proposition
  3.1}. 
Moreover, the positivity of the boundary term $\viint{\mu \Vecf^+}{(\I
  + \bar\CalK)^{-1} (\I - \bar\CalK) \Vecf^+}_{\mu>0}$ is guaranteed
by Lemma~\ref{lem:bar-K-invertible}.
Hence by the same argument as in \cite{LiLuSun}, we have that $\CalB$ satisfies the inf-sup condition. Boundedness of $\CalB$ and $l$ can be shown by direct applications of the Cauchy-Schwarz inequality. Thus the weak formulation~\eqref{eq:weak-form-transport}  has a unique solution. This also implies that the half-space equation~\eqref{eq:kinetic} has a unique solution in the distributional sense. In addition, the half-space equation itself shows $\mu \del_x \Vecf = -\CalL \Vecf \in (L^2(\frac{1}{a}\dsigma\dx))^m$ where $a$ is the attenuation coefficient defined in \eqref{def-a}. Hence the full trace of $\Vecf$ in $L^2(|\mu|\dsigma)$ is well-defined. 
\end{proof}

As in \cite{LiLuSun} we will solve the damped equation \eqref{eq:damped} by a Galerkin
method.

\begin{prop}[Approximations in $\R^d$]\label{prop:approximation-NK}
Suppose $\{\psi_k^{(1)}(\mu)\psi_{n_2}^{(2)}(v_2) \cdots \psi_{n_d}^{(d)}(v_d) \}_{k, n_2, \cdots n_d=1}^\infty$ is an orthonormal basis of $L^2(\!\dsigma)$ such that
\begin{itemize}
\item $\psi_{2n-1}^{(1)}(\mu)$ is odd and $\psi_{2n}^{(1)}(\mu)$ is even in $\mu$ for any $n \geq 1$;

\item $\mu \psi_{2n}^{(1)}(\mu) \in \Span\{\psi_1^{(1)}, \cdots, \psi_{2n+1}^{(1)}\}$ for each $n \geq 1$.
\end{itemize}
Define the closed subspace $\Gamma_{NK}$ as
\begin{equation*}
   \Gamma_{NK} 
   = \left\{\Vecg(x, v) \in \Gamma \Big| \, \Vecg(x, v) = \sum_{i=1}^m \sum_{k=1}^{2N+1}\sum_{n_2,\cdots,n_d =1}^K  g_{k,n_2,\cdots,n_d}^{(i)}(x) \psi_k^{(1)}
(\mu)\psi_{n_2}^{(2)}(v_2) \cdots \psi_{n_d}^{(d)}(v_d) \, \Be_i \right\}
   \,,
\end{equation*} 
where $\Be_i = (0, \cdots, 0, 1, 0, \cdots, 0)^T$ is the standard $i^{th}$ basis vector with $1 \leq i \leq m$ and $g_{k,n_2 \cdots n_d}^{(i)} \in  H^1(\R^+)$. 
Then 
\begin{itemize}
\item[(a)]
there exists a unique $\Vecf_{NK} \in \Gamma_{NK}$ such that  
\begin{equation}\label{def:f-NK}
     \Vecf_{NK}(x, v) 
     =  \sum_{i=1}^m \sum_{k=1}^{2N+1} \sum_{n_2, \cdots, n_d =1}^K  a_{k,n_2,\cdots,n_d}^{(i)}(x) \psi_k^{(1)}(\mu)\psi_{n_2}^{(2)}(v_2) \cdots \psi_{n_d}^{(d)}(v_d) \, \Be_i \,, 
\end{equation}
which satisfies 
\begin{equation}\label{eq:variational-finitedim-NK}
      \CalB(\Vecf_{NK}, \Vecg) = l(\Vecg) \quad \text{for every $\Vecg \in \Gamma_{NK}$} \,, 
\end{equation}
where $\CalB$ and $l$ are defined in \eqref{def:B} and \eqref{def:l} respectively. The coefficients $\{a_{k,n_2, \cdots,n_d}^{(i)}(x)\}$ satisfy that
\begin{equation*}
     a_{k,n_2,\cdots,n_d}^{(i)}(\cdot) \in C^1[0, \infty) \cap H^1(0, \infty), \qquad 
1 \leq k \leq 2N+1 \,,  \,\, 1 \leq n_2,\cdots,n_d \leq K \,,
\,\, 1 \leq i \leq m \,. 
\end{equation*}

\item[(b)] The approximation is quasi-optimal, that is, there exists a constant $\kappa_2 > 0$ such that 
\begin{equation*}
    \|\Vecf - \Vecf_{NK}\|_\Gamma 
\leq 
    \kappa_2 \inf_{\vec w \in \Gamma_{NK}} \|\Vecf - \vec w\|_{\Gamma} \,. 
\end{equation*}
\end{itemize}
\end{prop}
\begin{proof}
Part (a) and (b) follow directly from the Babu\v{s}ka-Aziz lemma as long as we verify that $\CalB(\cdot, \cdot)$ satisfies the inf-sup condition over $\Gamma_{NK}$. The only modification is in the choice of $\Vecphi_f$ where $\Vecphi_2$ is projected onto $\Gamma_{NK}$. The proof again follows along the same line to the proof of Proposition 3.2 in \cite{LiLuSun} using the positivity of the boundary term guaranteed by Lemma~\ref{lem:bar-K-invertible}.
\end{proof}

The following Proposition reformulates \eqref{eq:variational-finitedim-NK} into an ODE with explicit boundary conditions. 
\begin{prop}
Let 
\begin{equation*}
      \mathsf{A} = \left(\La \mu \psi_k^{(1)}, \,\, \psi_j^{(1)}\Ra_{\mu} \right)_{(2N+1)\times (2N+1)} \,.
\end{equation*}
Define the 2(d+1)-tensors $\mathfrak{A}$ 
and $\mathfrak{B}$ as
\begin{equation}\label{def:D-B}
\begin{aligned}
      \mathfrak{A} &= \mathsf{A} \otimes \ii \otimes \cdots \otimes \ii \otimes \ii = \big(\mathsf{A}_{ik} \delta_{n_2 l_2} \cdots \delta_{n_d l_d} \delta_{pq}\big)_{(2N+1)^2 \times K^2 \times \cdots \times K^2 \times m^2}\,,
\\
      \mathfrak{B}_{kn_2\cdots n_dp}^{il_2\cdots l_d q}
      &= - \La \psi_k^{(1)}(\mu)\psi_{n_2}^{(2)}(v_2) \cdots \psi_{n_d}^{(d)}(v_d) \, \Be_p,  \,\,
         \CalL_d \left(\psi_i^{(1)}(\mu)\psi_{l_2}^{(2)}(v_2) \cdots \psi_{l_d}^{(d)}(v_d) \, \Be_q \right)\Ra_v 
\end{aligned}
\end{equation}
for $1 \leq i,k \leq 2N+1$, $1 \leq n_2, \cdots, n_d \leq K$, $1 \leq l_2, \cdots, l_d \leq K$, and $1 \leq p, q \leq m$.
Then the variational form \eqref{eq:variational-finitedim-NK} is equivalent to the  following ODE for the coefficients $a_{kn_2\cdots n_d}^{(p)}(x)$:
\begin{equation} \label{ODE-general}
        \sum_{p=1}^m \sum_{k=1}^{2N+1} \sum_{n_2, \cdots, n_d=1}^K \mathfrak{A}_{kn_2 \cdots n_d p}^{il_2 \cdots l_dq} \, \del_x a_{kn_2 \cdots n_d}^{(p)}(x) 
        = \sum_{p=1}^m\sum_{k=1}^{2N+1}  \sum_{n_2, \cdots, n_d=1}^K \mathfrak{B}_{kn_2 \cdot n_d p}^{il_2 \cdots l_d q} \, a_{kn_2 \cdots n_d}^{(p)}(x),
\end{equation}
together with the boundary conditions at $x=0$:
\begin{equation}\label{Boundary-general}
\begin{aligned}
    \sum_{k=1}^{N+1} \La \mu \psi_{2k-1}^{(1)}, \,\, \psi_{2i}^{(1)} \Ra_{\mu}  a_{2k-1,n_2, \cdots, n_d}^{(q)}(0)
    + 2 \La \mu \vec\psi_{2i,n_2,\cdots,n_d}^{(q)}, \,\, (\I + \bar\CalK)^{-1}(\I - \bar\CalK)\Vecf_{NK}\Ra_{\mu>0}
\\
  = 2 \int_{\mu>0} \mu \, (\I + \bar\CalK)^{-1} \Vech \cdot \, \vec\psi_{2i,n_2,\cdots,n_d}^{(q)}\dsigma
\end{aligned}
\end{equation}
for $i = 1, \cdots, N$, $n_2, \cdots, n_d = 1, 2, \cdots, K$, and $q = 1, \cdots, d$.
Here $\Vecf_{NK}$ is defined in~\eqref{def:f-NK} at $x = 0$ and the basis function $\vec\psi_{2i,n_2,\cdots,n_d}^{(q)}$ is
\begin{align*}
   \vec\psi_{2i,n_2,\cdots,n_d}^{(q)} = \psi_{2i}^{(1)}(\mu)\psi_{n_1}^{(2)}(v_2) \cdots \psi_{n_d}^{(d)}(v_d) \, \Be_q  \,.
\end{align*}
\end{prop}

\begin{proof}
Equation~\eqref{ODE-general} is obtained by choosing the test function $\Vecg$ in~\eqref{eq:variational-finitedim-NK} as 
\begin{align*}
   \Vecg 
   = g_0(x) \, \psi_k^{(1)}(\mu)\psi_{n_2}^{(2)}(v_2) \cdots \psi_{n_d}^{(d)}(v_d) \, \Be_p
\end{align*}
for each basis function $\psi_k^{(1)}(\mu)\psi_{n_2}^{(2)}(v_2) \cdots \psi_{n_d}^{(d)}(v_d) \, \Be_p$ and for any arbitrary $g_0(\cdot) \in C_c^1(0, \infty)$.
The boundary condition~\eqref{Boundary-general} is derived by choosing the test functions as 
\begin{align*}
   \Vecg 
   = g_1(x) \, \psi_{2k}^{(1)}(\mu)\psi_{n_2}^{(2)}(v_2) \cdots \psi_{n_d}^{(d)}(v_d) \, \Be_p \,,
\end{align*}
for each basis function $\psi_{2k}^{(1)}(\mu)\psi_{n_2}^{(2)}(v_2) \cdots \psi_{n_d}^{(d)}(v_d) \, \Be_p$ in $\Gamma_{NK}$ and for any arbitrary $g_1(\cdot) \in C_c^1[0, \infty)$.
\end{proof}

Since the tensors $\mathfrak{A}, \mathfrak{B}$ are the same as in \cite{LiLuSun}, we have that there are $mNK^{d-1}$ positive, $mNK^{d-1}$ negative, and $mK^{d-1}$ generalized eigenvalues of $(\mathfrak{A}, \mathfrak{B})$. Note that there are $m(2N+1)K^{d-1}$ unknowns in the ODE system~\eqref{ODE-general} and $mNK^{d-1}$ boundary conditions. This is again the correct number of boundary conditions for~\eqref{ODE-general} to have a unique decaying solution.

\subsection{Recovery} In this part we show the procedures to recover the solution to the original kinetic equation~\eqref{eq:kinetic}. To this end, let 
$\Vecf$ be the solution to the damped equation~\eqref{eq:damped}. For all $1 \leq i \leq \nu_0$ and $1 \leq j \leq \gamma_+$, let $\Vecg_{0,i}, \Vecg_{0,j}$ be the solution to~\eqref{eq:damped} with $\Vech = \VecX_{0, i} - \CalK (\VecX_{0, i} \big|_{\mu<0})$ and $\Vech = \VecX_{+, j}- \CalK (\VecX_{+, j} \big|_{\mu<0})$ respectively. More explicitly, for each $1 \leq i \leq \nu_0$,
\begin{align} \label{eq:g-0-i-damped}
     \mu &\del_x \Vecg_{0, i} + \CalL_d \Vecg_{0, i} = 0 \,, \nn
\\
     \Vecg_{0, i} \big|_{\mu>0}
     =&\, \vpran{\VecX_{0, i} - \CalK (\VecX_{0, i} \big|_{\mu<0})}
        + \CalK \vpran{\Vecg_{0, i} \big|_{\mu<0}} \,, 
\qquad &&\mu > 0 \,,
\\
     &\Vecg_{0, i} \to 0 \,, &&\text{as $x \to \infty$}, \nn
\end{align}
and for each $1 \leq j \leq \nu_+$,
\begin{align} \label{eq:g-plus-j-damped}
     \mu &\del_x \Vecg_{+, j} + \CalL_d \Vecg_{+, j} = 0 \,, \nn
\\
     \Vecg_{+, j} \big|_{\mu>0}
     =&\, \vpran{\VecX_{+, j} - \CalK (\VecX_{+, j} \big|_{\mu<0})} 
        + \CalK \vpran{\Vecg_{+, j} \big|_{\mu<0}} \,, 
\qquad &&\mu > 0 \,,
\\
     &\Vecg_{+, j} \to 0 \,, &&\text{as $x \to \infty$}, \nn
\end{align} 
The key idea is that the damping terms in $\CalL_d$ vanish for a proper linear combination of $\Vecf$, $\Vecg_{0, i}$'s, and $\Vecg_{0, j}$'s.
The recovering procedures rely on the uniqueness of solutions to the original kinetic equation~\eqref{eq:kinetic}.

\begin{prop} \label{prop:g}
There exists a unique sequence of constants $c_{0, i}, c_{+, k} \in \R$ for $1 \leq i \leq \gamma_0$ and $1 \leq j \leq \gamma_+$ such that if we define
\begin{align} \label{def:g}
    \Vecg 
    = \sum_{i=1}^{\gamma_0} c_{0, i} \Vecg_{0, i}
       + \sum_{i=1}^{\gamma_+} c_{+, j} \Vecg_{+, j} \,,
\end{align}
then
\begin{align} \label{eq:shift}
     \viint{\mu \VecX_{0,i_0}}{(\Vecf - \Vecg)(x)}_v = 0 \,,
\quad
     \viint{\mu \VecX_{\pm,i_\pm}}{(\Vecf - \Vecg)(x)}_v = 0 \,,
\quad
     \viint{\mu \CalL^{-1}(\mu \VecX_{0,i_0})}{(\Vecf - \Vecg)(x)}_v
 = 0 
\end{align}
for all $x \geq 0$ and all $1 \leq i_0 \leq \gamma_0$, $1 \leq i_\pm \leq \gamma_\pm$.
\end{prop}
\begin{proof}
Since the proof of this proposition only depends on the structure of the kinetic equation instead of the particular form of the boundary condition, the details are the same as in Proposition 3.8 in \cite{LiLuSun}. We explain the main idea here. The key structure we utilize here is that the coefficients in the added damping terms only depends on the average of the damped solution against $\mu \VecX_{\alpha, i}$ or $\mu \CalL^{-1}(\mu \VecX_0)$. Hence to remove the damping effect, we only need to choose $\Vecg$ carefully such that $\Vecf - \Vecg$ will have zero averages. Bearing this in mind, we denote
\begin{align*}
  \VU_+ & = \left(\La \mu \VecX_{+,1}, f \Ra, \,\, \cdots \,, \,\,
    \La \mu \VecX_{+,\nu_+}, f \Ra \right)^{\TT} \,, 
\qquad
  \VU_- = \left(\La \mu \VecX_{-,1}, f \Ra, \,\, \cdots \,, \,\,
    \La \mu \VecX_{-,\nu_-}, f \Ra \right)^{\TT} \,, \\
  \VU_0 & = \left(\La \mu \VecX_{0,1}, f \Ra, \,\, \cdots \,, \,\, \La
    \mu \VecX_{0,\nu_0}, f \Ra \right)^{\TT} \,, \\
  \VU_{\VecL, 0} & = \left(\La \mu \VecL^{-1}( \mu \VecX_{0,1}), f
    \Ra_v, \,\, \cdots \,, \,\, \La (v_1 + u) \VecL^{-1}( (v_1 + u
    )X_{0,1})X_{0,\nu_0}, f \Ra_v\right)^{\TT} \,, 
\end{align*}
and
\begin{equation}
  \VU_f 
  = \left(\VU_+^{\TT}, \,\, \VU_-^{\TT}, \,\, \VU_0^{\TT}, \,\, \VU_{\VecL, 0}^{\TT}
  \right)^{\TT} \,. \label{def-U}
\end{equation}
By multiplying $X_{+,j}, X_{-,i}, X_{0, k}, \CalL^{-1}(v_1 \chi_{0,m})$ to \eqref{eq:damped} and integrating over $v \in \R^d$, we have
\begin{equation}
     \del_x \VU + \BA \VU  = 0 \,, \label{ode-U-2}
\end{equation}     
where the coefficient matrix $\BA$ is 
\begin{equation}
    \BA
    =\left(
       \begin{array}{c|c|c}
        \begin{matrix}
              \alpha D_{+} \\       
          && -\alpha D_{-} \\       \end{matrix}        
       & \mbox{0}
        & \begin{matrix}\alpha A_{21} \\  \alpha A_{22}\end{matrix}
\\ \hline
        0 & 0 
     & \alpha B
\\ \hline
       \begin{matrix}\alpha  A_{21}^{\TT} && \alpha A_{22}^{\TT} \end{matrix}& I + \alpha B & \alpha D
\end{array}\right) \,,
\end{equation}
where $D_\pm$ are positive diagonal matrices and
\begin{equation*}
\begin{aligned}
     A_{21, ik} 
     &= \left(\La \mu \VecX_{+,i}, \,\, \VecL^{-1}(\mu \VecX_{0,k})\Ra\right)_{\gamma_+ \times \gamma_0} \,,
\qquad
     A_{22, jk} 
     = \left(\La \mu \VecX_{-,j}, \,\, \VecL^{-1}(\mu  \VecX_{0,k})\Ra\right)_{\gamma_- \times \gamma_0} \,,
\\
     B_{ij}
     & = \La \mu  \VecX_{0,i}, \,\, \VecL^{-1}(\mu \VecX_{0,j})\Ra_{\gamma_0 \times \gamma_0}  \,,
\qquad
    D_{ij}
    =  \La \mu \VecL^{-1}(\mu \VecX_{0,i}), \,\,
                  \VecL^{-1}(\mu \VecX_{0,j})\Ra_{\gamma_0 \times \gamma_0}  \,, 
\end{aligned}
\end{equation*}
where $B$ is symmetric positive definite and $D$ is symmetric.  Thus $\BA$ is a matrix of size $(\gamma_+ + \gamma_- + 2\gamma_0) \times (\gamma_+ + \gamma_- + 2\gamma_0)$. The proof of Proposition 3.8 in \cite{LiLuSun} shows that $\BA$ has $\gamma_- + \gamma_0$ negative eigenvalues $\{{\bf v_i}\}_{i=1}^{\gamma_- + \gamma_0}$. Moreover, $\BA$ is of rank $\gamma_+ + \gamma_0$ since the original kinetic equation~\eqref{eq:kinetic} satisfies the uniqueness property in Lemma~\ref{lem:uniqueness}. Note that by the boundedness, all the solutions to the damped equation~\eqref{eq:damped} will be orthogonal to $\Span\{{\bf v_i}\}_{i=1}^{\gamma_- + \gamma_0}$. Hence for any solution $\Vecf$ to the damped equation, there exists a unique set of $\{c_{0, i}\}_{i=1}^{\gamma_0} \cup \{c_{+, j}\}_{i=1}^{\gamma_+}$
such that for $\Vecg$ defined in~\eqref{def:g} with these coefficients, we have
\begin{align*}
    \BA \VU_g = \BA \VU_f \,,
\end{align*}
which is equivalent to~\eqref{eq:shift}.
 \end{proof}

Now we can construct the solution to the original kinetic equation.
\begin{prop} \label{prop:eta}
   Let $\Vecf$ be the solutions to the damped equation~\eqref{eq:damped} with $\Vech$ and $\Vecg$ the function defined in Proposition~\ref{prop:g}. Let $c_{0, i}$'s and $c_{+, i_+}$'s be the coefficient of $\Vecg$ given in Proposition~\ref{prop:g}. Let
\begin{align} \label{def:eta}
   \Vec \eta 
   = \Vecf - \Vecg 
       + \sum_{i=1}^{\nu_0} c_{0, i} \VecX_{0, i}
       + \sum_{i=1}^{\nu_+} c_{+, j} \VecX_{+, j} \,. 
\end{align}
Then $\Vec\eta$ is the unique solution to the original half-space equation~\eqref{eq:kinetic}. 
\end{prop}
\begin{proof}
Note that $\Vec\eta_1 = \Vecf - \Vecg$ satisfies 
\begin{align*}
    \mu \del_x \Vec\eta_1 &+ \CalL \Vec\eta_1 = 0 \,,
\\
    \Vec\eta_1 \big|_{\mu>0}
    =  \Vech &+ \CalK\vpran{\Vec\eta \big|_{\mu<0}}
       - \vpran{\sum_{i=1}^{\nu_0} c_{0, i} \VecX_{0, i}\big|_{\mu>0}
       + \sum_{i=1}^{\nu_+} c_{+, j} \VecX_{+, j}\big|_{\mu>0}} \,,
&& x = 0\,,  \, \mu > 0 \,,
\\
   \Vec\eta_1 &\to 0 \,,
&& \text{as $x \to \infty$}.
\end{align*}
Hence $\Vec\eta$ defined in~\eqref{def:eta} satisfies
\begin{align*}
    \mu \del_x \Vec\eta &+ \CalL \Vec\eta = 0 \,,
\\
    \Vec\eta = & \, \Vech + \CalK\Vec\eta \,,
\hspace{3.4cm} x = 0 \,, \mu > 0 \,,
\\
   \Vec\eta \to \sum_{i=1}^{\nu_0}  &\,c_{0, i} \VecX_{0, i}
       + \sum_{i=1}^{\nu_+} c_{+, j} \VecX_{+, j} \,,
\hspace{1cm} \text{as $x \to \infty$}.
\end{align*}
where $c_{0, i}, c_{0, j}$ are the coefficients defined in Proposition~\ref{prop:g}. We thereby have recovered $\Vec\eta$ as the unique solution to~\eqref{eq:kinetic}.
\end{proof}
 
Combining the error estimate in Proposition \ref{prop:approximation-NK} and the damping terms, we derive the final error estimate for our method as follows:
\begin{prop}
Suppose $\vec\eta$ is constructed as in \eqref{def:eta} in Proposition~\ref{prop:eta} with $\Vecf$, $\Vecg_{+, i}$, $\Vecg_{0, j}$ being numerical approximations obtained in Proposition~\ref{prop:approximation-NK} to the damped equation with appropriate boundary conditions. Suppose $\Vecf_h$ is the unique solution to the equation \eqref{eq:kinetic}.  Then there exists a constant $C_0$ such that
\begin{equation*}
    \|\Vecf_h - \vec\eta\|_\Gamma 
    \leq 
    C_0 \left(\inf_{\vec w \in \Gamma_{NK}} \norm{\Vecf_h - \vec w}_{\Gamma} 
      + \inf_{\vec w \in \Gamma_{NK}} \norm{\Vecf - \vec w}_\Gamma
      + \delta_{NK} \norm{\Vecf}_{(L^2(a\dv\dx))^m} \right)\,, 
  \end{equation*}
where $\norm{\cdot}_\Gamma$ is the norm defined in \eqref{def:norm-Gamma} and
\begin{equation*}
    \delta_{NK} 
    := \sum_{i=1}^{\nu_+} \inf_{\vec w \in \Gamma_{NK}} \|\Vecg_{+, i} - \vec w\|_\Gamma
    + \sum_{j=1}^{\nu_0} \inf_{\vec w \in \Gamma_{NK}} \|\Vecg_{0, j} - \vec w\|_\Gamma \,. 
  \end{equation*}
\end{prop}
\begin{proof}
The proof of this proposition only depends on the recovery procedures and the quasi-optimality shown in Proposition \ref{prop:approximation-NK}. In particular, it does not depend on the specific form of the boundary conditions. Hence it is identical to the proof for Proposition 3.9 in \cite{LiLuSun} and we omit the details. 
\end{proof}

\section{Numerical examples} \label{section:numerics} In this section
we show the numerical results of our algorithm for three
models, which cover the cases 
for multi-species, multi-dimensional (in the velocity variable), and multi-frequency systems. The three examples are: a linear transport
equation with two species, linearized BGK/Boltzmann equations with velocity in $\R^2$, and a
linearized transport equation with multi-frequency.  We treat these
three cases in order. Recall the general form of the half-space
equation:
\begin{equation}
  \begin{aligned} \label{eq:kinetic-repeat}
    & \mu \del_x \Vecf + \CalL \Vecf  = 0 \,, && \text{in } (0, \infty)\times \mathbb{V} \,,
    \\
    & \Vecf \big|_{\mu > 0} = \Vech(\mu) + \CalK \vpran{\Vecf \big|_{\mu < 0}} \,, && \text{at } x = 0 \,.
\end{aligned}
\end{equation}
As mentioned before, the boundary conditions for all these examples are either the Dirichlet condition with given incoming data or the classical Maxwell boundary condition such that
\begin{align*}
    \CalK = \alpha_r \CalK_r = \alpha_d \CalK_d + \alpha_s \CalK_s \,,
\end{align*}
where $\CalK_d$ is the diffuse reflection and $\CalK_s$ the specular reflection. 
For the convenience of numerical computation, we list out two properties of such $\CalK$ which can be verified by direct calculation:
\begin{lem} \label{lem:K-Maxwell-exp}
Let $\CalK$ the Maxwell boundary operator and let $\bar\CalK, \bar\CalK_s, \bar\CalK_d$ be the operators defined in Lemma~\ref{lem:bar-K-invertible}. Then
\begin{itemize}
\item[(a)] $\bar\CalK_s = \I$; 
\smallskip

\item[(b)] $(\I + \bar\CalK)^{-1}$ has the explicit form as 
\begin{align} \label{eq:I-bar-K-inv}
   (\I + \bar\CalK)^{-1} 
   = \frac{1}{1 + \alpha_{s}} \vpran{\I - \gamma_d \bar\CalK_{d}} \,,
\end{align}
where $\gamma_d = \alpha_{d}\bigl(1 +  \alpha_{d} + \alpha_{s}\bigr)^{-1}$. 
\end{itemize}
\end{lem}

\subsection{Linear Transport Equation with Two Species} 
\subsubsection{\underline{Formulation}} The first example that we
consider is the steady radiative transfer equation (RTE) with Thomson
(Rayleigh) scattering and polarization effect in planar geometry (see
\cite{Pom1973}*{Section 4.5}). In this model, the variables $I$ and
$Q$ denote the total intensity and the intensity difference of
light. The system \cite{Pom1973}*{Eq.~(4.211), page 135} depends on
the frequency which only serves as a parameter. Hence we simply ignore
the frequency dependence here. In this case the scattering
coefficients $\sigma, \sigma_s$ in \cite{Pom1973} are both constants. We consider a pure
scattering case with no source such that $\sigma = \sigma_s$ and
rescale $\sigma$ to be one. The speed of light $c$ is also normalized
to be one. Then the RTE has the form
\begin{equation} \label{eq:2-RTE}
\begin{gathered}
   \mu \del_x I 
   + I - \frac{1}{2} 
          \vpran{\int_{-1}^1 \vpran{1 + \frac{1}{2} p_2(\mu) p_2(\mu')}I(\mu') \dmu' 
                    - \frac{1}{2}\int_{-1}^1 p_2(\mu) \vpran{1 - p_2(\mu')} Q(\mu') \dmu'}
  = 0 \,,
\\
  \mu \del_x Q + Q 
      - \frac{1}{2} 
        \vpran{- \frac{1}{2} \int_{-1}^1 (1 - p_2(\mu)) p_2(\mu') I(\mu') \dmu'
                   + \frac{1}{2} \int_{-1}^1 (1 - p_2(\mu))(1 - p_2(\mu')) Q(\mu') \dmu'}
 = 0 \,,
\end{gathered}
\end{equation}
where $p_2(\mu) = \frac{3}{2} \mu^2 - \frac{1}{2}$ is the second-order Legendre polynomial. 

\subsubsection{\underline{Properties of $\CalL$ and $\CalK$}}
Denote $\Vecf = (I, Q)^T$.
Then the collision operator $\CalL$ has the form
\begin{align} \label{def:L}
  (\CalL \Vecf)(\mu)
  & = \Vecf(\mu) - \vint{\Sigma(\mu, \cdot) \Vecf(\cdot)}_{\mu} \,,
\qquad
   \Sigma(\mu, \mu')
=  \begin{pmatrix}
        \frac{1}{2} \vpran{1 + \frac{1}{2} p_2(\mu) p_2(\mu')} &  
        - \frac{1}{4} (1 - p_2(\mu)) p_2(\mu')  \\[2pt]
       - \frac{1}{4} (1 - p_2(\mu')) p_2(\mu) &
       \frac{1}{4} (1 - p_2(\mu)) (1- p_2(\mu'))
       \end{pmatrix}\,,
\end{align}
where we recall the notation $\viint{g_1}{g_2}_\mu = \int_{-1}^1 g_1 g_2 \dmu$. In this case, we have $\dsigma = \dmu$. First we check that
\begin{lem} \label{lem:assump-L-1}
The scattering operator $\CalL$ defined in~\eqref{def:L} satisfies (\PL1)-(\PL5).
\end{lem}
\begin{proof}
The attenuation coefficient $a$ in this case is $a(\mu) = 1$. One can then directly check that $\CalL$ is self-adjoint and $\NullL = \Span\{\VecX_0\}$ with
\begin{align} \label{def:X-0-polar}
    \VecX_0 = (1, 0)^T 
\quad \text{and} \quad
   \CalL^{-1}(\mu \VecX_0) = \mu \VecX_0 \,.
\end{align}
Property (\PL2) is also readily verified by the boundedness of $\Sigma(\mu, \mu')$. Furthermore, we can show by direct calculation again that 
there exists a constant $\beta_2 > 0$ such that 
\begin{align} \label{property:coer-L}
     \int_{\VV} \Vecf \cdot \CalL \Vecf \dmu
\geq \beta_2 \norm{\CalP^\perp \Vecf}_{(L^2(\dmu))^m}^2 \,,
\end{align}
where $\CalP^\perp$ is the projection onto $(\NullL)^\perp$.  Hence (\PL4) holds. 

To show that $\CalL$ satisfies (\PL5), we prove a more general statement: suppose $\CalL$
satisfies (\PL1)-(\PL3) together with \eqref{property:coer-L}, then
there exist constants $\alpha, \sigma_0 > 0$ such that (\PL5) holds. 
Indeed,
write $\Vecf = \Vecf^\perp + \Vecf_0 + \Vecf_+ + \Vecf_-$, where
\begin{align*}
    \Vecf^\perp = \CalP^\perp \Vecf \,,
\qquad
   \Vecf_{\,0} &= \sum_{i=1}^{\nu_0} \viint{\VecX_{0, i}}{\Vecf}_v \VecX_{0, i} \,,
\qquad
   \Vecf_\pm = \sum_{i=1}^{\nu_\pm} \viint{\VecX_{\pm, j}}{\Vecf}_v \VecX_{\pm, j} \,.
\end{align*}
Recall that $X_{0, i}, X_{\pm, j}$ are defined in Section~\ref{assump:L}.
Then by Cauchy-Schwarz there exist $c_1, c_2 > 0$ such that
\begin{equation}
\begin{gathered} \nn
     \viint{\Vecf}{\CalL \Vecf}_v 
\geq 
    \beta_2 \norm{\Vecf^\perp}_{(L^2(\dsigma))^m}^2 \,,
\qquad
   \sum_i \viint{\mu \VecX_{0,i}}{\Vecf}_v^2
   + \sum_i \viint{\mu \CalL^{-1}(\mu \VecX_{0,i})}{\Vecf}_v^2 \geq 0 \,,
\\
      \sum_i \viint{\mu \CalL^{-1}(\mu \VecX_{0,i})}{\Vecf}_v^2
\geq c_1 \norm{\Vecf_0}_{L^2(\dsigma)}^2
        - c_2 \vpran{\norm{\Vecf^\perp}_{L^2(\dsigma)}^2
                             + \norm{\Vecf_+}_{L^2(\dsigma)}^2
                             + \norm{\Vecf_-}_{L^2(\dsigma)}^2}  \,,
\\  
     \sum_{j} \viint{\mu \VecX_{+,j}}{\Vecf}_v^2
     +   \sum_k \viint{\mu \VecX_{-,k}}{\Vecf}_v^2
 \geq
    c_1 \vpran{\norm{\Vecf_+}_{L^2(\dsigma)}^2
                  + \norm{\Vecf_-}_{L^2(\dsigma)}^2 }
    - c_2 \norm{\Vecf^\perp}_{L^2(\dsigma)}^2 \,.
\end{gathered}
\end{equation}
Therefore, if the coefficient $\alpha$ in the definition of $\CalL_d$ in~\eqref{def:L-d-1} is small enough, then
\begin{align} \label{est:L-d-1}
   \int_\VV \Vecf \cdot \CalL_d \Vecf \dsigma
&\geq 
   \frac{\beta_2}{2} \norm{\Vecf^\perp}_{(L^2(\dsigma))^m}^2
   + c_1 \alpha \vpran{\norm{\Vecf_+}_{(L^2(\dsigma))^m}^2
                  + \norm{\Vecf_-}_{(L^2(\dsigma))^m}^2 } \,.
\end{align}
Furthermore, 
\begin{equation}
\begin{aligned}
  \int_\VV \Vecf \cdot \CalL_d \Vecf \dsigma
&\geq
   \alpha\sum_i \viint{\mu \CalL^{-1}(\mu \VecX_{0,k})}{\Vecf}_v^2 
\\
&\hspace{-0.5cm}\geq
   c_1 \alpha \norm{\Vecf_{0}}_{(L^2(\dsigma))^m}^2
   - c_2 \alpha \vpran{\norm{\Vecf^\perp}_{(L^2(\dsigma))^m}^2
                       + \norm{\Vecf_+}_{(L^2(\dsigma))^m}^2
                  + \norm{\Vecf_-}_{(L^2(\dsigma))^m}^2} \,.
\end{aligned}
\end{equation}
Hence by multiplying~\eqref{est:L-d-1} by $\max\{\frac{2c_2}{c_1}, \frac{4c_1\alpha}{\beta_2}\}$, we have 
\begin{align*}
     \int_{\VV} \Vecf \cdot \CalL_d \Vecf \dsigma
\geq \sigma_1 \norm{\Vecf}_{(L^2(\dsigma))^m}^2
\end{align*}
for some $\sigma_1 > 0$ which depends on $\alpha$. Applying the above estimates to $\CalL$ given in~\eqref{def:L} (note that in this case we only have $H^0$ or $X_0$), we conclude that such $\CalL$ satisfies assumptions (\PL1)-(\PL5).
\end{proof}

The boundary operator $\CalK$ is defined as
\begin{align} \label{def:K}
   \CalK \vpran{\Vecf \big|_{\mu<0}} 
&  = \alpha_d \CalK_d \Vecf + \alpha_s \CalK_s\vpran{\Vecf \big|_{\mu<0}}
\\
&  = \alpha_d \vpran{\int_{\mu' < 0} |\mu'| \Vecf \cdot \VecX_0 \dmu'} 
                  \frac{\VecX_0}{\int_{\mu' > 0} \mu' \VecX_0 \cdot \VecX_0 \dmu'}
                   + \alpha_s \, \Vecf \big|_{\mu<0}(-\mu) \,, \nn
\end{align}
where $\VecX_0$ is given in~\eqref{def:X-0-polar}, and $\alpha_d, \alpha_s \geq 0$, $0 \leq \alpha_d + \alpha_s < 1$. We have 
\begin{lem} \label{lem:assump-K-1}
The boundary opeartor $\CalK$ defined in~\eqref{def:K} satisfies (\PK).
\end{lem}
\begin{proof}
Since Lemma~\ref{lem:suff-cond-PK} only covers the scalar case, we show the details of proof for the current vector case. As commented in Remark~\ref{rmk:Maxwell-BC}, we only need to show that $\CalK_d$ satisfies (\PK). 
To this end, we use the symmetry of $\VecX_0$ in $\mu$ and Cauchy-Schwarz to obtain that
\begin{align*}
  \int_{\mu > 0} \mu |\CalK_d\Vecf|^2 \dmu
&= \vpran{\int_{\mu > 0} \mu  \abs{\VecX_0}^2 \dmu} 
     \abs{\frac{\int_{\mu < 0} |\mu| \Vecf \cdot \VecX_0 \dmu}{\int_{\mu > 0} \mu \VecX_0 \cdot \VecX_0 \dmu}}^2
\leq 
   {\int_{\mu < 0} |\mu| |\Vecf|^2 \dmu} \,,
\end{align*}
which is the desired property (\PK).
\end{proof}

The above two Lemmas show that the theory in Section~\ref{sec:well-posedness} applies to equation~\eqref{eq:2-RTE}. The details for terms in the weak formulation are as follows. First, 
\begin{align*}
    \CalL_d \Vecf 
&   = \CalL \Vecf - \alpha \mu \VecX_0 \viint{\mu \VecX_0}{\Vecf}_\mu 
       - \alpha \mu \CalL^{-1}(\mu \VecX_0) \viint{\mu \CalL^{-1}(\mu \VecX_0)}{\Vecf}_\mu
\\
& = \CalL \Vecf - \alpha \mu \VecX_0 \viint{\mu \VecX_0}{\Vecf}_\mu 
       - \alpha \mu^2 \VecX_0 \viint{\mu^2 \VecX_0}{\Vecf}_\mu \,.
\end{align*}
The Maxwell boundary condition is (recall \eqref{def:K})
\begin{align} \label{cond:bdry-polar} \Vecf \big|_{\mu> 0}  = \Vech +
  \CalK \vpran{\Vecf \big|_{\mu<0}} 
   = \Vech + \alpha_d \CalK_d + \alpha_s \CalK_s \,. 
\end{align}
Thus, by Lemma~\ref{lem:K-Maxwell-exp}, the boundary operators in the
weak formulation $\CalB(\Vecphi, \Vecf) = l(\Vecphi)$ are
\begin{align*}
    (\I + \bar\CalK)^{-1} 
    &= \frac{1}{1+\alpha_s} \vpran{\I - \frac{\alpha_d}{1+\alpha_d+\alpha_s} \bar\CalK_d}\,,
\\
   (\I + \bar\CalK)^{-1} (\I - \bar\CalK) 
   &= \frac{1-\alpha_s}{1+\alpha_s} 
     \vpran{\I - \frac{2\alpha_d}{(1-\alpha_s)(1+\alpha_d+\alpha_s)} \bar\CalK_d} \,,
\end{align*}
where
\begin{align} \label{def:bar-K-2-transport}
   \bar\CalK_d \Vecf
   = \bar\CalK_d (I, Q)^T
   = 2\vpran{\int_{\mu' > 0} \mu' I(0, \mu') \dmu'} (1, 0)^T \,.
\end{align}
Let $\Vecg_0$ be the special solution to the damped equation with the boundary condition
\begin{align*}
 \Vecg_0 \big|_{\mu>0} 
 = \vpran{\VecX_0 - \CalK (\VecX_0 \big|_{\mu<0})} + \CalK (\Vecg_0 \big|_{\mu<0}) \,, \qquad x = 0 \,. 
\end{align*}
Then the true solution is given by 
\begin{align*}
     \Vecf_h = \Vecf - c_h \Vecg_0 + c_h \VecX_0 \,,
\end{align*}
where $\Vecf$ is the solution to the damped equation with the boundary condition~\eqref{cond:bdry-polar} and $c_h = \frac{\viint{\mu \VecX_0}{\Vecf}_\mu}{\viint{\mu \VecX_0}{\Vecg_0}_\mu}$.

\subsubsection{\underline{Algorithm}}\label{sec:algorithm_ms}
For this example, we choose half-space Lagendre polynomials as the basis functions for each component of $(I, Q)^T$. Namely, we first find the half-space Lagendre polynomials by:
\begin{equation*}
    \int_0^1 \tilde{\phi}_m(\mu)\tilde{\phi}_n(\mu) \dmu = \delta_{mn}\,,\quad 
    m, n = 1, 2, \cdots\,.
\end{equation*}
Then we use the even-odd extension to obtain the basis functions
for the finite-dimensional space over $[-1,1]$:
\begin{equation*}
\Gamma_{N,1} = \Span\left\{\phi_{m} \right\}_{m=1}^{2N-1} \,,
\end{equation*}
where
\begin{align*}
   \phi_{2m}(\mu) 
  = \begin{cases}
      \tilde\phi_m(\mu) \,, & \mu \in [0, 1] \,, \\[2pt] 
      \tilde\phi_m(-\mu) \,, & \mu \in [-1, 0)
    \end{cases}
\qquad \text{and} \qquad
   \phi_{2m-1}(\mu) 
   = \begin{cases}
        \tilde{\phi}_m(\mu) \,, & \mu \in [0, 1] \,, \\[2pt]
        -\tilde{\phi}_m(-\mu) \,, & \mu \in [-1, 0) \,.
        \end{cases}
\end{align*}
In this multi-species case, $\Vecf$ is a two dimensional vector and
the basis function for $\Vecf$ is chosen to be:
\begin{equation*}
\Gamma_N = \Span\left\{\vec{\phi}_m \right\}_{m=1}^{4N-2} 
\end{equation*}
with
\begin{equation*}
\vec{\phi}_m = \left(\begin{array}{c}\phi_m(x)\\0\end{array}\right)
\quad \text{and} \quad
\vec{\phi}_{m+2N-1} =\left(\begin{array}{c}0\\ \phi_{m}(x)\end{array}\right)\,,
\quad m = 1, \cdots, 4N-2 \,.
\end{equation*}
Using these basis functions, the ODE system becomes
\begin{equation*}
\mathsf{A}\frac{\rd}{\rd x}\vec{\alpha} = \mathsf{B}\vec{\alpha}\,,
\end{equation*}
where 
\begin{align*}
    \mathsf{A}_{mn} = \int_{-1}^1 \mu \Vecphi_m \cdot \Vecphi_n \dmu \,,
\qquad
    \mathsf{B}_{mn} = \int_{-1}^1 \Vecphi_m \cdot \CalL_d \Vecphi_n \dmu
\end{align*}
for $m, n = 1, \cdots, 4N-2$. Therefore both $\mathsf{A}$ and $\mathsf{B}$ are of size $(4N-2)\times (4N-2)$.

\subsubsection{\underline{Numerical Results}} 
In this part we show the numerical results regarding this multi-species model with pure incoming data. 
\smallskip

\noindent\emph{Example 4.1 (Incoming boundary condition).}
Set $\Vech = (2\mu, \mu)^T$ and $\alpha_d = \alpha_s = 0$. The numerical solutions at $x=0$ for both components are shown on the left in Figure~\ref{fig:MS_Dirichlet_recover}. 
The plot on the right in Figure~\ref{fig:MS_Dirichlet_recover} shows the convergence rate of the second component $Q$ where we observe an algebraic convergence rate. This is within expectation, as even though the even-odd decomposition captures the jump discontinuity at $\mu = 0$, the solution still has a weak derivative discontinuity near $\mu = 0$ \cite{CLT2014, TF2013}.
\begin{figure}
\centering
\includegraphics[height = 0.25\textheight,width = 0.4\textwidth]{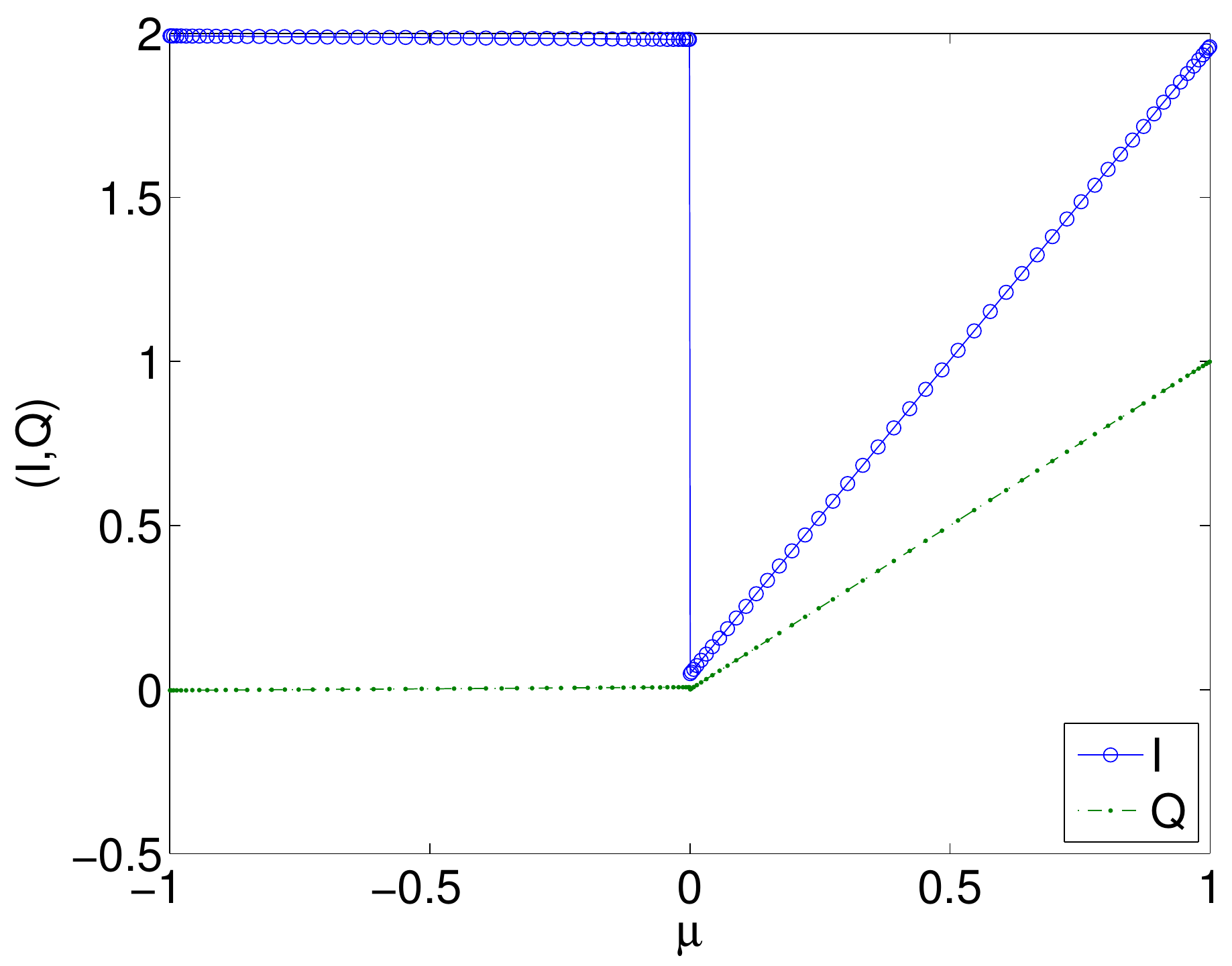}
\includegraphics[height = 0.25\textheight,width = 0.4\textwidth]{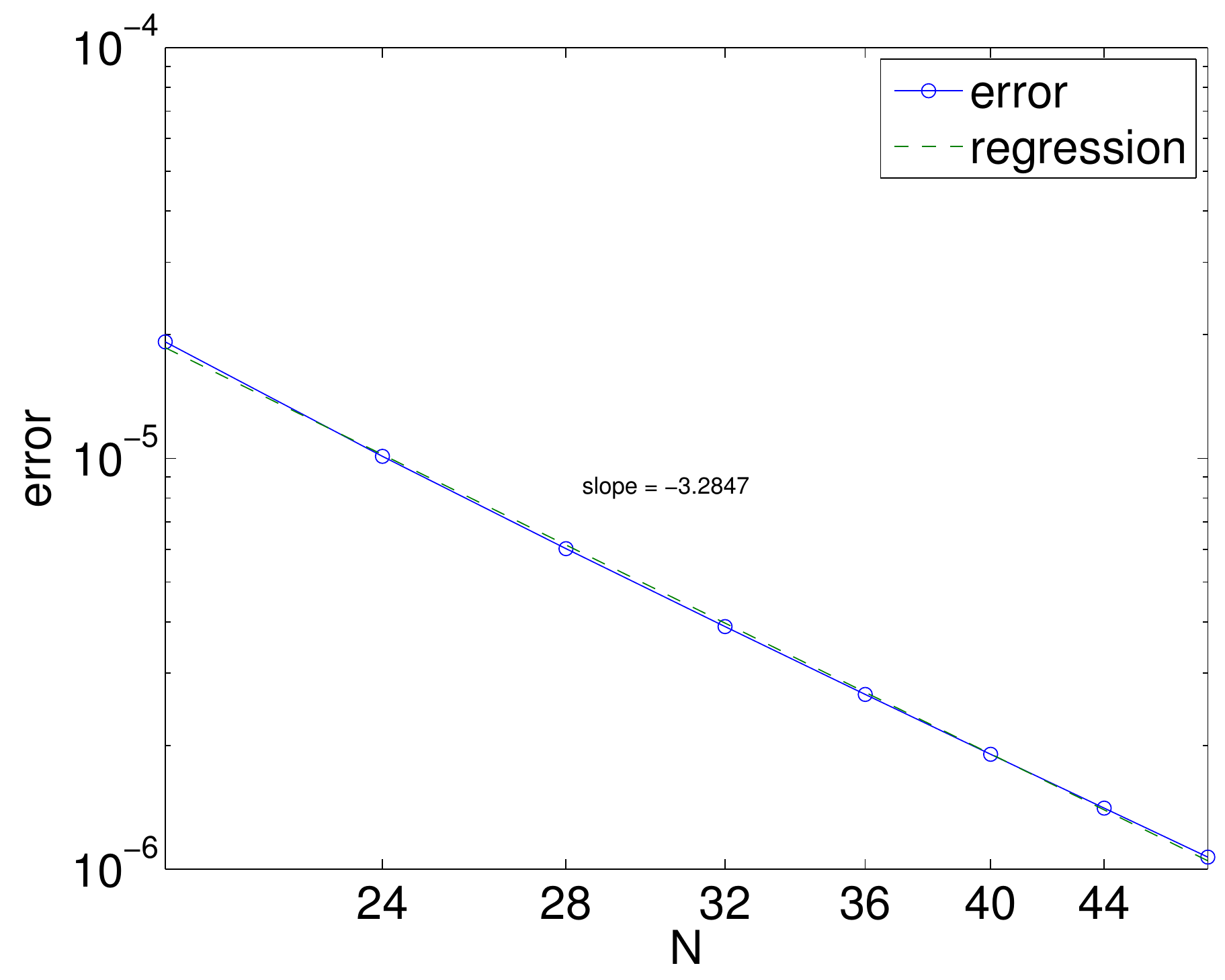}
\caption{(Example 4.1) For the two-species RTE~\eqref{eq:2-RTE}, we set $\Vech =
  (2\mu, \mu)^T$ and $\alpha_d = \alpha_s =0$. The left panel shows
  the numerical solutions at $x=0$ for both species using $31$ (i.e. $N = 16$) basis functions. The right panel shows the convergence rate of
  the $Q$ component of the RTE.}\label{fig:MS_Dirichlet_recover}
\end{figure}

\subsection{Linearized BGK Equation} 

\subsubsection{\underline{Formulation}}
The second example we consider is the time-independent linearized BGK equation for a single species with its velocity $v = (\mu, v_y) \in \R^2$. Here we normalize the wall temperature and denote the absolute Maxwellian $M$ as the wall Maxwellian such that
\begin{align*}
     M(v) = \frac{1}{2 \pi} e^{- \frac{|v|^2}{2}} \,.
\end{align*}
Suppose $F$ is the density function for the nonlinear BGK equation. Let $f$ be the perturbation such that
\begin{align*}
     F 
     = M + M f \,.
\end{align*}
Then the linearized collision operator $\CalL$ has the form
\begin{align} \label{def:L-BGK}
   \CalL f = f - \CalP f \,, 
\end{align}
where $\CalP: L^2(M\dv) \to \Span\{1, v, |v|^2\}$ is the projection operator. In this case, $\dv$ is the usual Lebesgue measure and $\dsigma = M\dv$.

The boundary operator $\CalK$ is given by
\begin{equation}
\begin{aligned} \label{def:K-BGK}
   \CalK f 
 & = \alpha_d \CalK_d f + \alpha_s \CalK_s\vpran{f \big|_{\mu<0}} \\   
   &= \alpha_d \vpran{\int_{\mu' < 0} |\mu'| f(v) \, M \dv} 
                  \frac{1}{\int_{\mu' > 0} |\mu'| M \dv}
                   + \alpha_s f \big|_{\mu<0}(\CalR v)
\\
  & = \alpha_d \sqrt{2 \pi} \int_{\mu' < 0} |\mu'| f(v) \, M \dv
                   + \alpha_s f \big|_{\mu<0}(\CalR v)
\,,
\end{aligned}
\end{equation}
where $\CalR v = (-\mu, v_2)$, $\alpha_d, \alpha_s \geq 0$ and $\alpha_d + \alpha_s < 1$. We have
\begin{lem}
The linear operator $\CalL$ defined in \eqref{def:L-BGK} satisfies (\PL1)-(\PL5) and the boundary opeartor $\CalK$ defined in~\eqref{def:K-BGK} satisfies ({\PK}).
\end{lem}
\begin{proof}
It is classical to show that $\CalL$ satisfies (\PL1)-(\PL4) and also the stronger coerciveness~\eqref{property:coer-L}. Therefore by the proof of Lemma~\ref{lem:assump-L-1}, condition (\PL5) also holds. The boundary operator $\CalK_d$ in this case fits the form in Lemma~\ref{lem:suff-cond-PK}, which guarantees that (\PK) holds.
\end{proof}

\begin{rmk}
Our method can also be applied to the linearized BGK equation for multi-species. The linearization of these species is chosen slightly differently depending on whether there is diffusion reflection or not. In the case where there is  nontrivial diffuse reflection from the wall, the equilibrium state for different particles will all be the same, which is the Maxwellian $M$ given by the wall.  We can then linearize the vector-valued density function $\vec F = (F_1, \cdots, F_m)$ as
\begin{align*}
    \vec F = M (1, \cdots, 1)^T + M(f_1, \cdots, f_m)^T \,.
\end{align*}

On the other hand, if there is only the incoming data and/or the specular reflection, $i.e.$, $\alpha_d = 0$, then we allow the equilibrium states of different species to be different. In this case, we linearize $\vec F = (F_1, \cdots, F_m)$ as
\begin{align*}
    F_i = M_i + \sqrt{M_i} f_i \,, 
\qquad i = 1, \cdots, m \,.
\end{align*}
The advantage of this linearization is that the function space for $\Vecf$ is given by $(L^2(\dv))^m$ instead of the weighted-$L^2$ by various Maxwellians $M_i$ for each component $f_i$. 
\end{rmk}

To define the particular damped operator in the case of linearized BGK equation, we compute the eigenmodes in $\NullL$: 
\begin{equation}
\begin{gathered} \nn
    X_{0, 1} =  \mu^2 + v_y^2 - 4 \,,
\qquad
    X_{0, 2} = v_y \,,
\\
    X_+ = \mu^2 + v_y^2 + 2 \mu \,,
\qquad
   X_- =  \mu^2 + v_y^2 - 2 \mu \,.
\end{gathered}
\end{equation}
The associated eigenspaces are 
\begin{align*}
    H^0 = \Span\{X_{0, 1}, X_{0, 2}\} \,,
\qquad
    H^\pm = \Span\{X_\pm\}.
\end{align*}
Moreover, $\CalL^{-1}(\mu X_{0,k})$ is computed as
\begin{align*}
    \CalL^{-1}(\mu X_{0, 1}) 
    = \mu X_{0, 1}
    = \mu (\mu^2 + v_y^2 - 4)  \,,
\qquad
    \CalL^{-1}(\mu X_{0, 2}) = \mu X_{0, 2} = \mu v_y \,.
\end{align*}
Hence, the damped operator has the form
\begin{align*}
  \CalL_d f 
 =& \CalL f + \alpha   
         \mu X_{+} \viint{\mu X_{+}}{f}_v
     + \alpha  
           \mu X_{-} \viint{\mu X_{-}}{f}_v
\\
  &  + \alpha \sum_{k=1}^{2} 
           \mu X_{0,k} \viint{\mu X_{0,k}}{f}_v                            
       + \alpha \sum_{k=1}^{2}
              \mu^2 X_{0,k}
             \viint{\mu^2 X_{0,k}}{f}_v \,,
\end{align*}
where $\viint{f_1}{f_2}_v = \int_{\R^2} f_1 f_2 M\dv$. The boundary condition is given as 
\begin{align*}
    f \big|_{\mu > 0} = h + \CalK \vpran{f \big|_{\mu<0}} \,,
\qquad x = 0 \,,
\end{align*}
where recall that $\CalK$ is given in~\eqref{def:K-BGK}. By Lemma~\ref{lem:K-Maxwell-exp},  the boundary operators in the weak formulation are
\begin{align*}
    (\I + \bar\CalK)^{-1} 
    &= \frac{1}{1+\alpha_s} \vpran{\I - \frac{\alpha_d}{1+\alpha_d+\alpha_s} \bar\CalK_d}\,,
\\
   (\I + \bar\CalK)^{-1} (\I - \bar\CalK) 
   &= \frac{1-\alpha_s}{1+\alpha_s} 
     \vpran{\I - \frac{2\alpha_d}{(1-\alpha_s)(1+\alpha_d+\alpha_s)} \bar\CalK_d} \,,
\end{align*}
where $\bar\CalK_d f= \sqrt{2\pi} \int_{\mu'>0} \mu' f(v) M\dv$.

In order to obtain the original solution to~\eqref{eq:kinetic}, we construct the special solutions $g_{0, 1}, g_{0, 2} , g_+$ such that
they satisfy the damped equation and the boundary conditions respectively:
\begin{align*}
 g_{0,1} \big|_{\mu>0} 
 = \vpran{X_{0,1} - \CalK (X_{0,1} \big|_{\mu<0})} + \CalK (g_{0,1} \big|_{\mu<0}) \,,
\\
 g_{0,2} \big|_{\mu>0} 
 = \vpran{X_{0,2} - \CalK (X_{0,2} \big|_{\mu<0})} + \CalK (g_{0,2} \big|_{\mu<0}) \,,
\\
 g_{+} \big|_{\mu>0} 
 = \vpran{X_{+} - \CalK (X_{+} \big|_{\mu<0})} + \CalK (g_{+} \big|_{\mu<0}) \,.
\end{align*}
The true solution $f_h$ is then given as
\begin{align*}
   f_h = f - \vpran{c_{0, 1} g_{0, 1} + c_{0, 2} g_{0, 2} + c_+ g_+}
            + \vpran{c_{0, 1} X_{0, 1} + c_{0, 2} X_{0, 2} + c_+ X_+} \,,
\end{align*}
where the coefficients $c_{0, 1}, c_{0, 2}, c_+$ satisfy that 
\begin{align*}
    \viint{\mu X_{0, 1}}{g} = \viint{\mu X_{0, 1}}{f} \,,
\qquad
    \viint{\mu X_{0, 2}}{g} = \viint{\mu X_{0, 2}}{f} \,,
\qquad
    \viint{\mu X_{+}}{g} = \viint{\mu X_{+}}{f} 
\end{align*}
with $g = c_{0, 1} g_{0, 1} + c_{0, 2} g_{0, 2} + c_+ g_+$. The unique solvability of $g$ is guaranteed by Proposition~\ref{prop:g}.

\subsubsection{\underline{Algorithm}:}
For the 2D-BGK case, we build the basis functions upon 
Hermite polynomials. Since the solution is regular in $v_y$, we use the full Hermite polynomials on $\R$ for $v_y$. To take into account of the jump discontinuity in $\mu = v_x$, we apply even/odd extensions of half-space Hermite polynomials on $[0, \infty)$.   
Specifically, the half-space Hermite polynomials $\{B_m(\mu)\}_{n=1}^\infty$ satisfy
\begin{equation*}
\int_0^\infty B_m(\mu)B_n(\mu) e^{-\frac{\mu^2}{2}} \dmu = \delta_{mn} \,.
\end{equation*}
Performing the even and odd extensions of $\{B_m(\mu)\}_{n=1}^\infty$ gives
\begin{equation*}
    B^E(\mu) 
    = \begin{cases}
          B_n(\mu)/\sqrt{2} \,, \quad & \mu>0 \,,\\[2pt]
          B_n(-\mu)/\sqrt{2} \,, \quad & \mu<0
        \end{cases}
\quad \text{and} \quad 
     B^O(\mu) 
     =  \begin{cases}
             B_n(\mu)/\sqrt{2} \,, \quad & \mu>0 \,,\\[2pt] 
             -B_n(-\mu)/\sqrt{2} \,,  \quad & \mu<0 \,.
           \end{cases}
\end{equation*}
Then the set of basis functions for the finite dimensional space in $\mu$ is given by 
\begin{equation*}
   \Gamma_{x,N}
   = \left\{\phi_{x,2n-1} \right\}_{n = 1}^{N+1}
       \cup \left\{\phi_{x,2n} \right\}_{n = 1}^{N}
   = \left\{B^O_{n-1}\right\}_{n = 1}^{N+1} 
      \cup \left\{B^E_{n-1} \right\}_{n = 1}^{N} \,.
\end{equation*}
The set of basis functions in $v_y$  is 
\begin{equation*}
    \Gamma_{y,N} 
    = \left\{\phi_{y,n} \right\}_{n=1}^{N} \,,\text{where}\quad\int_{-\infty}^\infty \phi_{y, m}(v_y)\phi_{n,y}(v_y) e^{-\frac{v_y^2}{2}} \dv_y = \delta_{mn} \,.
\end{equation*}
The basis for the approximation solution $f_N$ is then expanded by $\Gamma_{x,N}\otimes\Gamma_{y,N}$ such that 
\begin{align*}
   f_N(\mu, v_y) 
   = \sum_{m=1}^{2N+1}\sum_{n=1}^N
           \beta_{mn}\phi_{x,m}(\mu)\phi_{y,n}(v_y) \,.
\end{align*}
The ODE system still has the form 
\begin{align*}
   \mathsf{A}\frac{\rd}{\rd x}\alpha = \mathsf{B}\alpha \,,
\end{align*} 
where 
\begin{align*}
    \mathsf{A}_{mn}^{pq} 
    = \vpran{\int_{\R} \mu \phi_{x,m}(\mu) \phi_{x,n}(\mu) \dmu}
       \delta_{pq}  \,,
\quad
    \mathsf{B}_{mn}^{pq} 
    = \int_{\R^2} \phi_{x,m}(\mu)\phi_{y,p}(v_y) \, \big[\CalL_d \vpran{\phi_{x,n} \phi_{y,q}}\big](\mu, v_y) \dmu\dv_y \,.
\end{align*}

\subsubsection{\underline{Numerical Results}} Examples are shown in Figure~\ref{fig:MD_Dirichlet_recover} and Figure~\ref{fig:MD_Maxwell_recover} for both the pure incoming data and the Maxwell cases. 

\smallskip

\noindent \emph{Example 4.2.1 Incoming boundary condition.}  In the
first example, we use incoming boundary
condition. Figure~\ref{fig:MD_Dirichlet_recover} verifies that if $h
\in H^0 \oplus H^+$, then the solution is simply $f_h = h$, as
expected from the theory.  \smallskip

\begin{figure}
\centering
\includegraphics[height = 0.2\textheight,width = 0.3\textwidth]{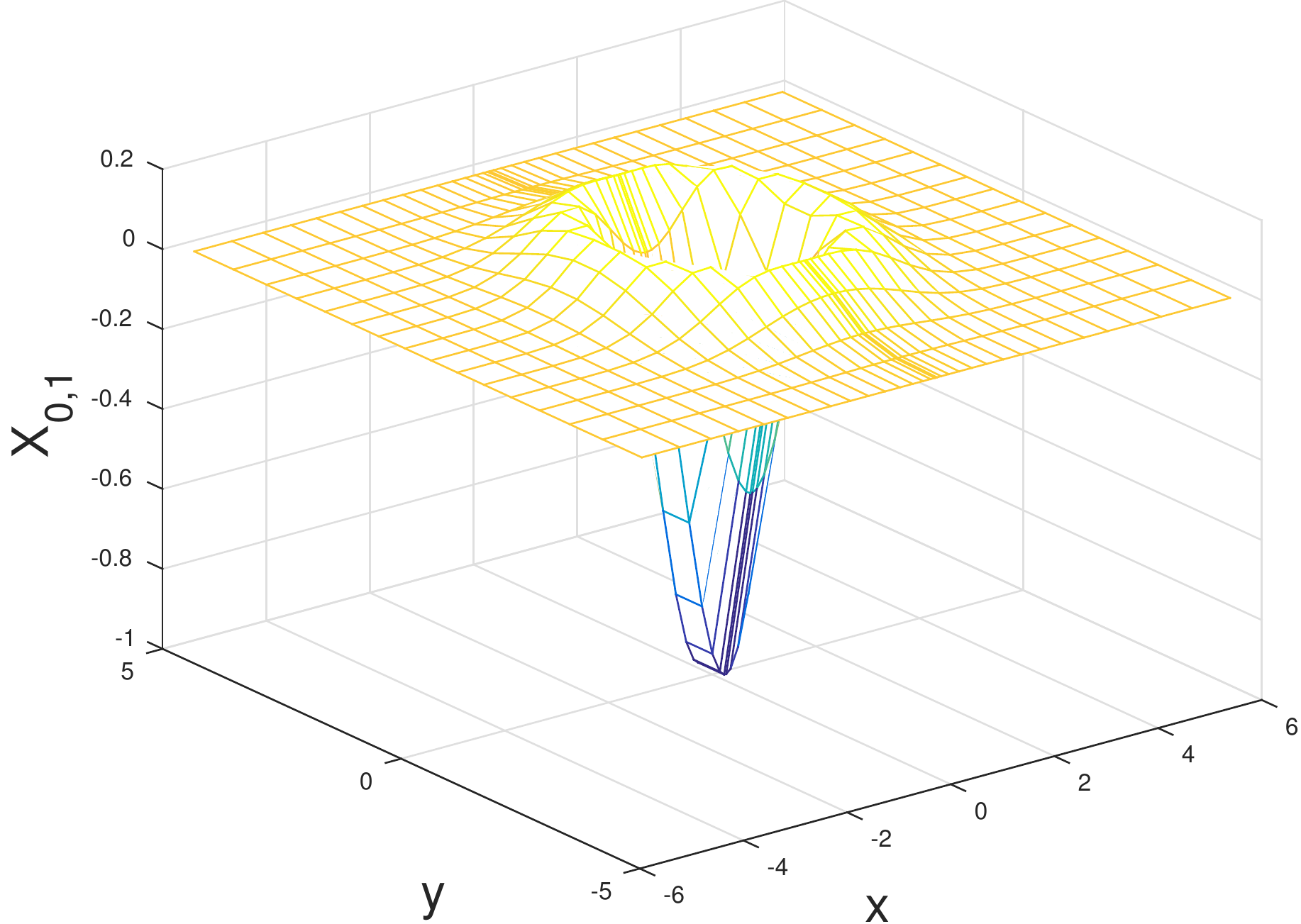}
\includegraphics[height = 0.2\textheight,width = 0.3\textwidth]{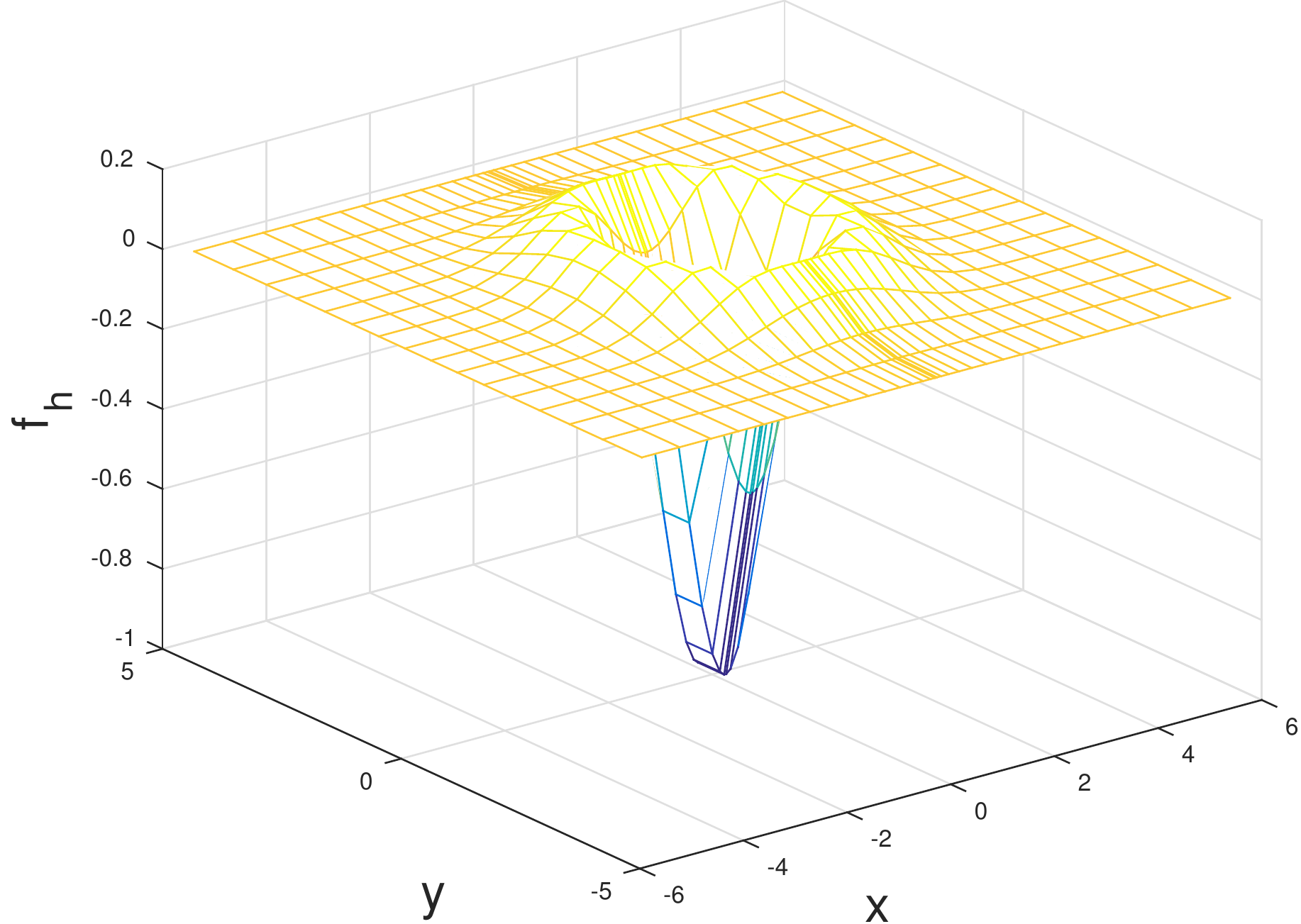}
\includegraphics[height = 0.2\textheight,width = 0.3\textwidth]{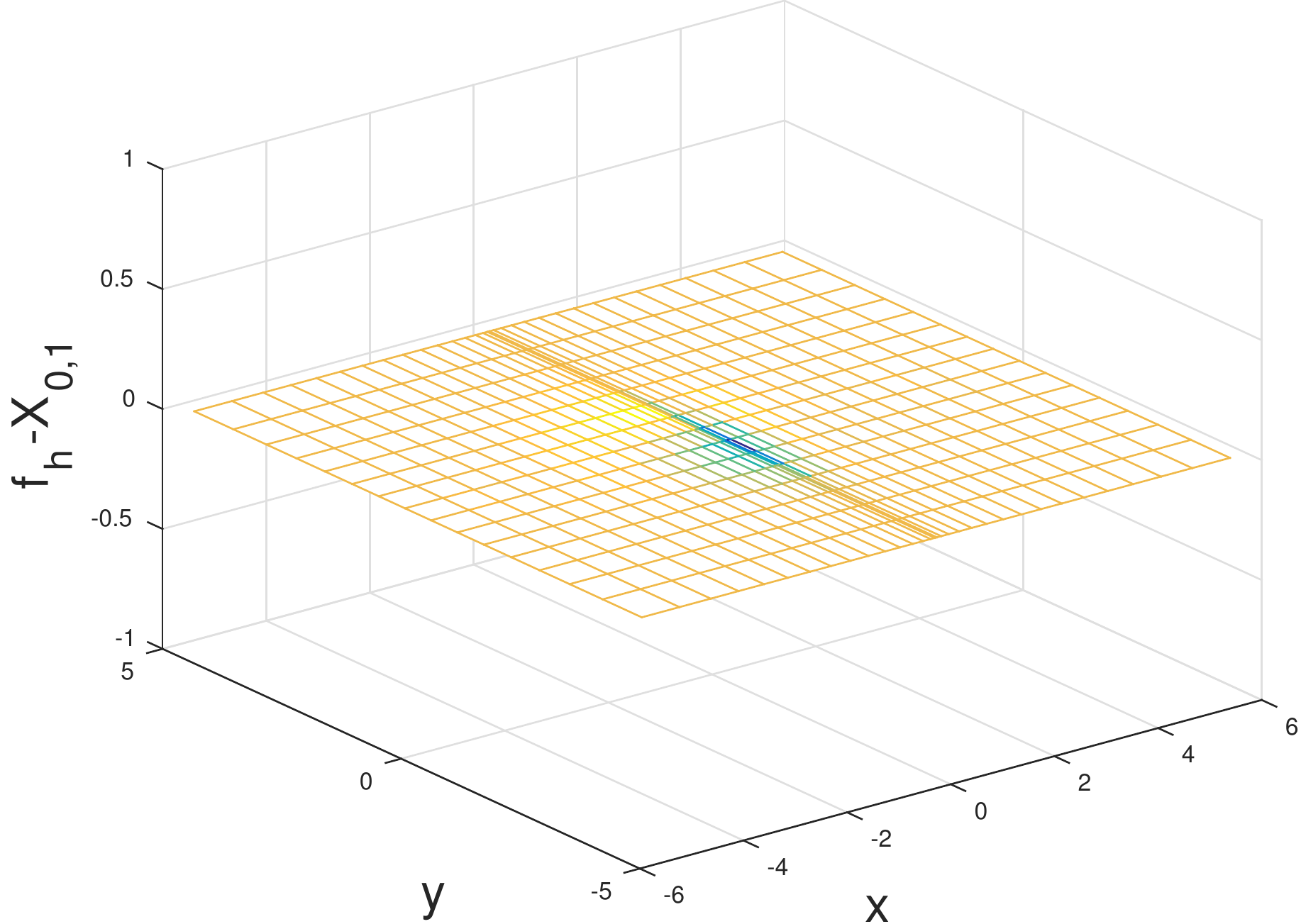}\\
\includegraphics[height = 0.2\textheight,width = 0.3\textwidth]{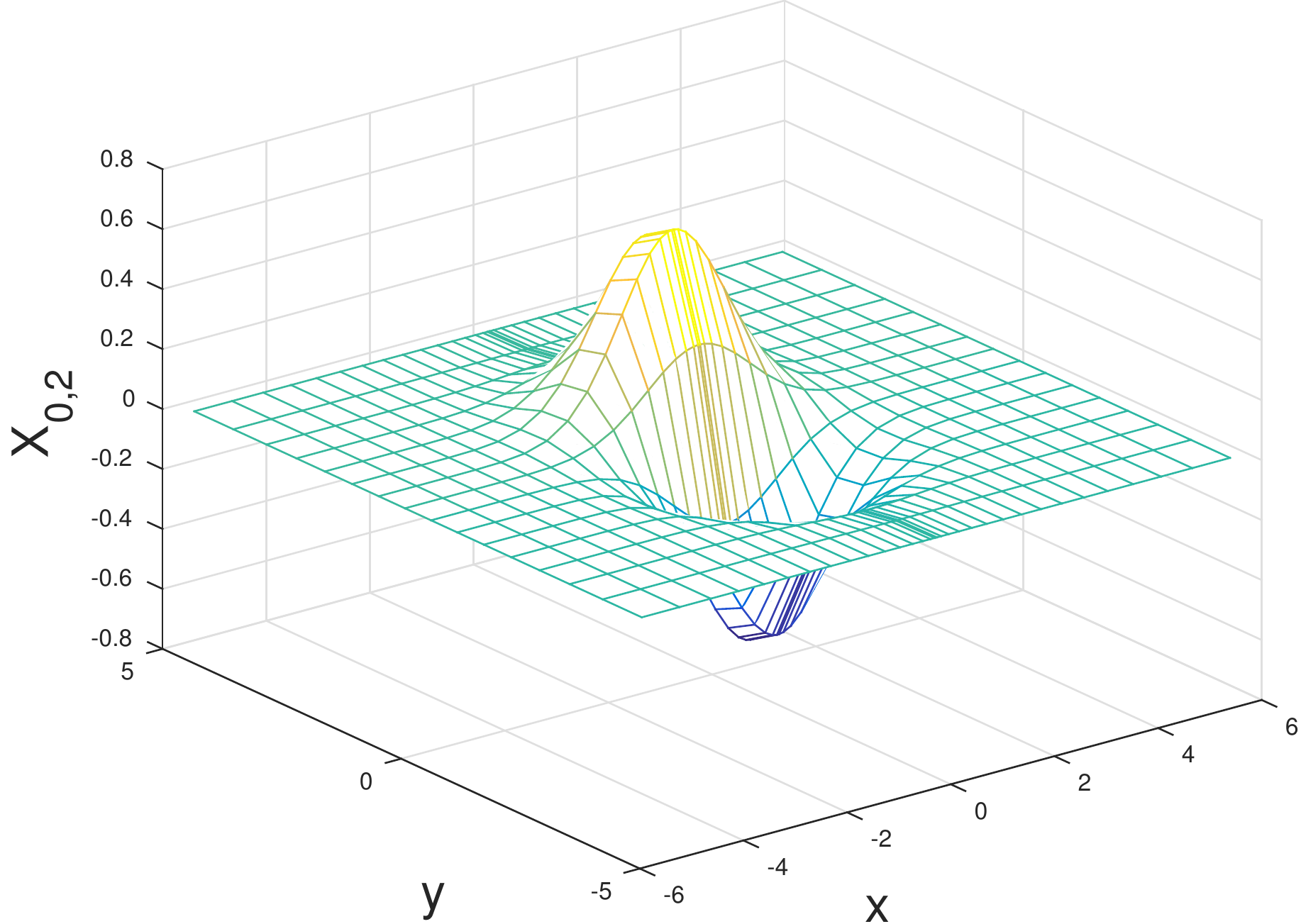}
\includegraphics[height = 0.2\textheight,width = 0.3\textwidth]{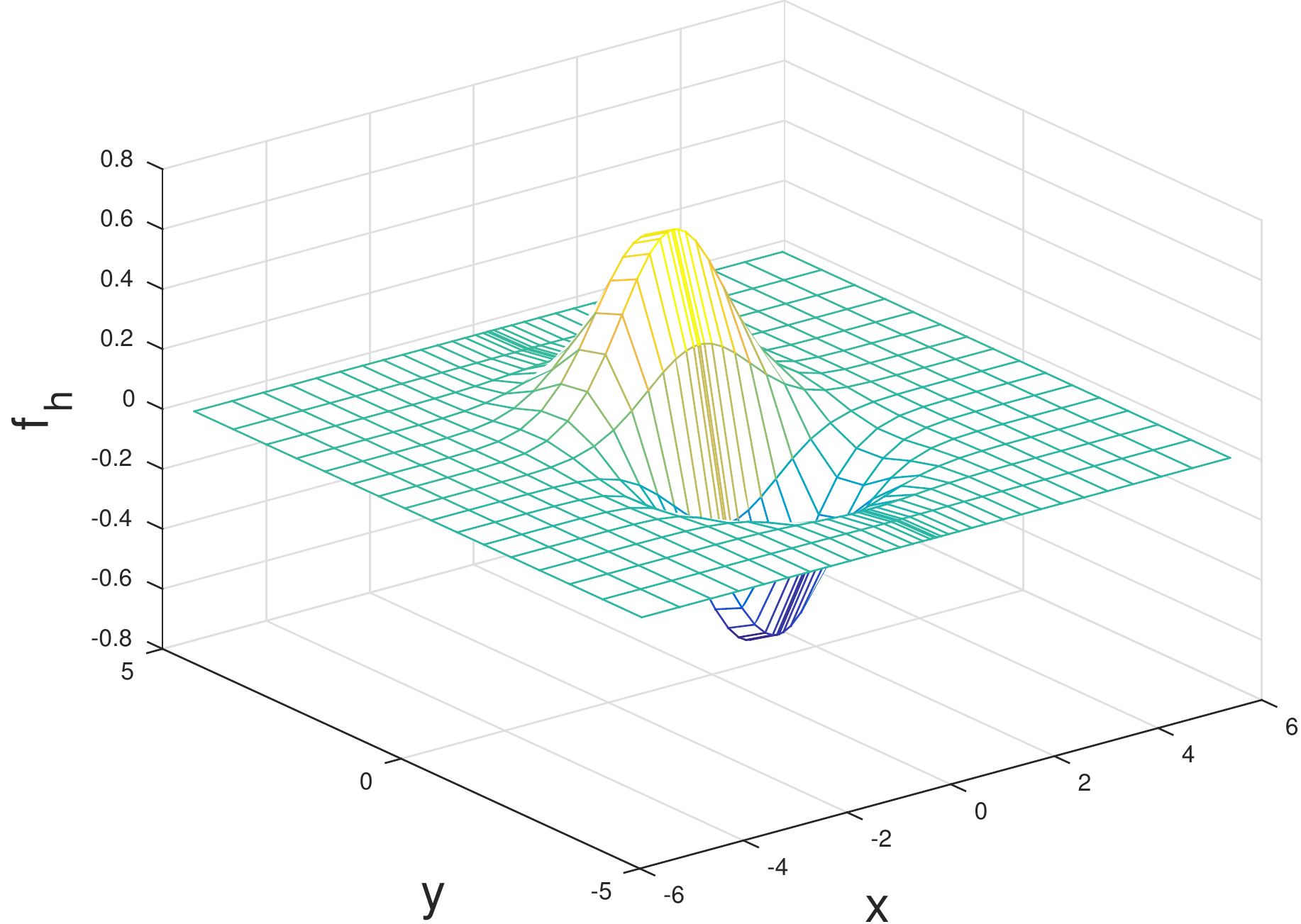}
\includegraphics[height = 0.2\textheight,width = 0.3\textwidth]{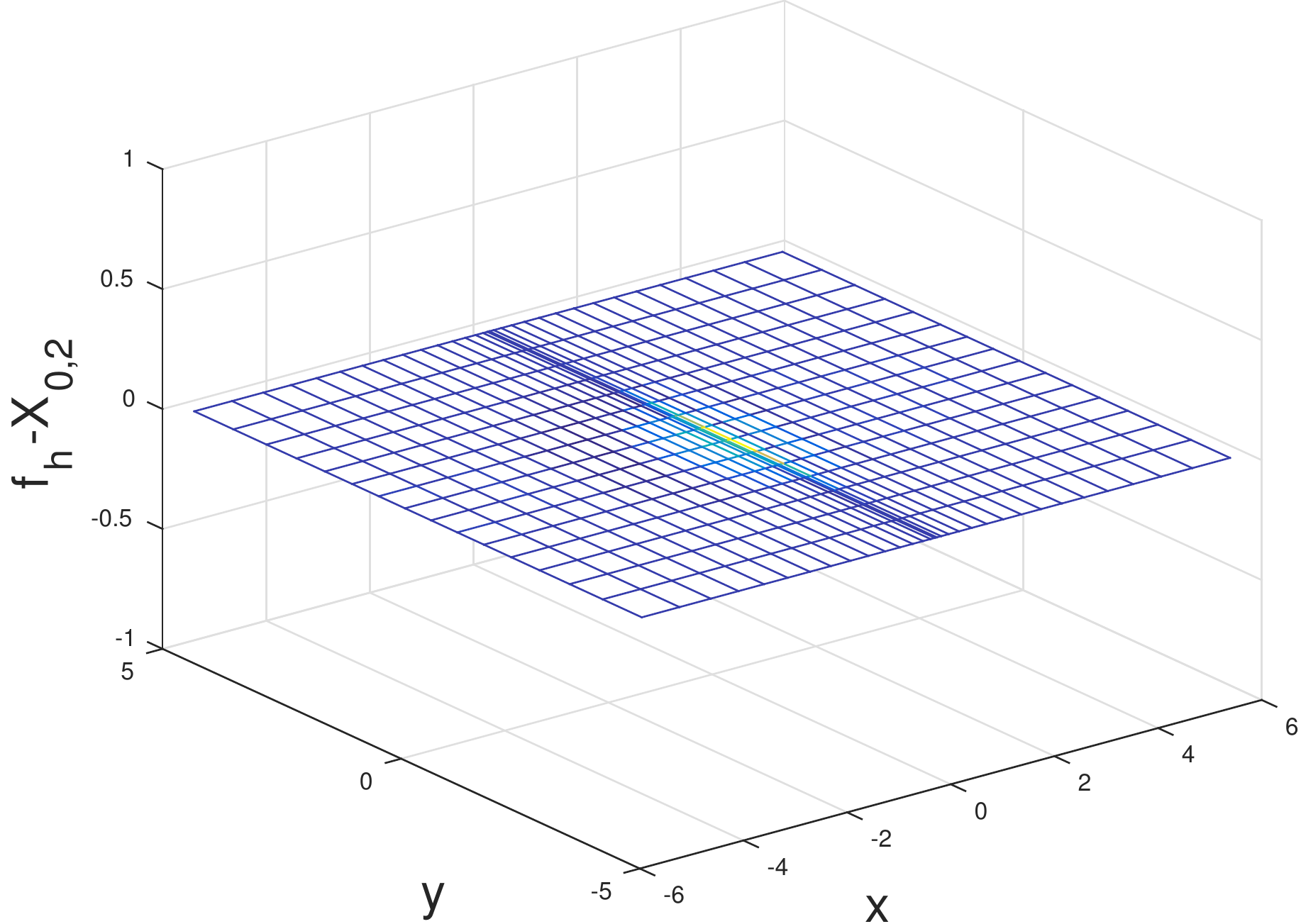}\\
\includegraphics[height = 0.2\textheight,width = 0.3\textwidth]{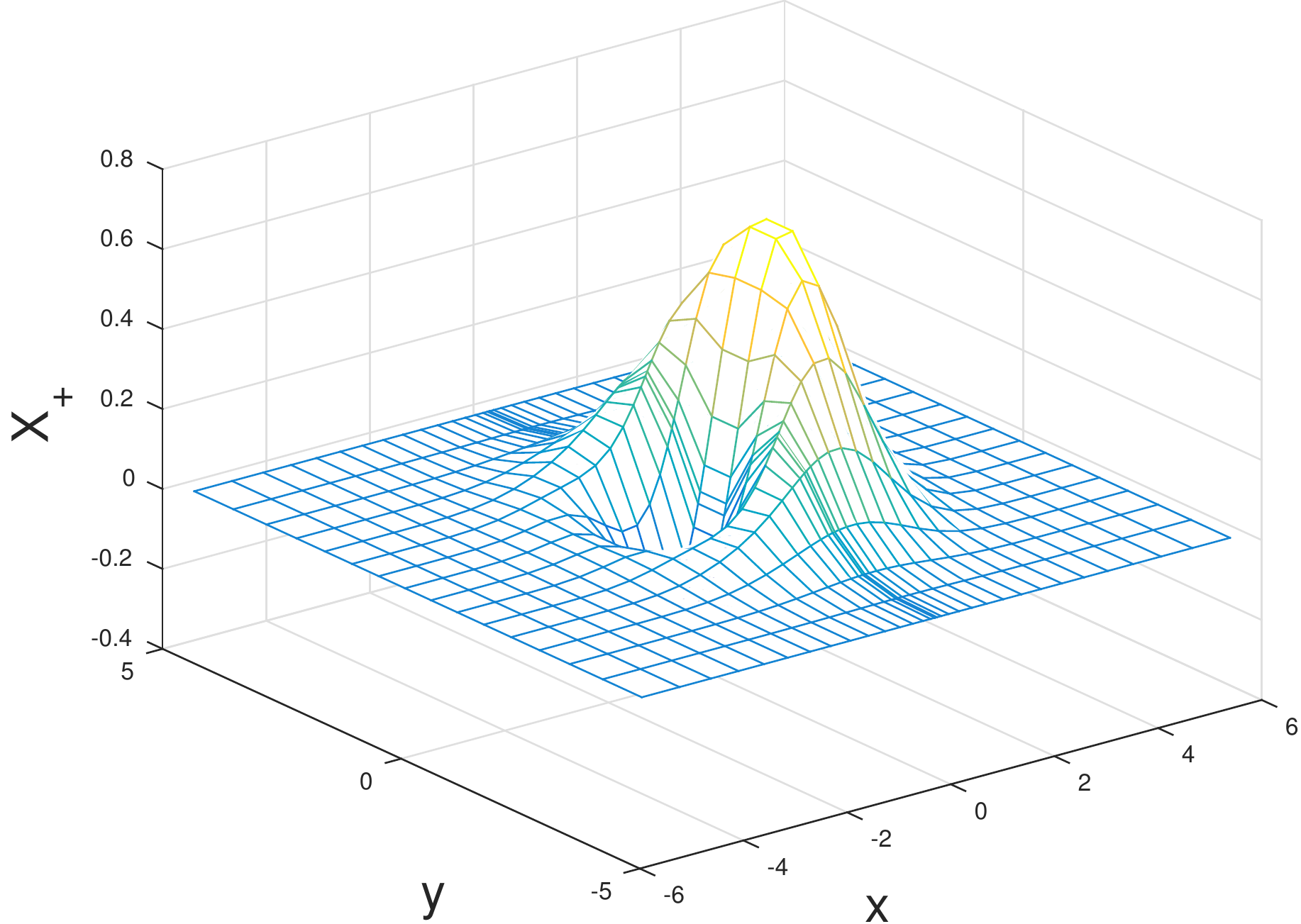}
\includegraphics[height = 0.2\textheight,width = 0.3\textwidth]{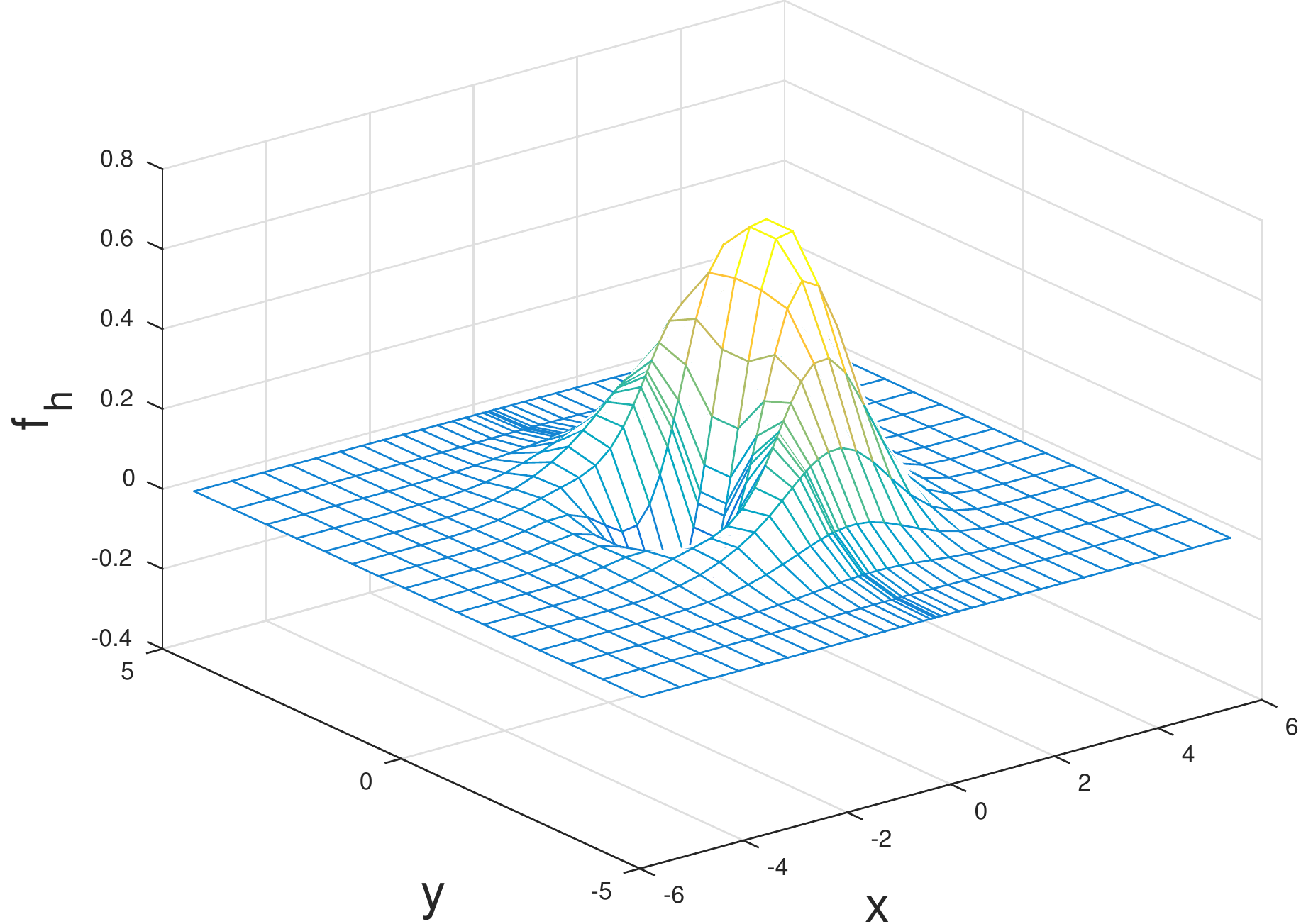}
\includegraphics[height = 0.2\textheight,width = 0.3\textwidth]{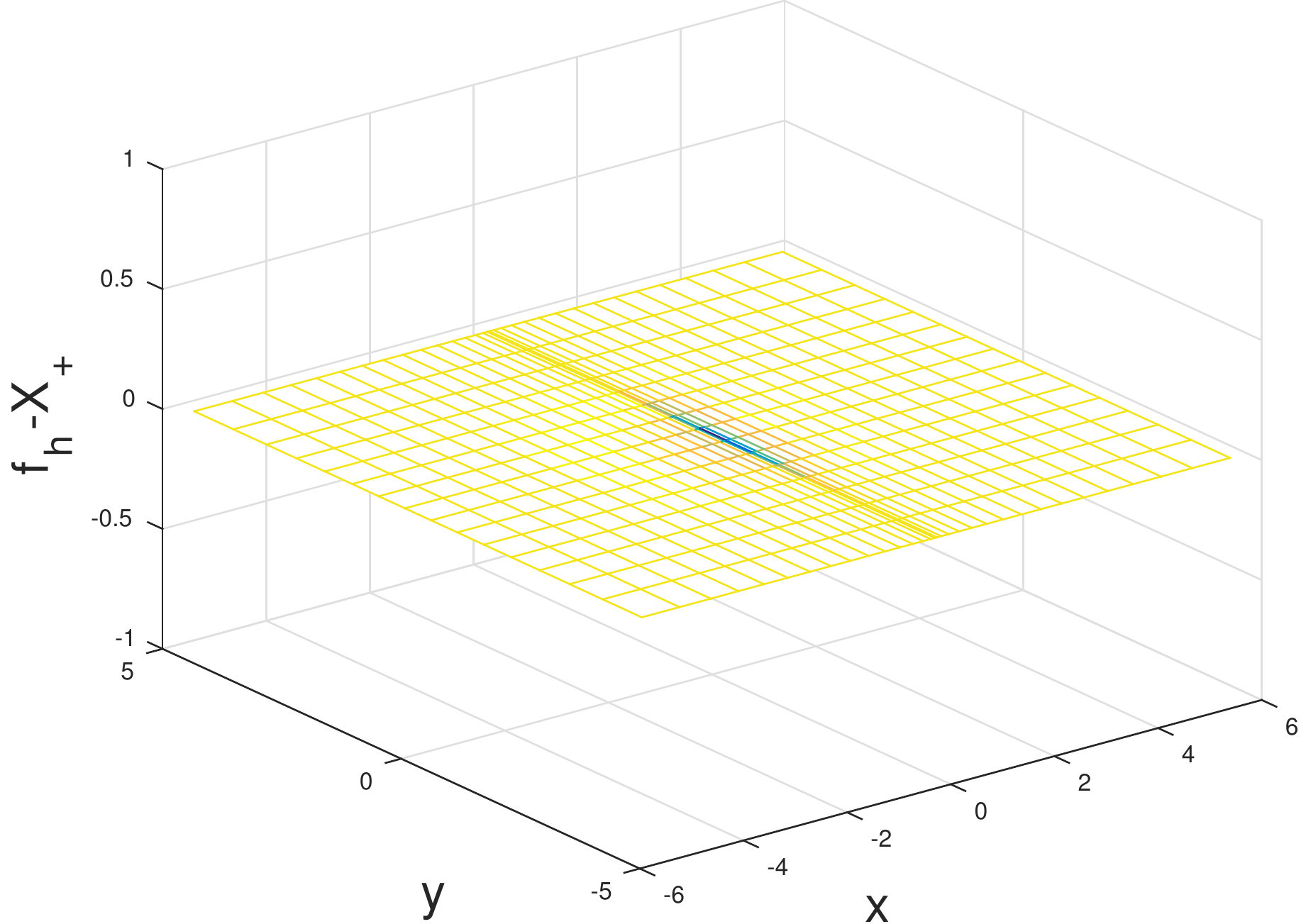}\\
\includegraphics[height = 0.2\textheight,width = 0.3\textwidth]{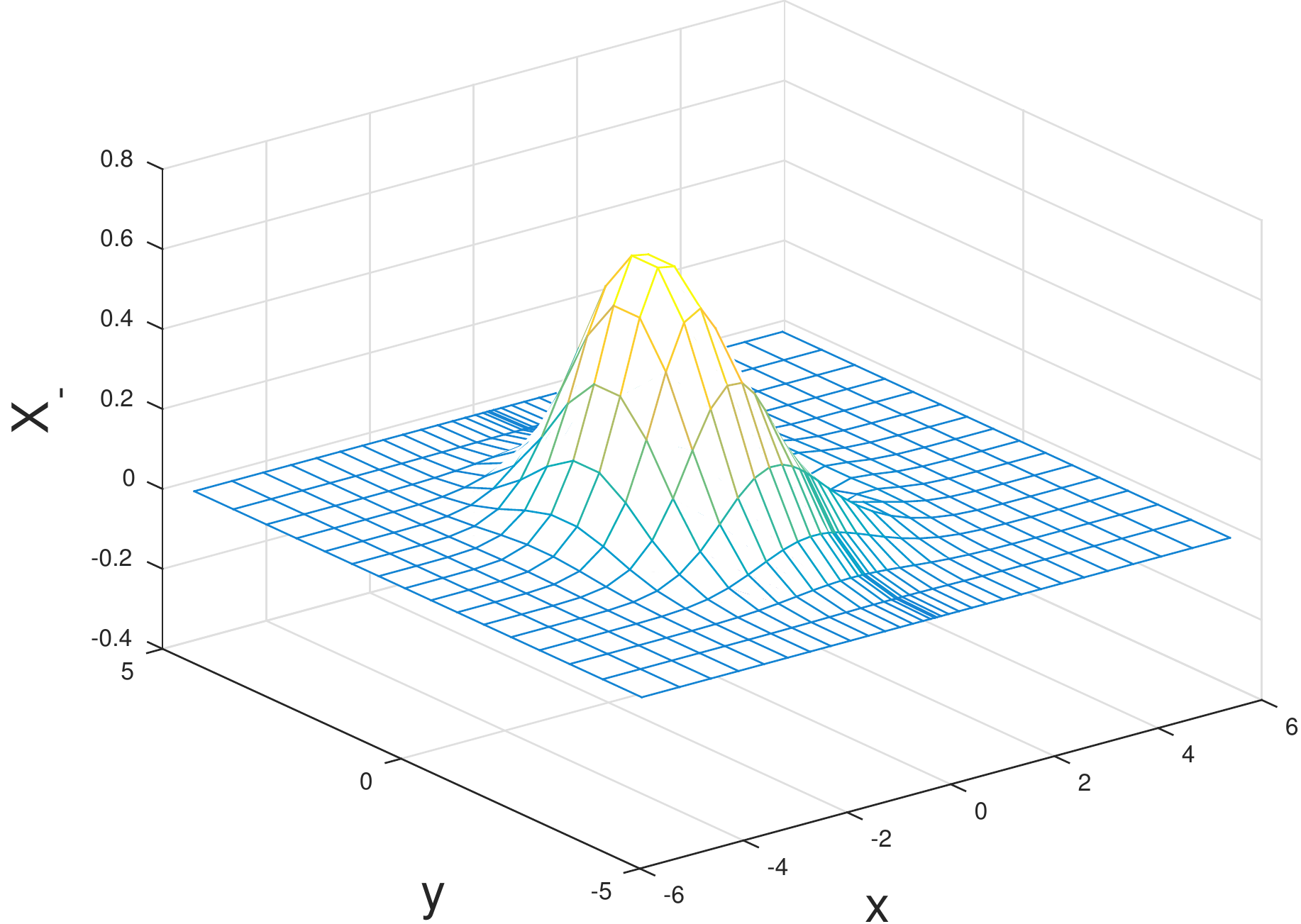}
\includegraphics[height = 0.2\textheight,width = 0.3\textwidth]{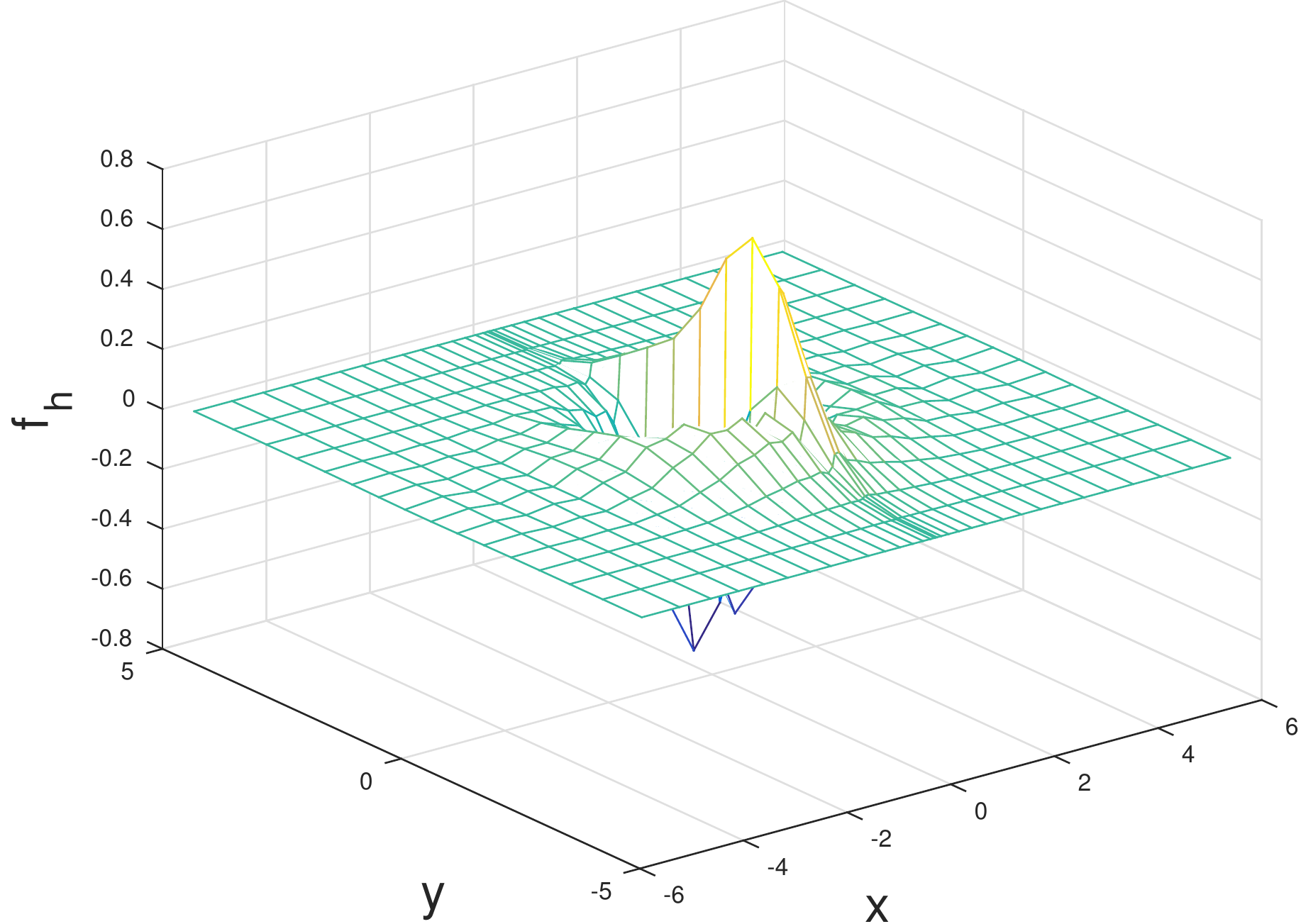}
\includegraphics[height = 0.2\textheight,width = 0.3\textwidth]{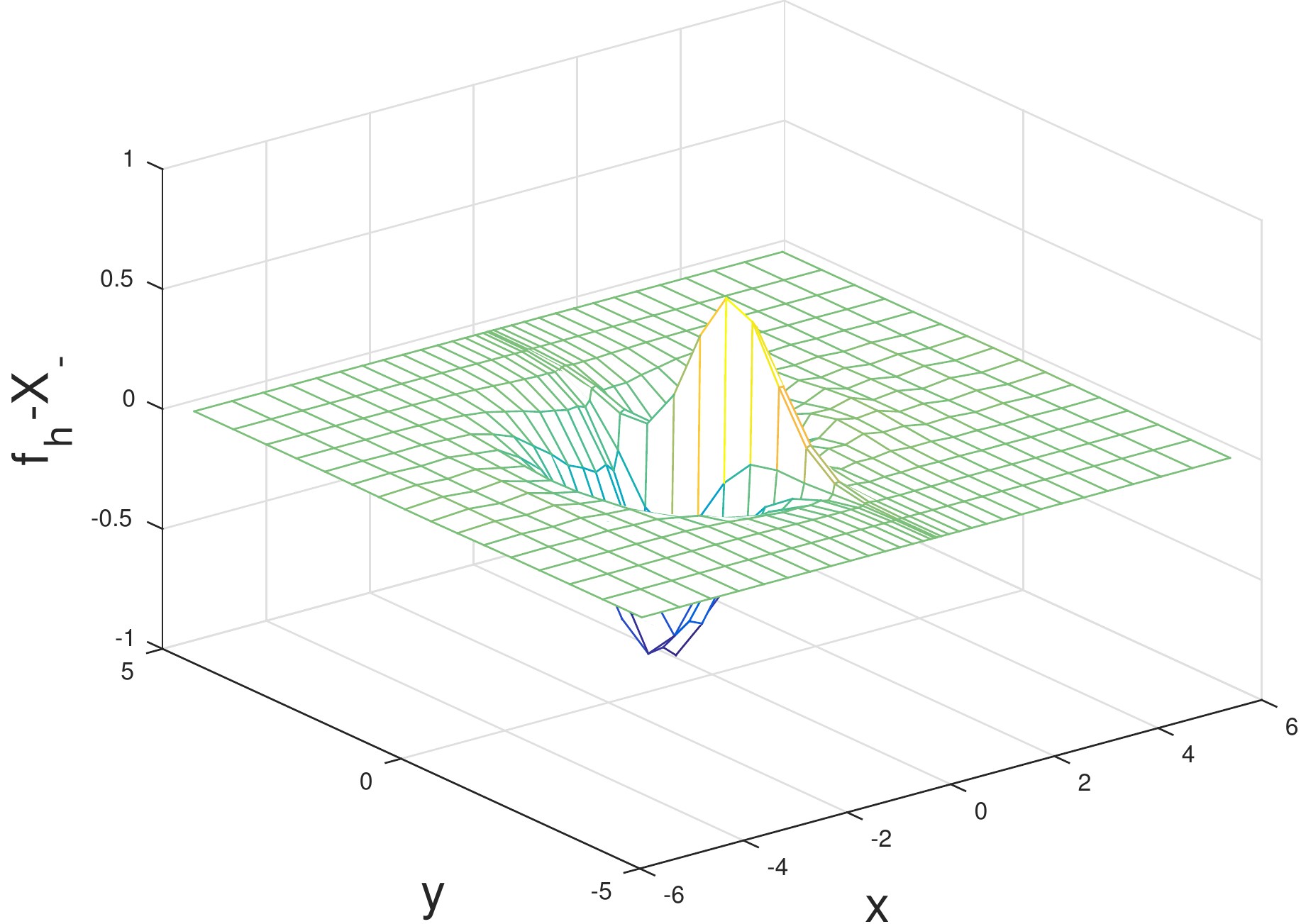}
\caption{(Example 4.2.1) These four rows of figures demonstrate the results computed using $h = X_{0,1}$, $h = X_{0,2}$, $h = X_+$ and $h = X_-$, and the three columns show $h$, recovered solution $f_h$ and the difference $h-f_h$ respectively. For the first three cases, the solutions satisfy that $f_h = h$ since $h \in H^0 \oplus H^+$. 
Solution to the last case does not satisfy $f_h = h$ since $h \in H^-$. In all examples, we use 31 basis functions along each direction.}
\label{fig:MD_Dirichlet_recover}
\end{figure}

\noindent \emph{Example 4.2.2 Maxwell boundary condition.}
In the second example show in Figure~\ref{fig:MD_Maxwell_recover}, we set the accommodation coefficients to be $(\alpha_2, \alpha_3) = (0.3,0.4)$. This time if $h$ is chosen such that
\begin{align*}
    h = X_{0,k} - \CalK \vpran{X_{0, k} \big|_{\mu<0}} \quad (k = 1, 2) 
\qquad \text{or} \qquad
    h = X_{+} - \CalK \vpran{X_{+} \big|_{\mu<0}}  \,,
\end{align*}
then the solution is $f_h = X_{0, k}$ or $X_+$, which again is consistent with the theory. 

\begin{figure}
\centering
\includegraphics[height = 0.18\textheight,width = 0.23\textwidth]{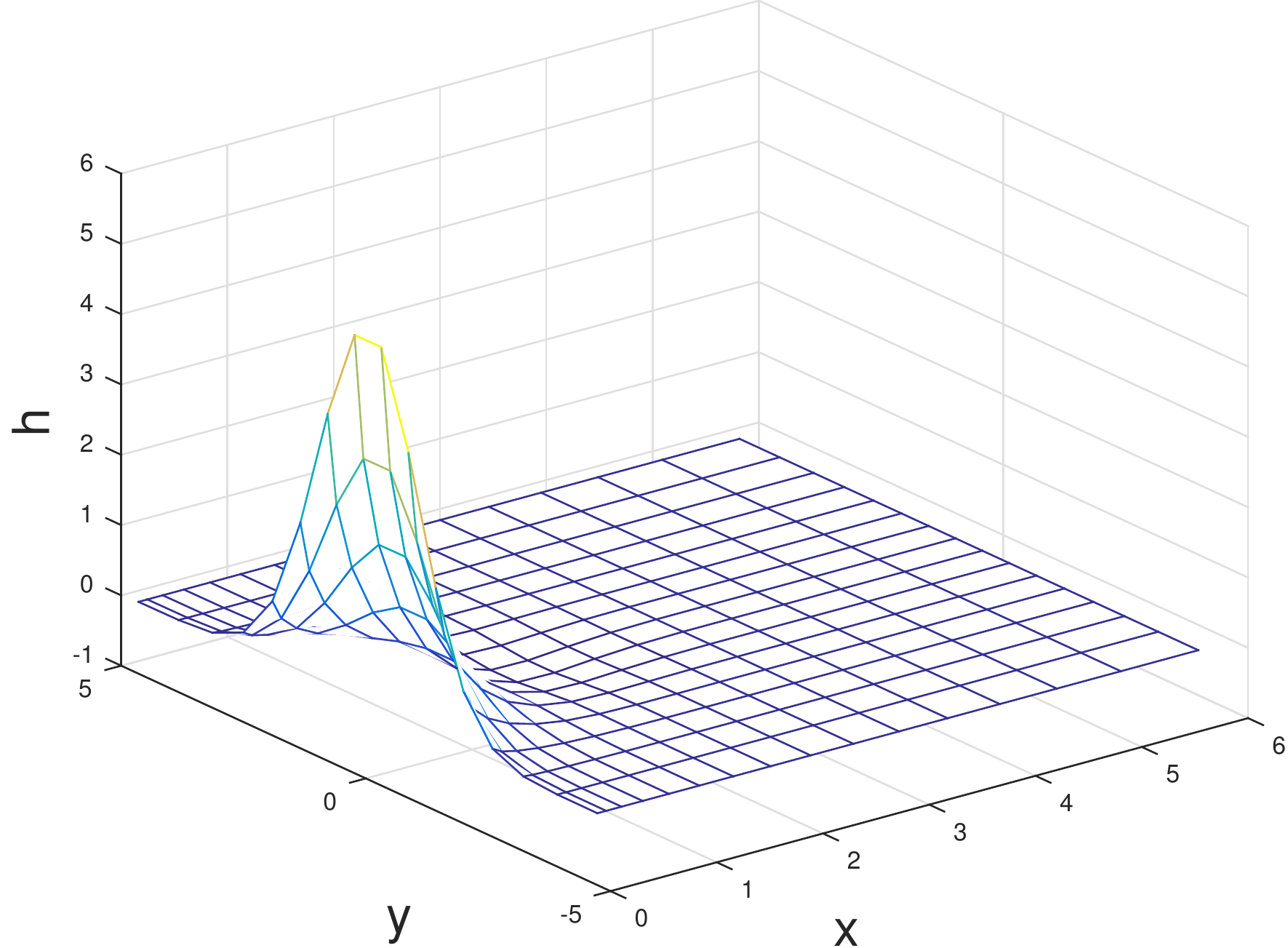}
\includegraphics[height = 0.18\textheight,width = 0.23\textwidth]{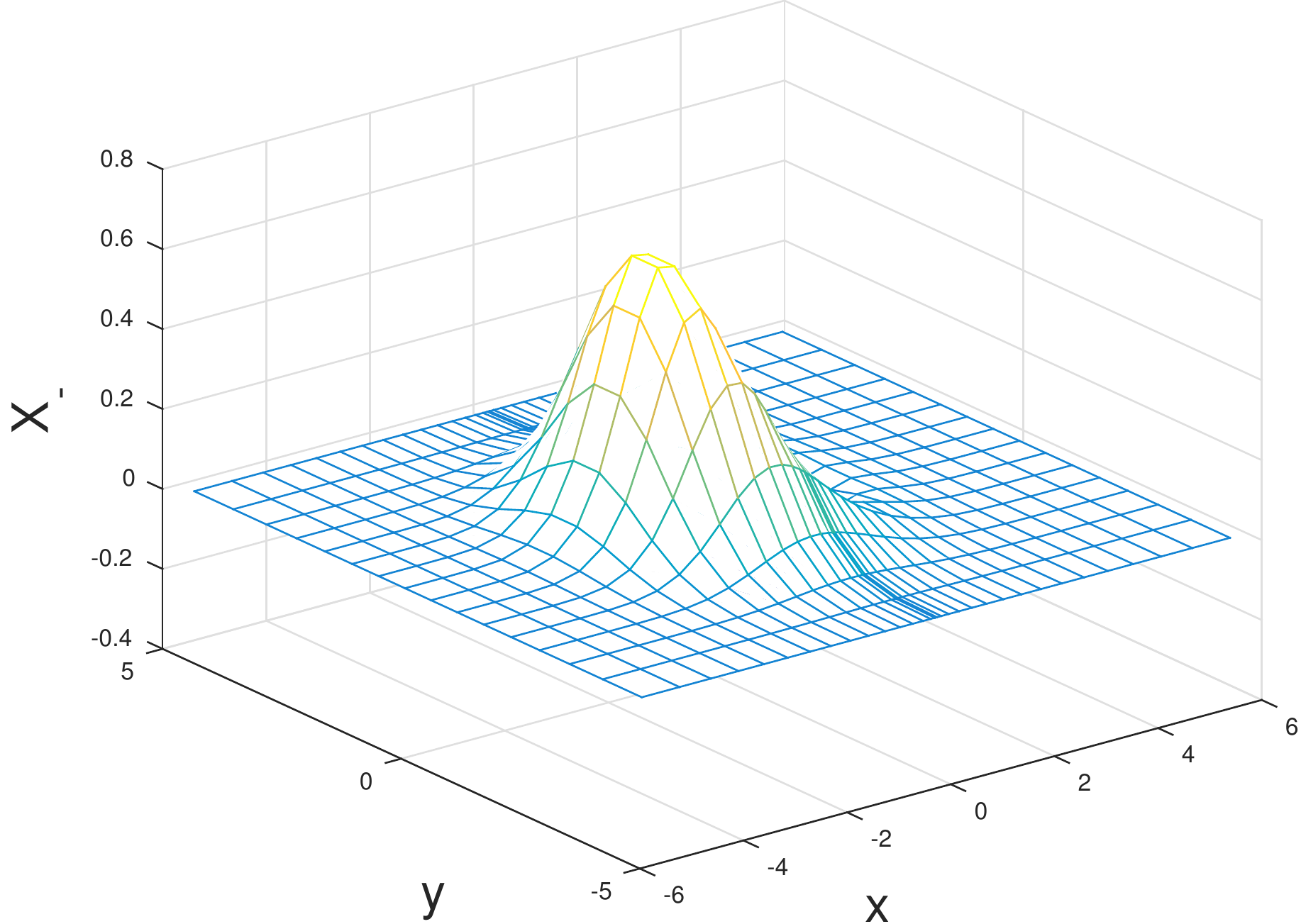}
\includegraphics[height = 0.18\textheight,width = 0.23\textwidth]{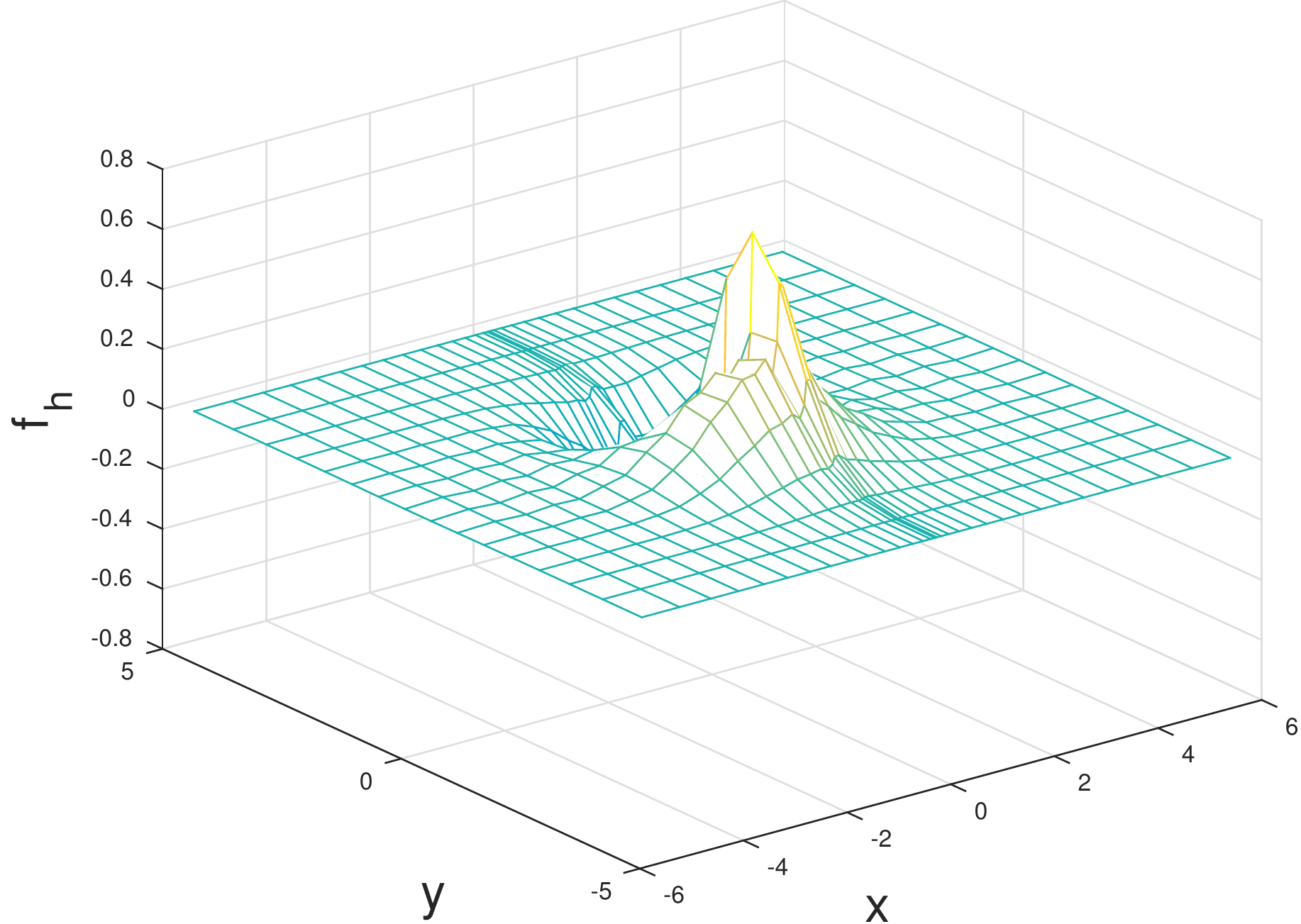}
\includegraphics[height = 0.18\textheight,width = 0.23\textwidth]{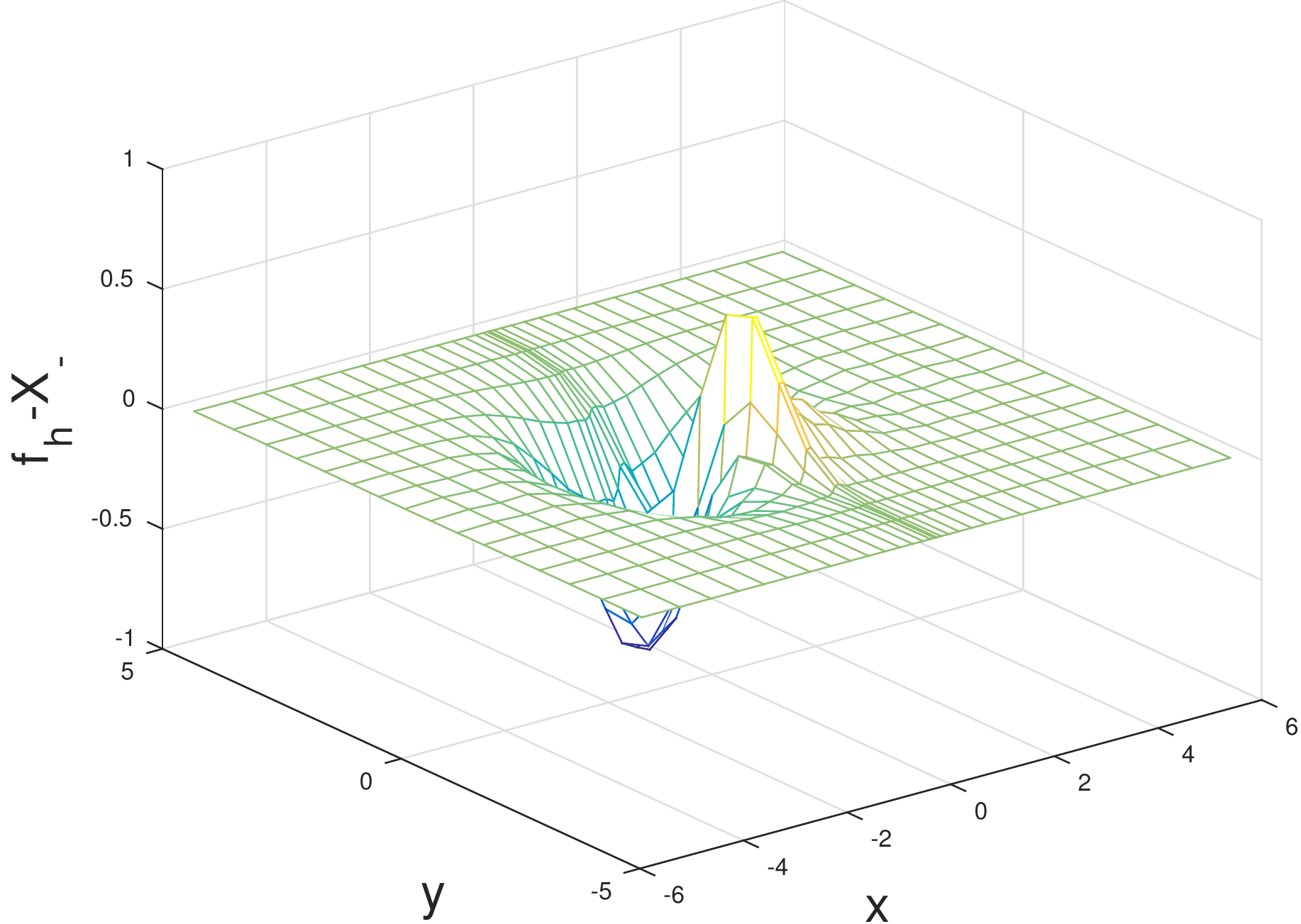}\\
\includegraphics[height = 0.18\textheight,width = 0.23\textwidth]{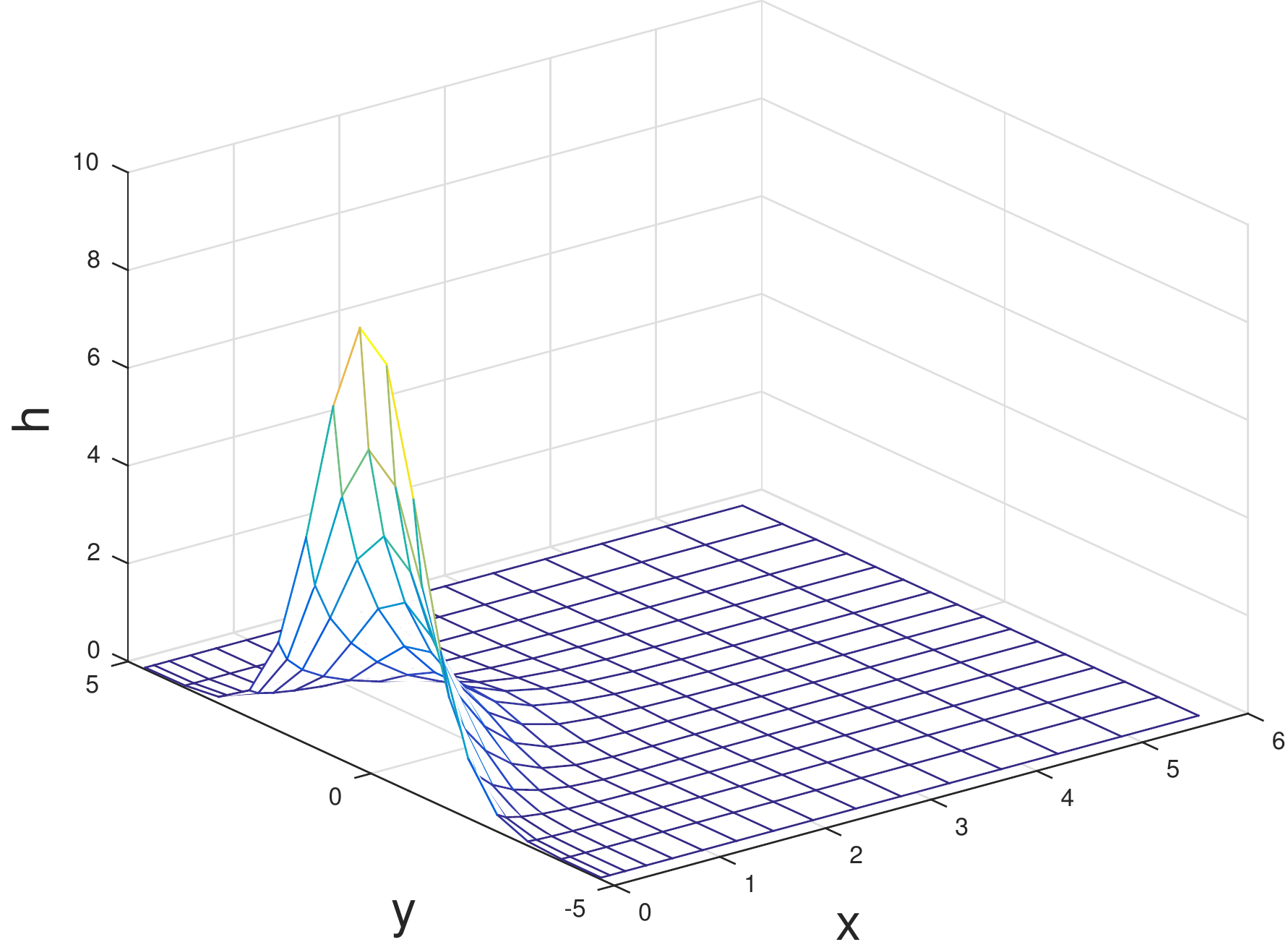}
\includegraphics[height = 0.18\textheight,width = 0.23\textwidth]{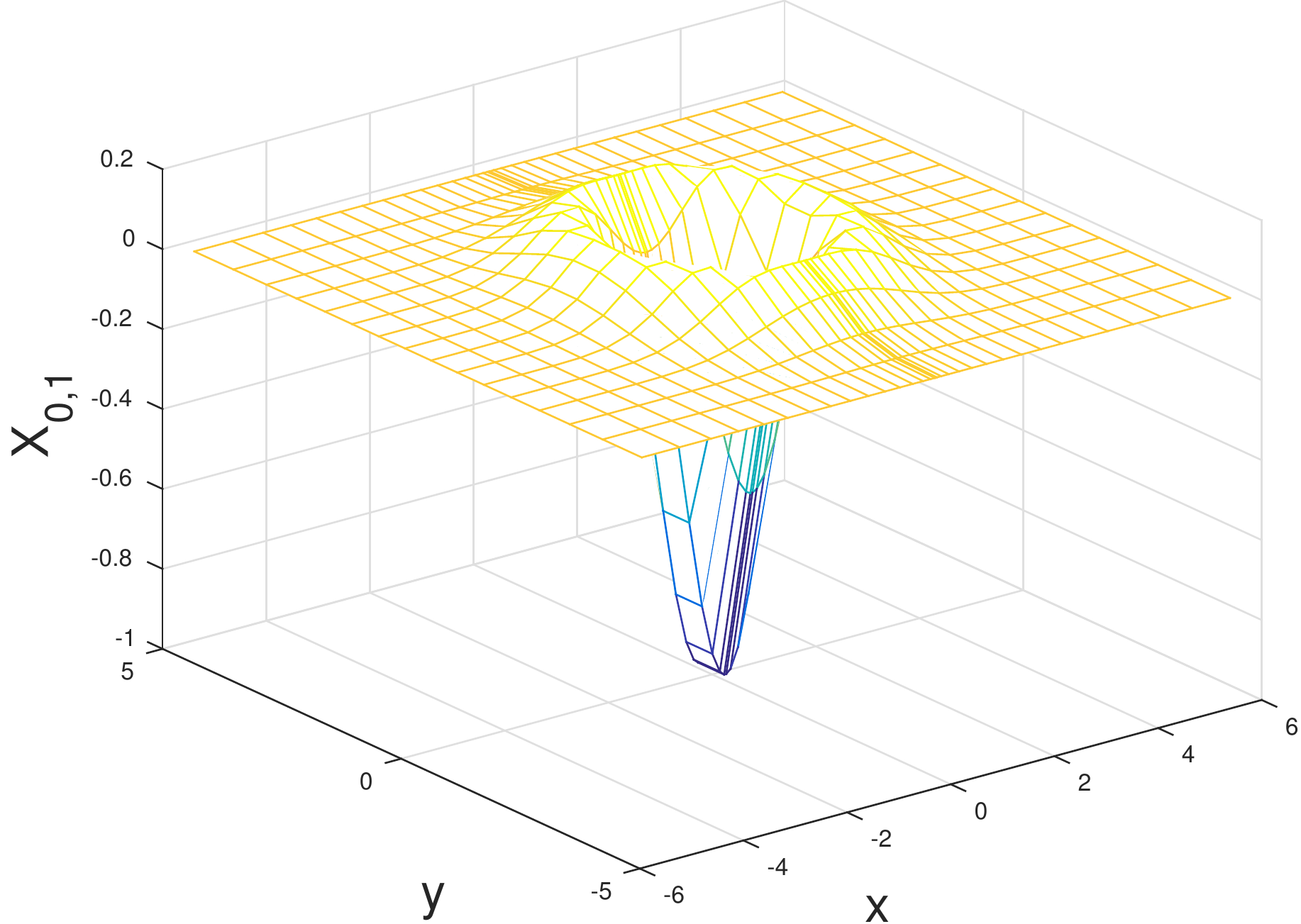}
\includegraphics[height = 0.18\textheight,width = 0.23\textwidth]{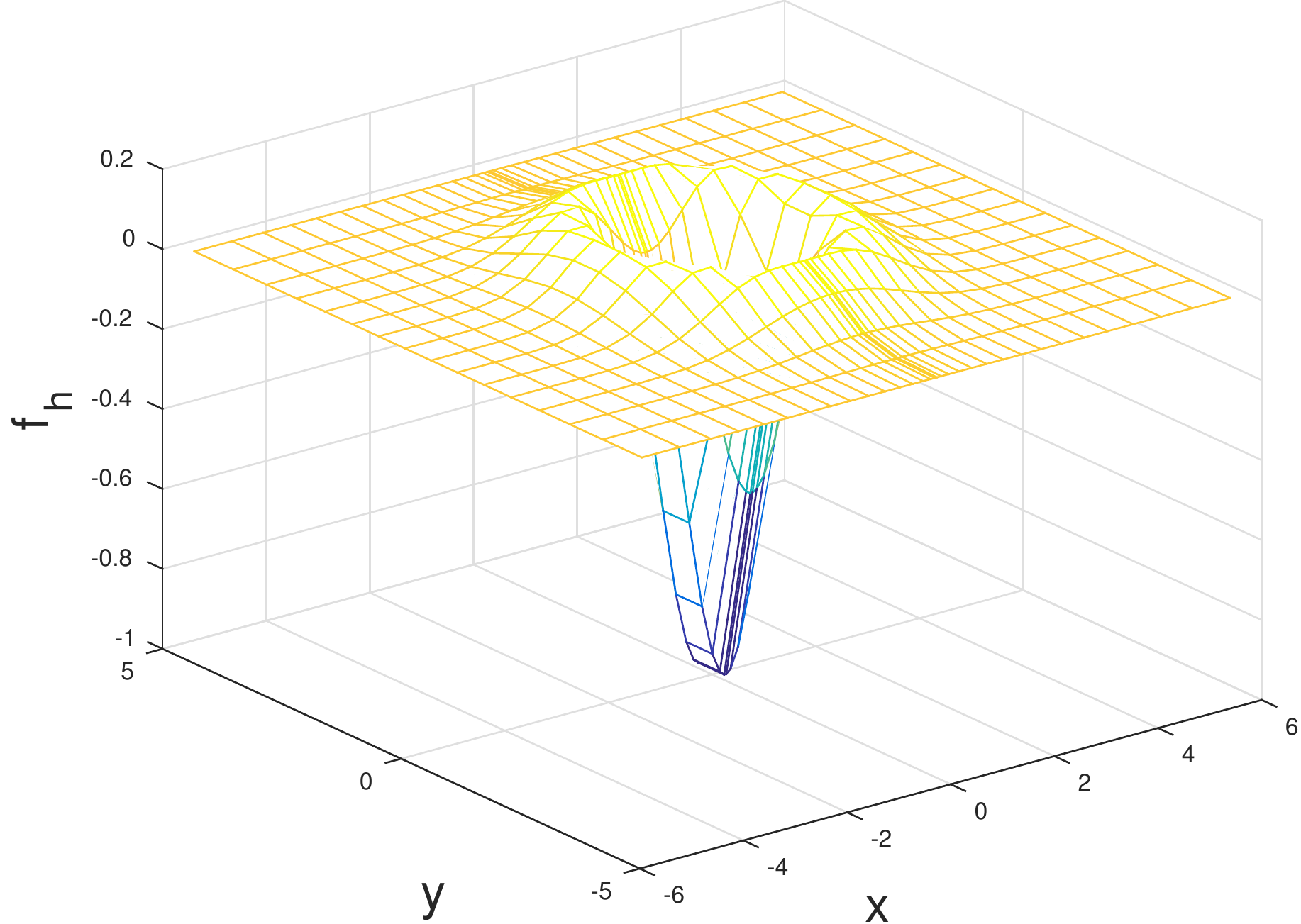}
\includegraphics[height = 0.18\textheight,width = 0.23\textwidth]{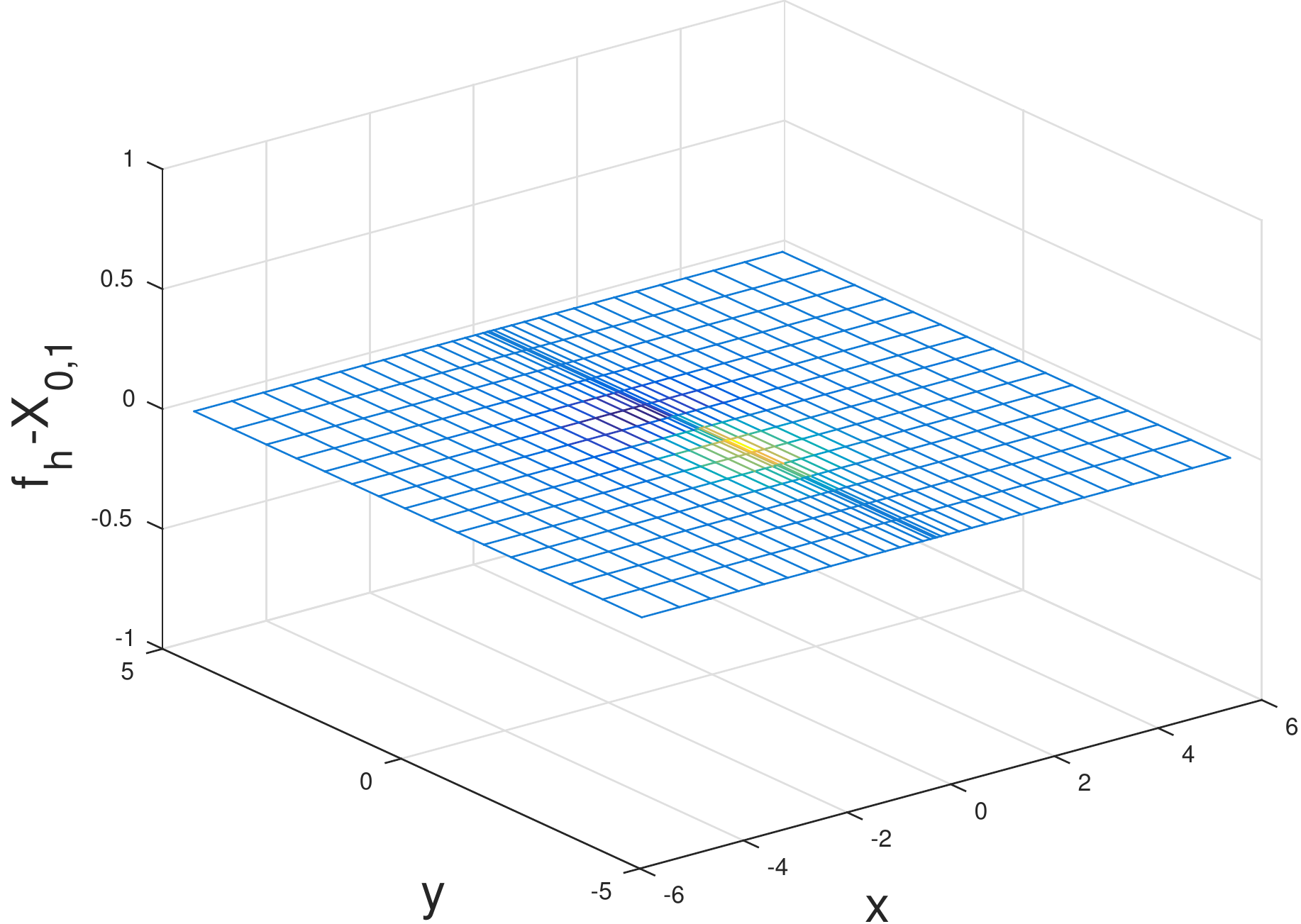}\\
\includegraphics[height = 0.18\textheight,width = 0.23\textwidth]{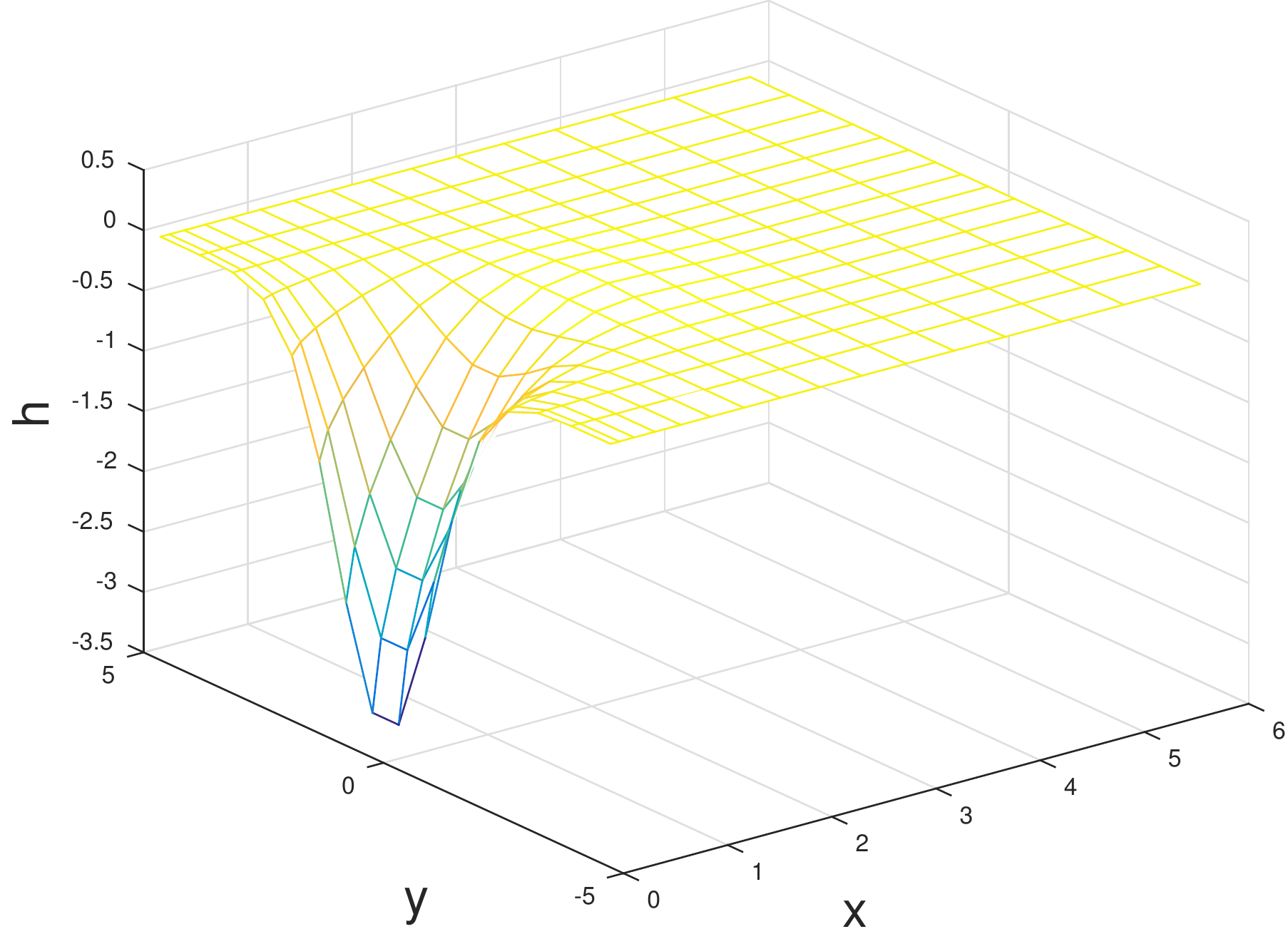}
\includegraphics[height = 0.18\textheight,width = 0.23\textwidth]{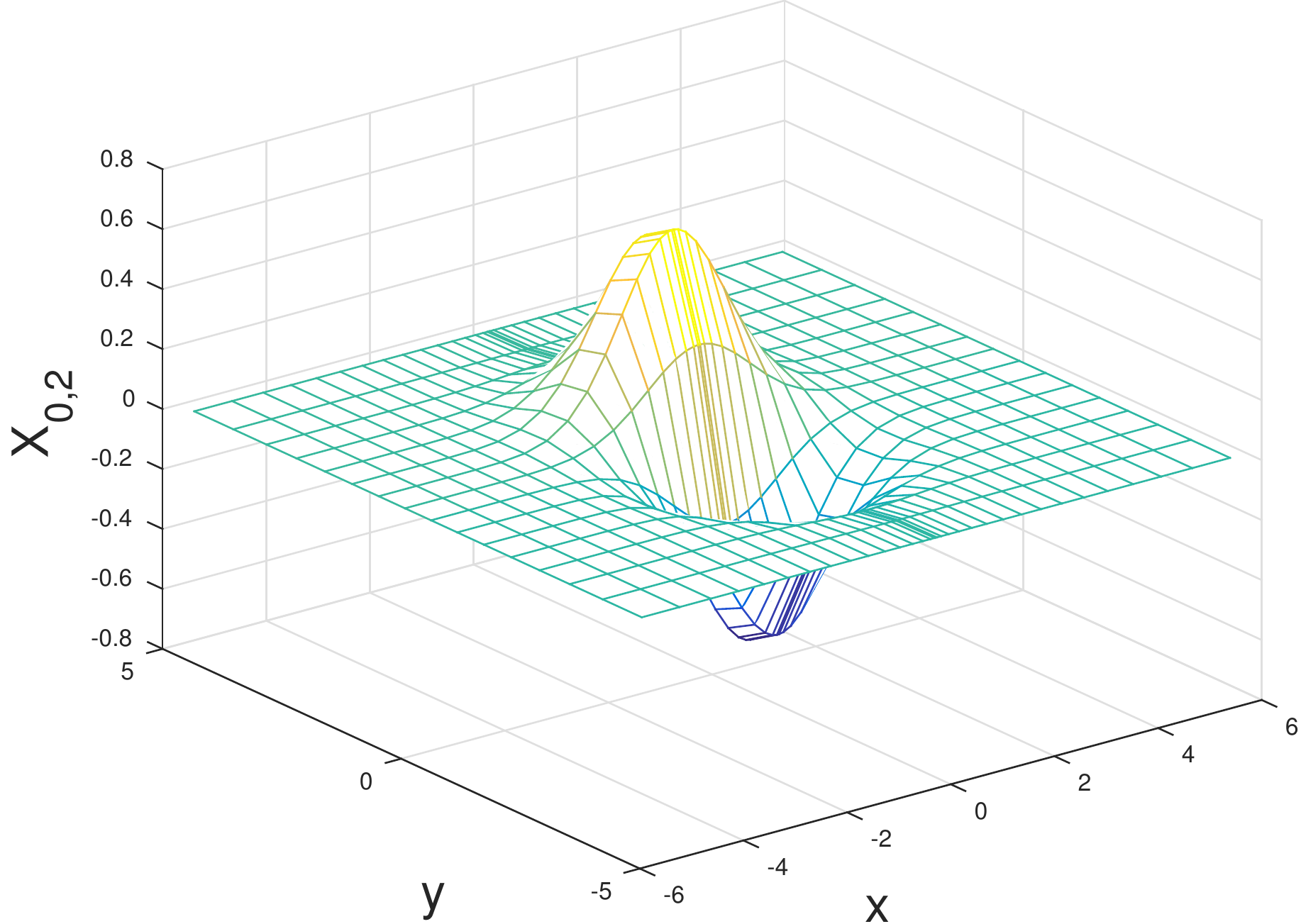}
\includegraphics[height = 0.18\textheight,width = 0.23\textwidth]{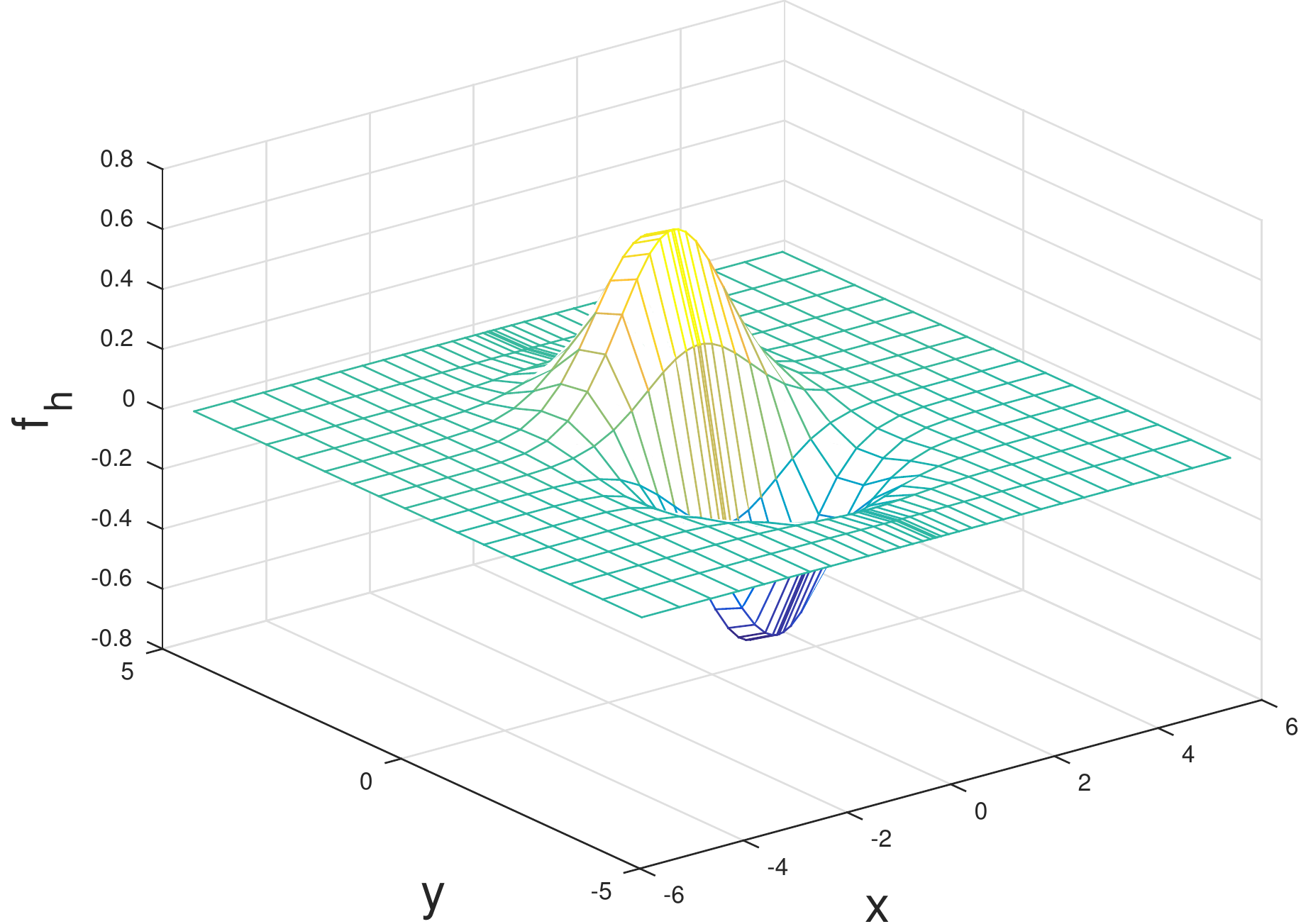}
\includegraphics[height = 0.18\textheight,width = 0.23\textwidth]{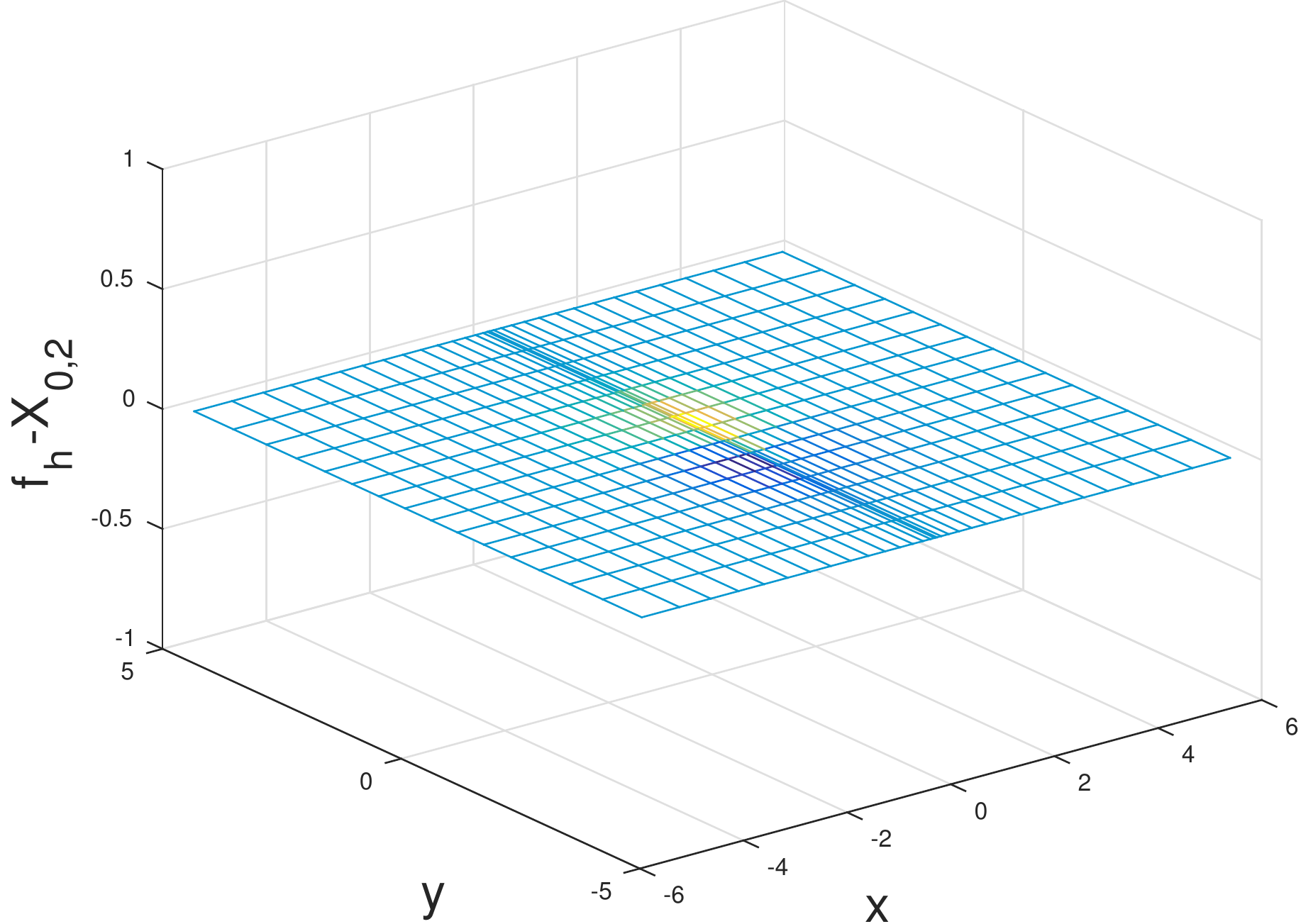}\\
\includegraphics[height = 0.18\textheight,width = 0.23\textwidth]{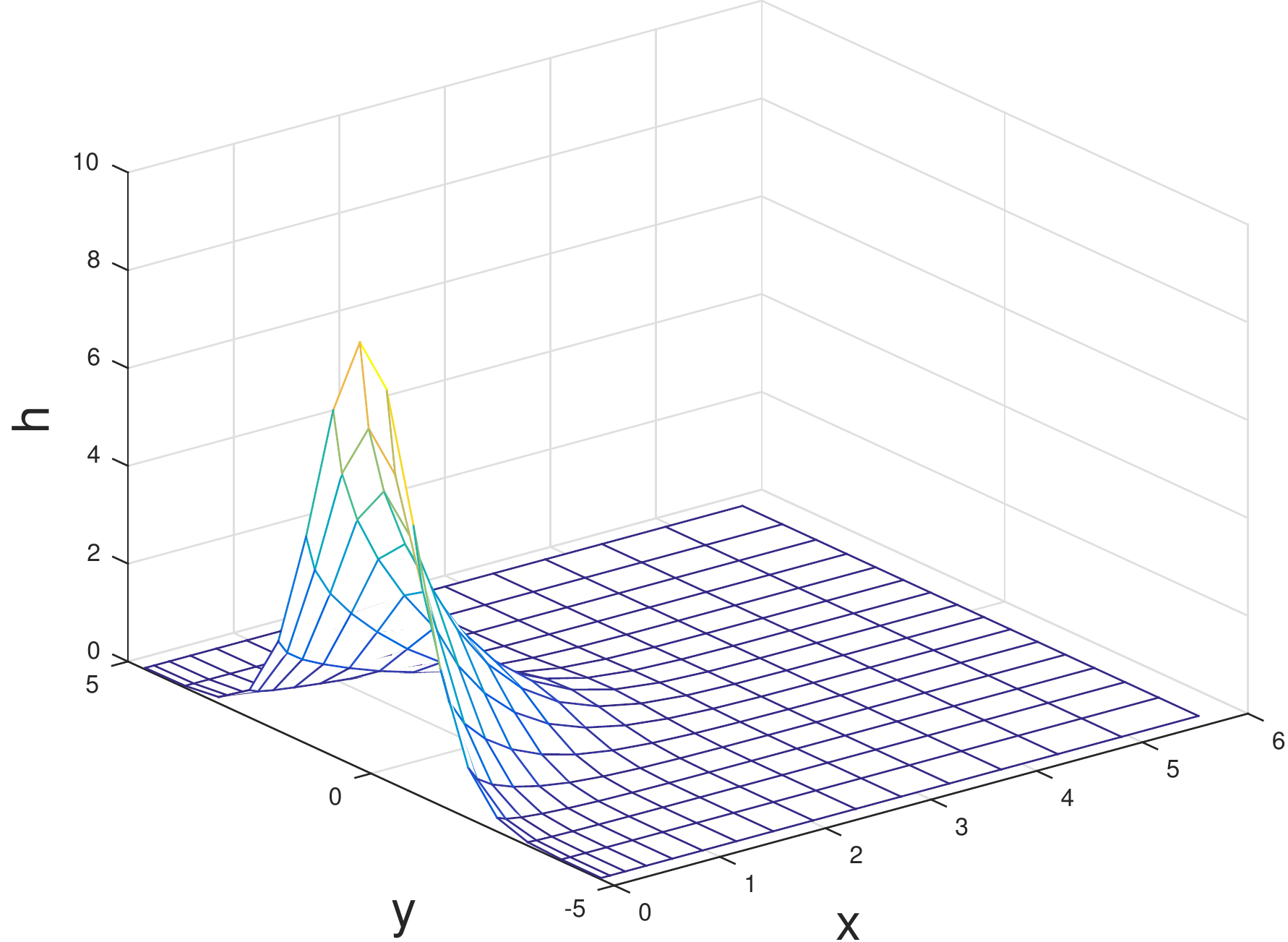}
\includegraphics[height = 0.18\textheight,width = 0.23\textwidth]{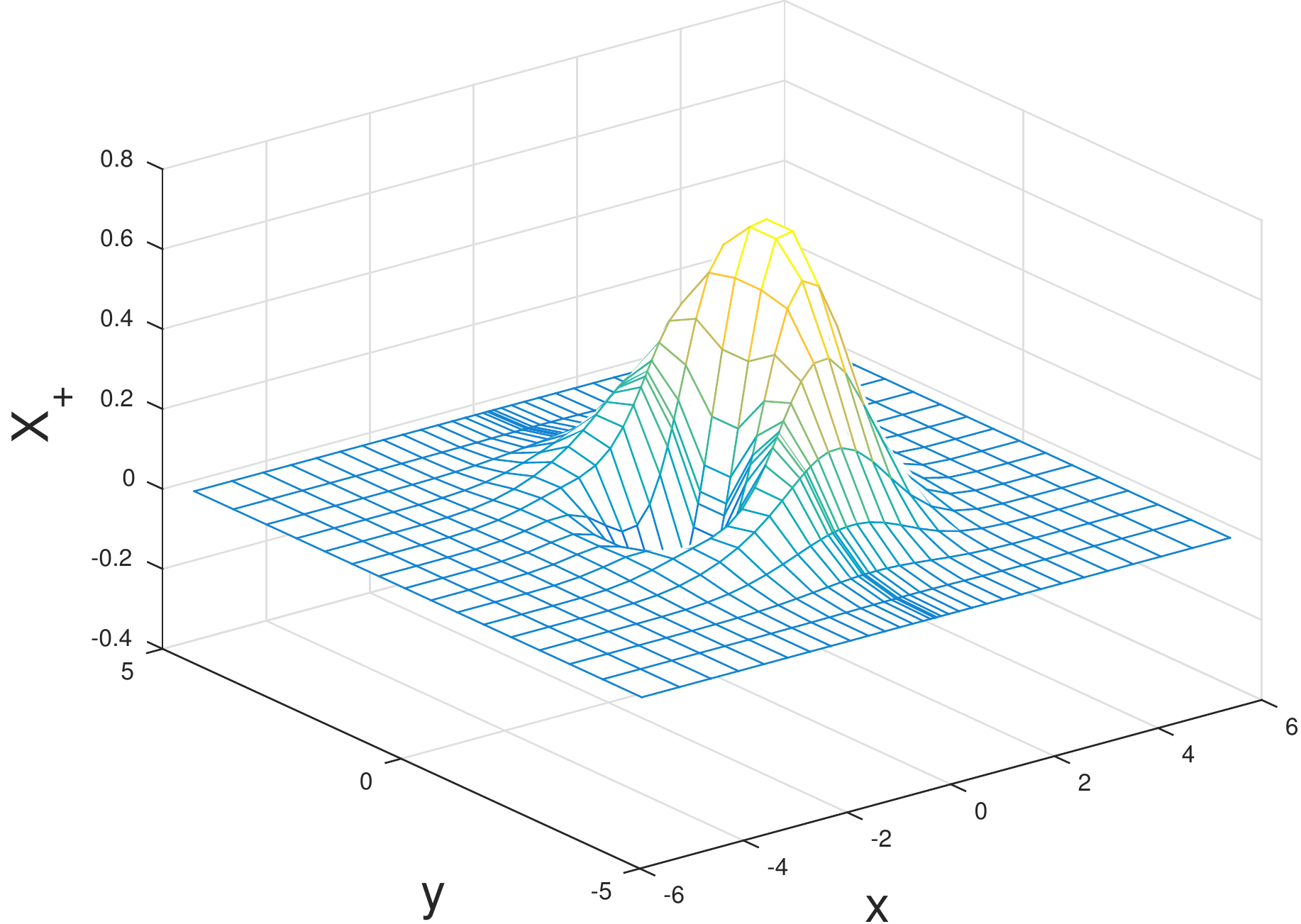}
\includegraphics[height = 0.18\textheight,width = 0.23\textwidth]{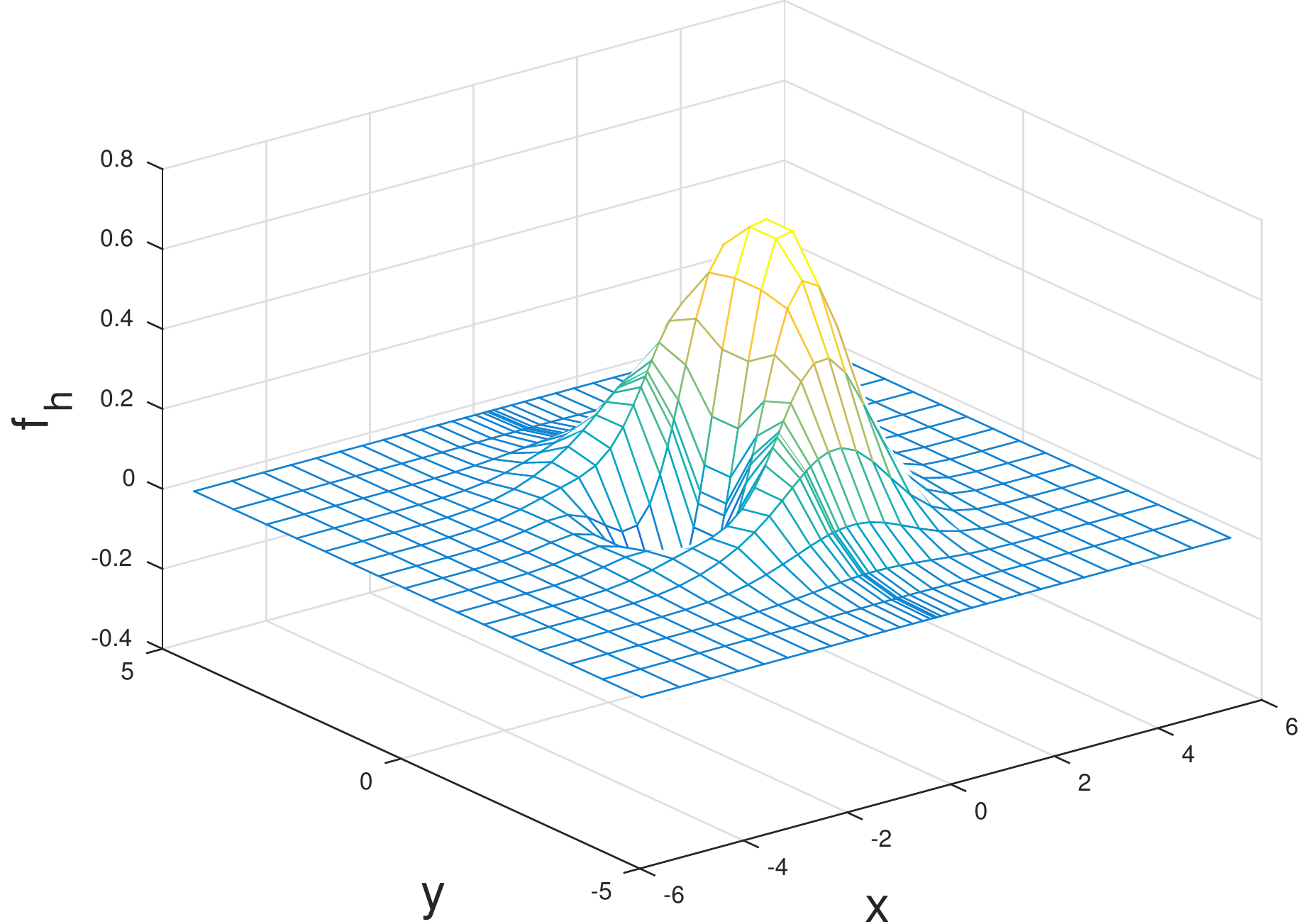}
\includegraphics[height = 0.18\textheight,width = 0.23\textwidth]{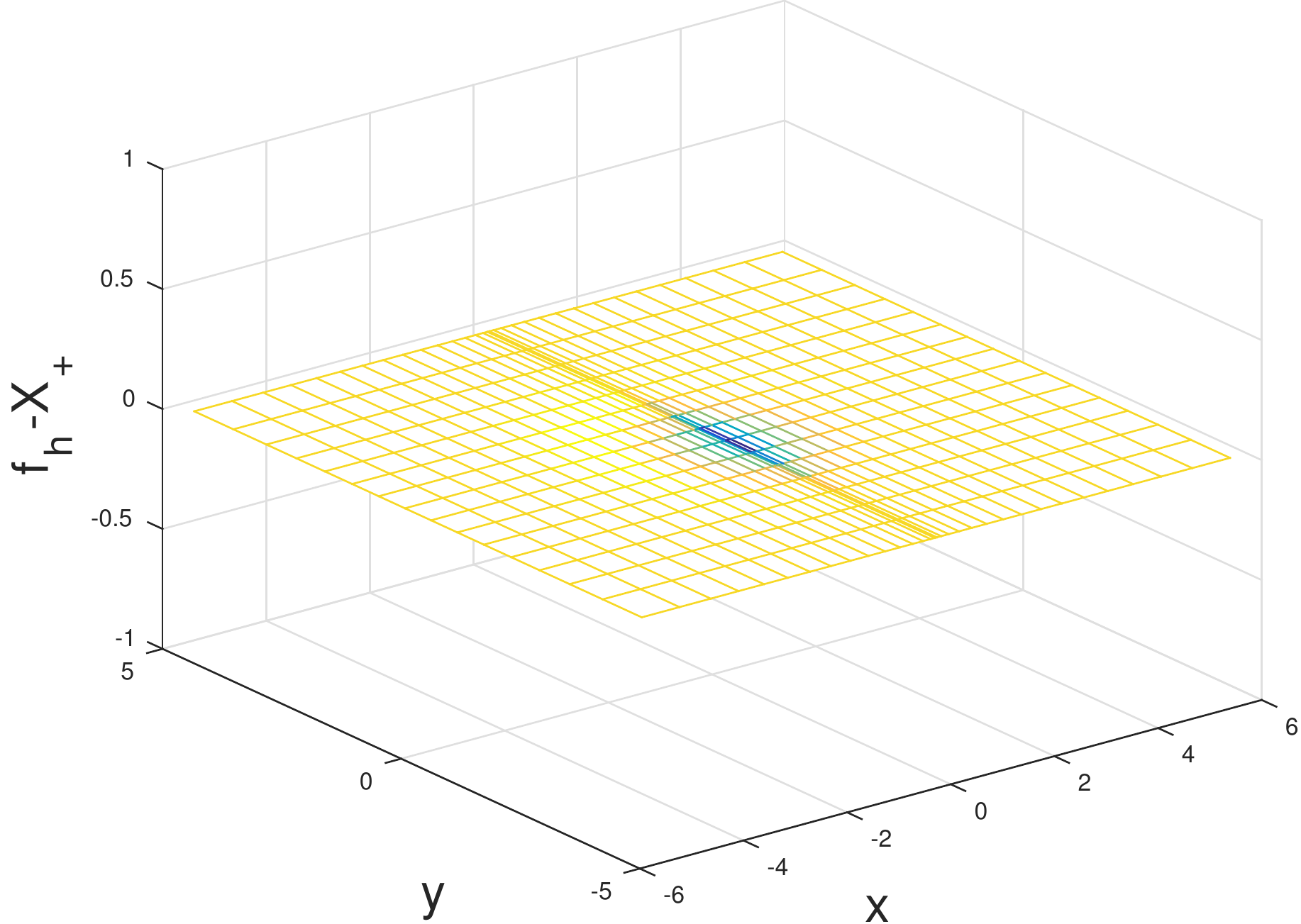}\\
\caption{(Example 4.2.2) These four rows show numerical results using boundary conditions provided by $h = X - \CalK(X \big|_{\mu<0})$ where $X = X_-, X_{0,1}, X_{0,2}, X_+$ respectively. The accommodation coefficients are set as $(\alpha_d, \alpha_s) = (0.3,0.4)$. The four columns (from left to right) are: $h$, $X$, recovered result $f_h$, and the difference $f_h - X$. In the bottom three cases, one recovers $X_{0, 1}, X_{0, 2}, X_+$ as the solutions as expected. In all examples, we use 31 basis functions along each direction.}
\label{fig:MD_Maxwell_recover}
\end{figure}

\subsection{Multi-frequency linearized transport equation} 

\subsubsection{\underline{Formulation}} 
In this third example, we consider a linearized BGK-type of equation
that models phonons with a continuous range of frequencies  \cite{MCMY2011}. We consider the time-independent half-space equation. Let $\omega \in (0,
\omega_m)$ be the angular frequency of phonons and $F$ be the density
function such that $F = F(x, \mu, \omega)$.

The stationary nonlinear equation has the form
\begin{align} \label{eq:phonon-1}
   \mu v(\omega) \del_x F = - \frac{F - F_{BE}}{\tau(\omega)} \,,
\end{align}
where $v(\omega)$ is the group velocity and $\tau(\omega)$ is the relaxation time. They are both frequency-dependent. We assume that $v(\omega) > 0$, although it may not have a positive lower bound. The equilibrium state $F_{BE}$ is given as the Bose-Einstein distribution function such that
\begin{align} \label{def:equil-BE}
   F_{BE}(\omega, T) = \frac{1}{e^{\hbar\omega/(k_B T)} - 1} \,,
\end{align}
where $\hbar$ is the reduced Planck constant, $k_B$ is the Boltzmann constant, and $T$ is the temperature.  Given a reference temperature $T_0$, we linearized $F(\omega, T)$ around $F(\omega, T_0)$ such that
\begin{align}\label{eq:linearF}
    F(\omega, T) = F(\omega, T_0) + G(\omega, T) \,.
\end{align}
The resulting equation has the form
\begin{align} \label{eq:phonon-2}
    \mu v(\omega) \del_x G = - \frac{G - C_\omega \Delta T}{\tau(\omega)} \,,
\end{align}
where $C_\omega = \frac{e^{\hbar\omega/(k_B T_0)}}{(e^{\hbar\omega/(k_B T_0)} - 1)^2} \frac{\hbar\omega}{k_BT_0^2}$. 
By the mass conservation, we have
\begin{align*}
    \int_0^{\omega_m} \int_{-1}^1 \frac{G}{\tau(\omega)} \dmu \domega
 = \vpran{\int_0^{\omega_m} \int_{-1}^1 \frac{C_\omega}{\tau(\omega)} \dmu\domega}
     \Delta T \,,
\end{align*}
which gives
\begin{align*}
  \Delta T
  = \frac{1}{\Theta_0} \int_0^{\omega_m} \int_{-1}^1 \frac{G}{\tau(\omega)} \dmu \domega
\qquad \text{with} \quad
   \Theta_0 = \int_0^{\omega_m} \int_{-1}^1 \frac{C_\omega}{\tau(\omega)} \dmu\domega \,.
\end{align*}
Equation~\eqref{eq:phonon-2} then has the form
\begin{align} \label{eq:phonon-3}
    \mu v(\omega) \del_x G 
= - \frac{1}{\tau(\omega)}
   \vpran{G - \frac{C_\omega}{\Theta_0}
                     \int_0^{\omega_m} \int_{-1}^1 \frac{G}{\tau(\omega)} \dmu \domega} \,.
\end{align}
The operator on the right-hand side of~\eqref{eq:phonon-3} is not self-adjoint in the space $L^2(\dmu \domega)$. However, we show below that it is symmetrizable. Indeed, if we define 
\begin{align} \label{linearization-2}
   \beta_\omega = v(\omega) \tau(\omega) \,,
\qquad
   G =  \frac{C_\omega}{\Theta_0\sqrt{\beta_\omega}} \ f \,,
\end{align}
then $f$ satisfies
\begin{align} \label{eq:phonon-4}
    \mu \del_x f 
    = - \frac{1}{\beta_\omega} f
       + \frac{1}{\sqrt{\beta_\omega}}\int_0^{\omega_m} 
          \int_{-1}^1 \frac{f}{\sqrt{\beta_\omega}} \dsigma
    \,\Denote\,
      - \frac{1}{\sqrt{\beta_\omega}} \vpran{\frac{f}{\sqrt{\beta_\omega}}
       - \vint{\frac{f}{\sqrt{\beta_\omega}}}_{\mu, \omega}}
   \Denote - \CalL(f) \,,
\end{align}
where $\dsigma = \frac{C_\omega}{\Theta_0 \tau(\omega)} \dmu \domega$ is a probability measure by the definition of $\Theta_0$. Note that the linear operator $\CalL$ on the right-hand side of~\eqref{eq:phonon-4} is now self-adjoint.

In the previous two examples, both linear operators $\CalL$ satisfy the classical ``coercivity'' condition~\eqref{property:coer-L}.
However, this property ceases to hold in the current example. 
Nevertheless, we show that with the help of the added damping terms, condition (\PL5) is still true. 
\begin{lem} \label{thm:damped-2}
Suppose $0 < \omega < \omega_m$, $0 < \beta_\omega \leq b_0$, and $f = f(\mu, \omega)$. Let $\CalL$ be the scattering operator given by 
\begin{align} \label{def:L-multi-freq}
    \CalL f = \frac{1}{\beta_\omega} f 
                 - \frac{1}{\sqrt{\beta_\omega}} 
                     \int_0^{\omega_m} \int_{-1}^1 \frac{f}{\sqrt{\beta_\omega}} \dsigma \,,
\end{align}
where $\dsigma = \frac{C_\omega}{\Theta_0 \tau(\omega)} \dmu \domega$ is a probability measure in $(\mu, \omega)$. 
Then 
\begin{itemize}
\item[(a)] $\CalL$ is self-adjoint, nonnegative, and $\NullL = \Span\{\sqrt{\beta_\omega}\}$. Moreover, $\sqrt{\beta_\omega} \in H^0$.

\item[(b)] Denote $X_0 = \frac{\sqrt{\beta_\omega}}{\sqrt{\int_\VV \beta_\omega \dsigma}}$. Define
\begin{align} \label{def:L-d-2}
  \CalL_d f 
 =\CalL f + \alpha \vpran{ 
           \mu X_{0} \viint{\mu X_{0}}{f}_{\mu, \omega}                            
       +   \mu \CalL^{-1}(\mu X_{0})
             \viint{\mu \CalL^{-1}(\mu X_{0})}{f}_{\mu, \omega}} \,.
\end{align}
where $\viint{g_1}{g_2}_{\mu, \omega} = \int_0^{\omega_m} \int_{-1}^1 g_1 g_2 \dsigma$. Then for $\alpha > 0$ small enough, there exists $\sigma_2$ (depending on $\alpha$) such that
\begin{align} \label{bound:L-d-multi-freq}
     \int_{\VV} f \CalL_d f \dsigma
\geq \sigma_2 \norm{f}_{L^2(\dsigma)}^2 \,.
\end{align}
Hence $\CalL$ satisfies all the assumptions (\PL1)-(\PL5).
\end{itemize}
\end{lem}
\begin{proof}
By the definition of $\CalL$, we have
\begin{align} \label{bound:L-multi-freq}
   \int_0^{\omega_m} \int_{-1}^1 f \CalL f \dsigma
&= \int_0^{\omega_m} \int_{-1}^1 
      \frac{f}{\sqrt{\beta_\omega}}
      \vpran{\frac{f}{\sqrt{\beta_\omega}} 
                 - \vint{ \frac{f}{\sqrt{\beta_\omega}}}_{\mu, \omega}} \dsigma \nn
\\
&= \int_0^{\omega_m} \int_{-1}^1 
      \vpran{\frac{f}{\sqrt{\beta_\omega}} 
                 - \vint{ \frac{f}{\sqrt{\beta_\omega}}}_{\mu, \omega}}^2 \dsigma \geq 0 \,.
\end{align} 
This shows $\CalL$ is nonnegative and $\NullL = \Span\{\sqrt{\beta_\omega}\}$. By direct calculation we have $\vint{\mu \beta_\omega}_{\mu, \omega} = 0$. Hence $\sqrt{\beta_\omega} \in H^0$. Denoting $X_0 = \frac{\sqrt{\beta_\omega}}{\sqrt{\int_\VV \beta_\omega \dsigma}}$, one can verify by direct calculation that
\begin{align*}
    \CalL^{-1}(\mu X_0) = \frac{\vpran{\mu (\sqrt{\beta_\omega})^3}}{\sqrt{\int_\VV \beta_\omega \dsigma}} \,.
\end{align*}
Hence,
\begin{align*}
   \viint{\mu \CalL^{-1}(\mu \VecX_{0})}{f}_{\mu, \omega}
&= \frac{1}{\sqrt{\int_\VV \beta_\omega \dsigma}}\viint{\mu^2 \beta_\omega^2}{\frac{f}{\sqrt{\beta_\omega}}}_{\mu, \omega}
\\
&= \frac{1}{\sqrt{\int_\VV \beta_\omega \dsigma}}\viint{\mu^2 \beta_\omega^2}
          {\frac{f}{\sqrt{\beta_\omega}} - \vint{\frac{f}{\sqrt{\beta_\omega}}}_{\mu, \omega}}_{\mu, \omega}
   + \alpha_0 \vint{\frac{f}{\sqrt{\beta_\omega}}}_{\mu, \omega} \,,
\end{align*}
where $\alpha_0 = \frac{\vint{\mu^2 \beta_\omega^2}}{\sqrt{\int_\VV \beta_\omega \dsigma}} > 0$. Thus,
\begin{align*}
   \int_0^{\omega_m}\int_{-1}^1 f \CalL_d f \dsigma
&= \vint{f \CalL f}_{\mu, \omega} + \alpha \viint{\mu X_0}{f}_{\mu, \omega}^2 
   + \alpha \viint{\mu \CalL^{-1}(\mu X_0)}{f}_{\mu, \omega}^2
\\
& \geq \vint{f \CalL f}_{\mu, \omega}
   + \alpha \viint{\mu \CalL^{-1}(\mu X_0)}{f}_{\mu, \omega}^2
\\
& \geq 
      \vint{f \CalL f}_{\mu, \omega} + \frac{\alpha \alpha_0^2}{2} \vint{\frac{f}{\sqrt{\beta_\omega}}}_{\mu, \omega}^2
       -  \frac{\alpha}{\int_\VV \beta_\omega \dsigma} \viint{\mu^2 \beta_\omega^2}
          {\frac{f}{\sqrt{\beta_\omega}} - \vint{\frac{f}{\sqrt{\beta_\omega}}}_{\mu, \omega}}_{\mu, \omega}^2 
\\
& \geq
     \frac{1}{4} 
     \int_0^{\omega_m} \int_{-1}^1
     \vpran{\frac{f}{\sqrt{\beta_\omega}} 
                 - \vint{ \frac{f}{\sqrt{\beta_\omega}}}_{\mu, \omega}}^2 \dsigma
     + \frac{\alpha \alpha_0^2}{2} \vint{\frac{f}{\sqrt{\beta_\omega}}}_{\mu, \omega}^2 \,,
\end{align*}
where the last inequality follows from choosing $\alpha$ small enough and then applying Cauchy-Schwartz and \eqref{bound:L-multi-freq}. Let $\tilde\alpha = \min\{\frac{1}{4}, \frac{\alpha \alpha_0^2}{2}\}$. Then by the assumption that $\beta_\omega \in (0, b_0]$, we have
\begin{align*}
    \int_0^{\omega_m}\int_{-1}^1 f \CalL_d f \dsigma 
\geq 
   \tilde\alpha \norm{\frac{f}{\sqrt{\beta_\omega}}}_{L^2(\dsigma)}^2 
\geq
   \sigma_2 \norm{f}_{L^2(\dsigma)}^2  \,,
\qquad
   \sigma_2 = \frac{\tilde\alpha}{\tau_0} \,,
\end{align*}
which proves the coercivity of the damped operator $\CalL_d$ on $L^2(\dsigma)$. Hence $\CalL$ satisfies all the assumptions (\PL1)-(\PL5).
\end{proof}

The type of boundary conditions we use here is the case where the wall does not change the frequency of the phonon. More precisely,  for each $\omega > 0$, the boundary condition for the original nonlinear equation reads 
\begin{align*}
    \mathscr{K} F \big|_{\mu > 0}
&  = \alpha_d \int_{\mu<0} |\mu| v(\omega) F \dmu
       \frac{F_{BE}(\omega, T_0)}{\int_{\mu>0} |\mu| v(\omega) F_{BE}(\omega, T_0) \dmu}
       + \alpha_s F \big|_{\mu < 0} (-\mu, \omega) \,.
\\
& = 2\alpha_d \int_{\mu<0} |\mu| v(\omega) F \dmu
       + \alpha_s F \big|_{\mu < 0} (-\mu, \omega) \,.
\end{align*}
Linearizing $F$ as in~\eqref{eq:linearF} and~\eqref{linearization-2}, we obtain the linearized boundary operator $\CalK$ as
\begin{equation}
\begin{aligned} \label{cond:bdry-multi-freq}
   \CalK f
  = \alpha_d\CalK_d f + \alpha_s\CalK_s (f \big|_{\mu<0})
  = 2\alpha_d \int_{\mu<0} |\mu| f \dmu
       + \alpha_s f \big|_{\mu<0}(\CalR v) \,,
\end{aligned}
\end{equation}
where
$\alpha_d, \alpha_s \geq 0$, and $\alpha_d + \alpha_s < 1$. Now we verify that 
\begin{lem}\label{lem:phonon}
The boundary operator $\CalK$ defined in~\eqref{cond:bdry-multi-freq} satisfies ({\PK}).
\end{lem}
\begin{proof}
Again we only need to show that $\CalK_d$ satisfies (\PK).
By the definition of $\CalK_d$ in~\eqref{cond:bdry-multi-freq}, we have
\begin{align*}
   \int_{\mu > 0} \mu |\CalK_d f|^2 \dsigma
&= 4\int_{\mu > 0} \mu \vpran{\int_{\mu'<0} |\mu'| f(\mu', \omega) \dmu'}^2 \frac{C_\omega}{\Theta_0 \tau(\omega)} \domega \dmu 
\\
& \leq
     2 \int_{\mu > 0} \mu \vpran{\int_{\mu'<0} |\mu'| f^2(\mu', \omega) \frac{C_\omega}{\Theta_0 \tau(\omega)} \domega \dmu' } \dmu
\\
& = \int_{\mu < 0} |\mu| f^2(\mu, \omega) \dsigma \,,
\end{align*}
which shows (\PK) holds. 
\end{proof}

In summary, the damped equation has the form
\begin{equation} \label{soln:damp-phonon}
\begin{aligned}
    \mu \del_x f 
    = - \frac{1}{\sqrt{\beta_\omega}} \vpran{\frac{f}{\sqrt{\beta_\omega}}
       - \vint{\frac{f}{\sqrt{\beta_\omega}}}_{\mu, \omega}}
       &- \alpha \mu X_0 \viint{\mu X_0}{f} 
       - \alpha \mu \CalL^{-1}(\mu X_0) \viint{\mu \CalL^{-1}(\mu X_0)}{f} \,.
\end{aligned}
\end{equation}
The boundary condition is given as 
\begin{align}  \label{cond:bdry-multi-freq-recall}
    f \big|_{x=0} = h + \CalK f \,,
\qquad \mu > 0 \,,
\end{align}
where $\CalK$ is given in~\eqref{cond:bdry-multi-freq}.
We again have
\begin{align*}
    (\I + \bar\CalK)^{-1} 
    &= \frac{1}{1+\alpha_s} \vpran{\I - \frac{\alpha_d}{1+\alpha_d+\alpha_s} \bar\CalK_d}\,,
\\
   (\I + \bar\CalK)^{-1} (\I - \bar\CalK) 
   &= \frac{1-\alpha_s}{1+\alpha_s} 
     \vpran{\I - \frac{2\alpha_d}{(1-\alpha_s)(1+\alpha_d+\alpha_s)} \bar\CalK_d} \,,
\end{align*}
where
\begin{align*}
   \bar\CalK_d f
   = 2\int_{\mu > 0} \mu f \dmu \,.
\end{align*}
The special solution $g_0$ is constructed as
\begin{equation}
\begin{aligned} \label{eq:phonon-g-0}
    \mu \del_x g_0
    &+  \CalL_d g_0 = 0 \,,
\\   
   g_{0} \big|_{\mu>0} 
 = &\vpran{X_{0} - \CalK (X_{0} \big|_{\mu<0})} + \CalK (g_{0} \big|_{\mu<0}) \,,
\qquad
   \mu > 0 \,.
\end{aligned}
\end{equation}
Finally, the exact solution for equation~\eqref{eq:phonon-4} with boundary condition~\eqref{cond:bdry-multi-freq} is
\begin{align*}
    f_h = f - c_h g_0 + c_h X_0 \,,
\end{align*}
where $f$ solves the damped equation~\eqref{soln:damp-phonon} with the boundary condition~\eqref{cond:bdry-multi-freq-recall}  and $\displaystyle c_h = \frac{\viint{\mu X_0}{f}_{\mu, \omega}}{\viint{\mu X_0}{g_0}_{\mu, \omega}}$.

\subsubsection{Numerical Results}
We discretize the $\omega$-variable uniformly and replace the integral in $\omega$ by the trapezoidal rule. The resulting system can be viewed as a multi-species system. Hence the construction of basis functions is the same as in Section~\ref{sec:algorithm_ms}.

We again show examples with both pure incoming data and Maxwell boundary condition. For computational convenience, we modify the equilibrium state as 
\begin{align*}
   F_{BE}(\omega, T) = e^{-\hbar\omega/(k_B T)} \,.
\end{align*}

\smallskip

\noindent \emph{Example 4.3.1 Incoming boundary condition.} \,\,
In the first example, we set 
\begin{align*}
   \frac{C_\omega}{\tau(\omega)}
   =\omega\exp\left(-\omega/1000\right) \,,
\qquad
   \tau(\omega) v(\omega) = \frac{1}{\omega} \,,
   \qquad
   \omega\in[1,8].
\end{align*} 
Then
$\NullL = \Span\{\sqrt{\tau(\omega) v(\omega)}\} =
\Span\{\sqrt{1/\omega}\}$.
Figure~\ref{fig:MF_Dirichlet_recover} shows that if $h = X_0$, then
the numerical solution is in good agreement with the analytical
solution where $f_h = X_0$.
\begin{figure}
\centering
\includegraphics[height = 0.2\textheight,width = 0.23\textwidth]{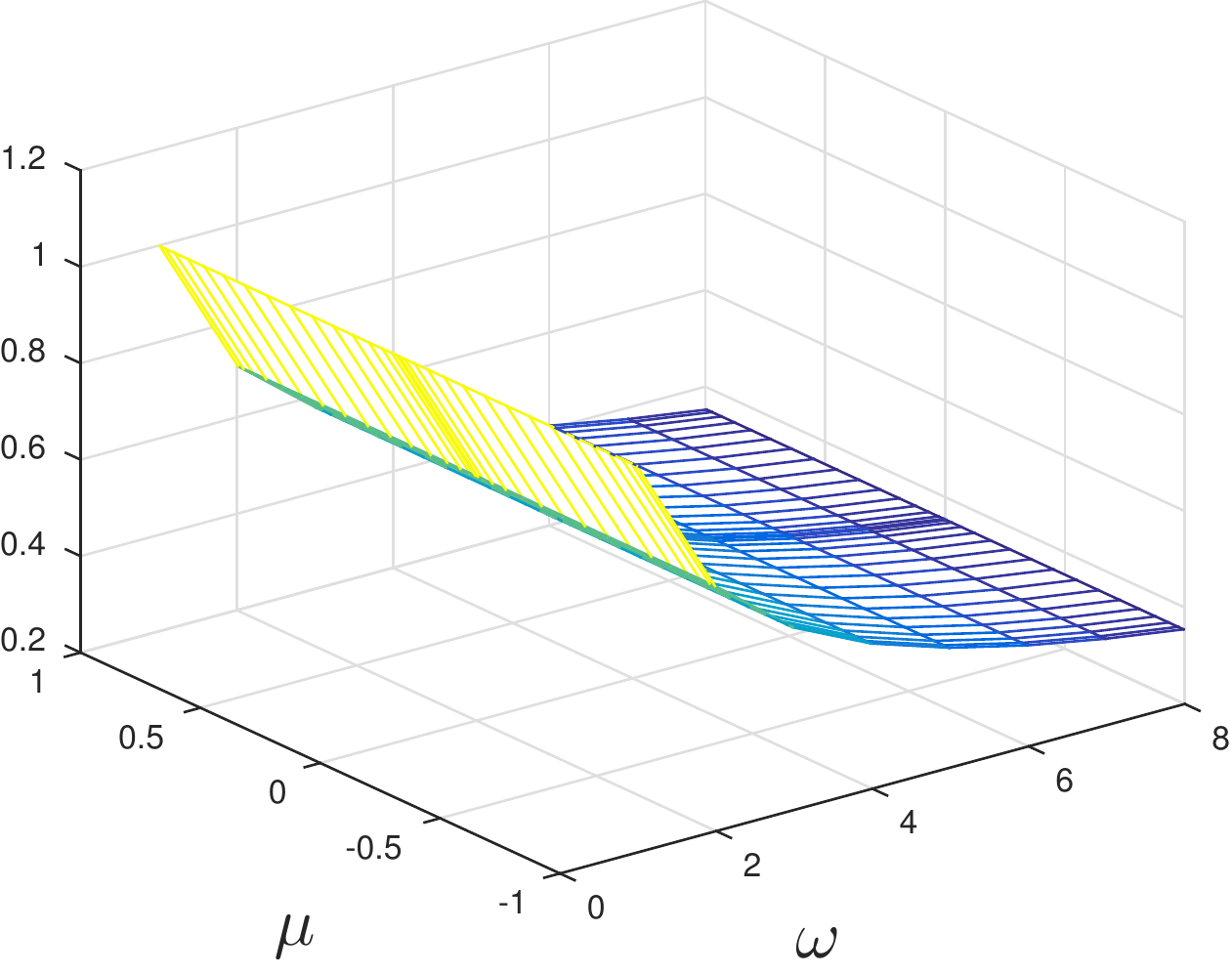}
\includegraphics[height = 0.2\textheight,width = 0.23\textwidth]{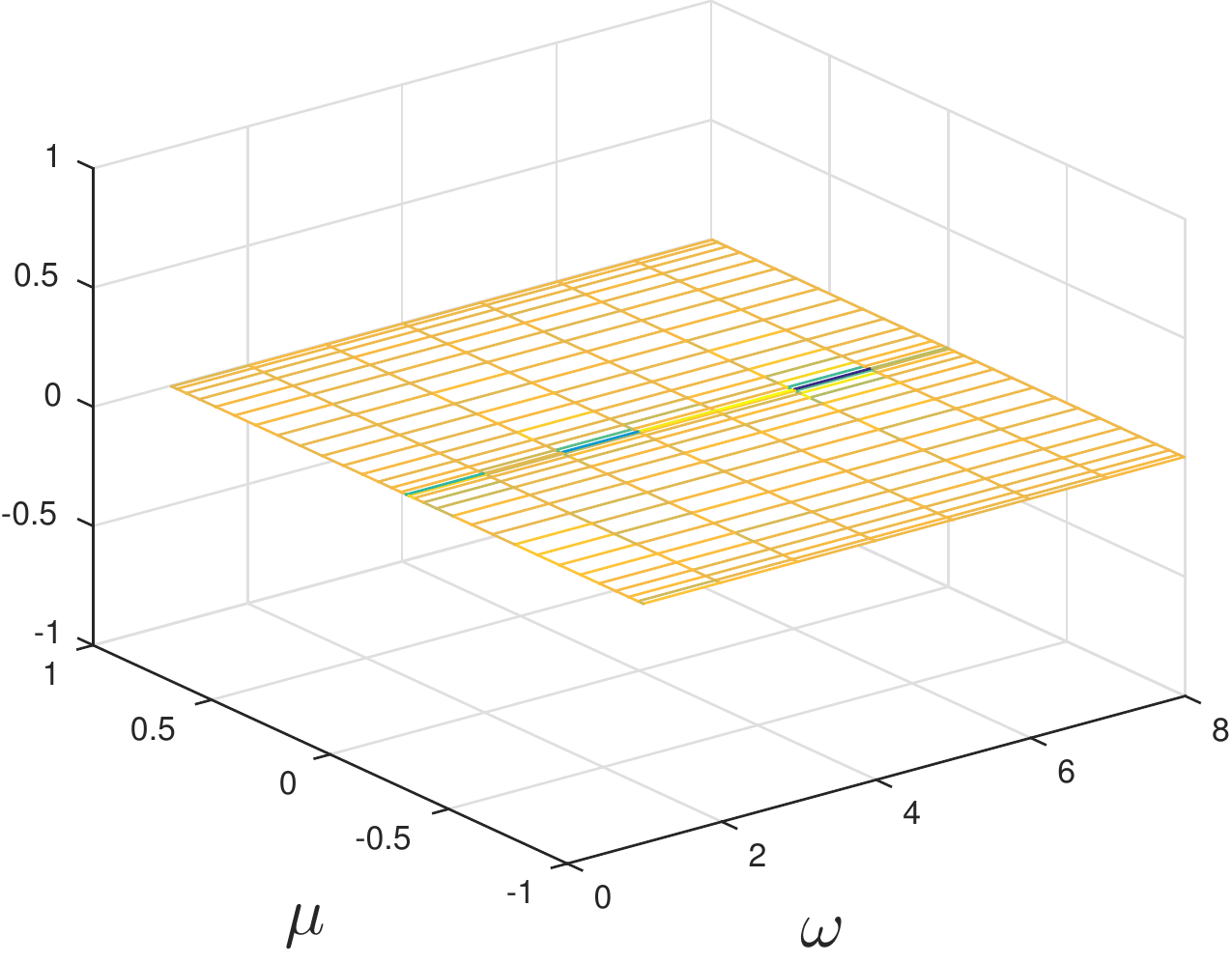}
\includegraphics[height = 0.2\textheight,width = 0.23\textwidth]{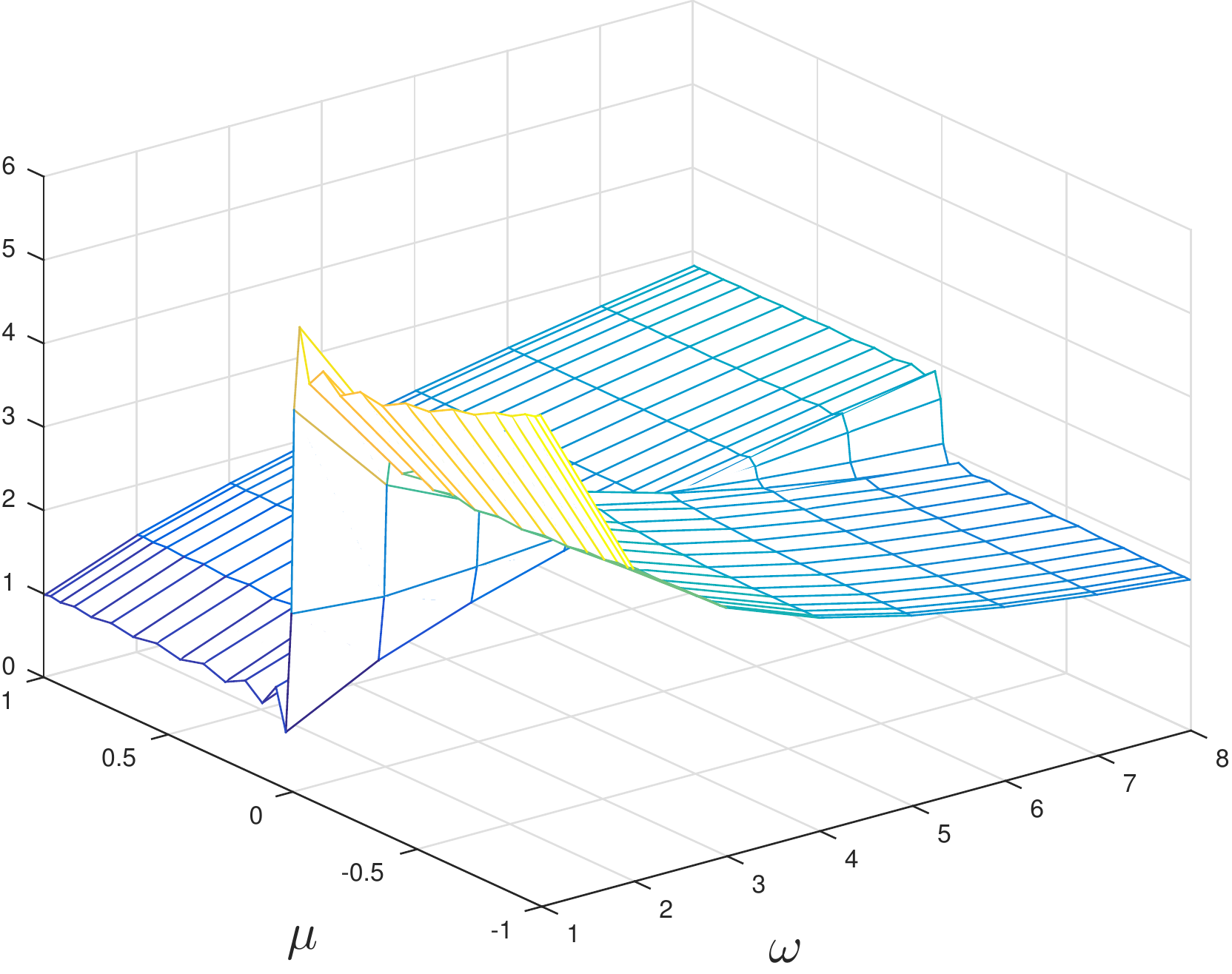}
\includegraphics[height = 0.2\textheight,width = 0.23\textwidth]{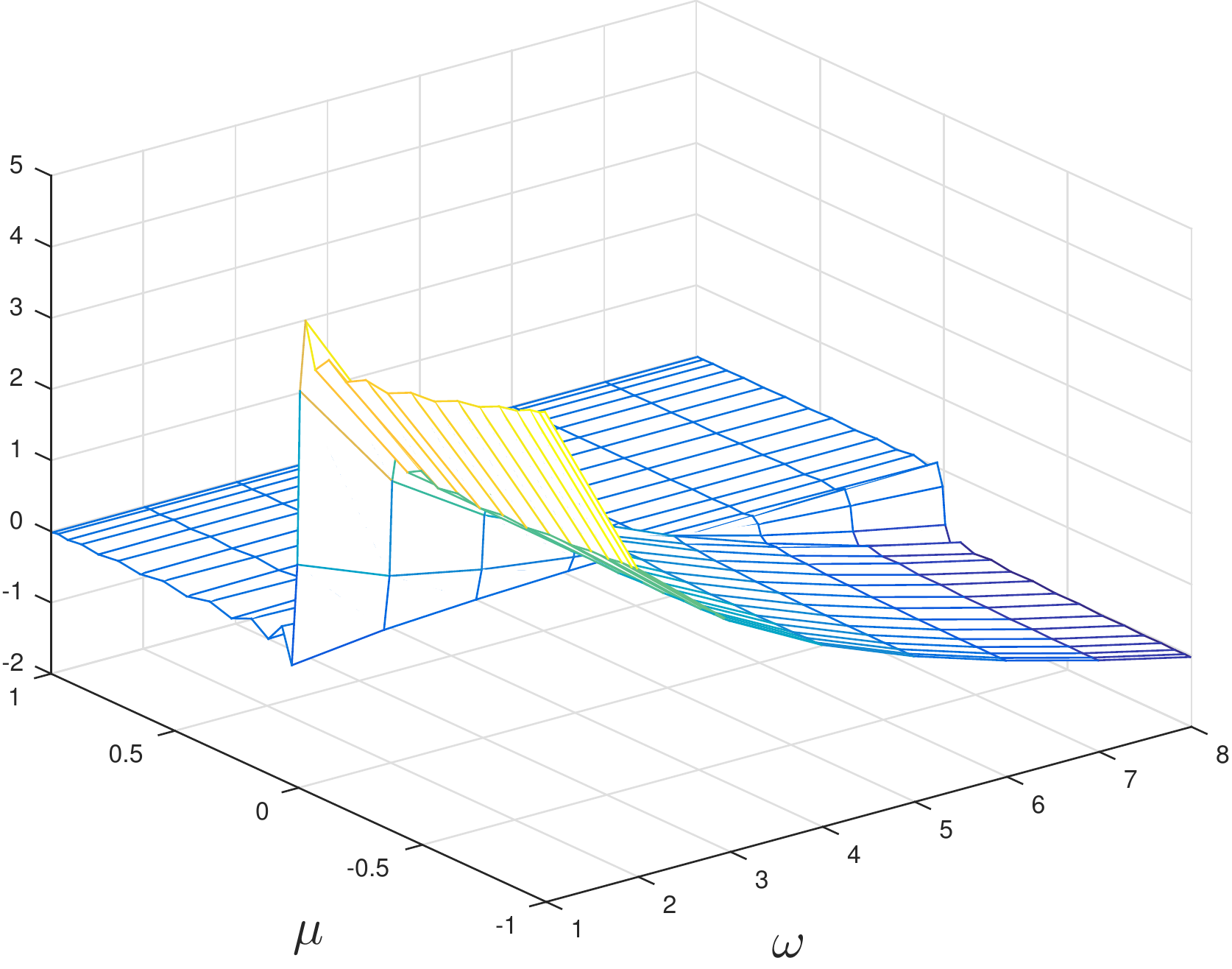}
\caption{(Example 4.3.1) The two figures on the left are boundary data and the difference ($f_h(x=0) - h$) obtained for $h = \sqrt{1/\omega} \in\NullL$. The two figures on the right are the solution at the boundary and the difference ($f_h(0, \mu, \omega)-h$) obtained for $h = \sqrt{\omega}$. In this case we expect the difference to be of order $\mathcal{O}(1)$. In both cases, we use $31$ basis functions in $\mu$ direction and sample $8$ grid points along $\omega$.}\label{fig:MF_Dirichlet_recover}
\end{figure}

\smallskip

\noindent\emph{Example 4.3.2  Maxwell boundary condition.}
In the second example, we take the same $C_{\omega}$, $\tau(\omega)$
and $v(\omega)$ as in the previous example and set the accommodation
coefficients as $(\alpha_d, \alpha_s) = (0.3,0.4)$. Once again if
$h = X_0 - \CalK(X_0 \big|_{\mu<0})$, the exact solution must be
$f_h = X_0$. The numerical solution demonstrated in
Figure~\ref{fig:MF_Maxwell_recover} shows a good match with the exact
solution.
\begin{figure}
\centering
\includegraphics[height = 0.18\textheight,width = 0.23\textwidth]{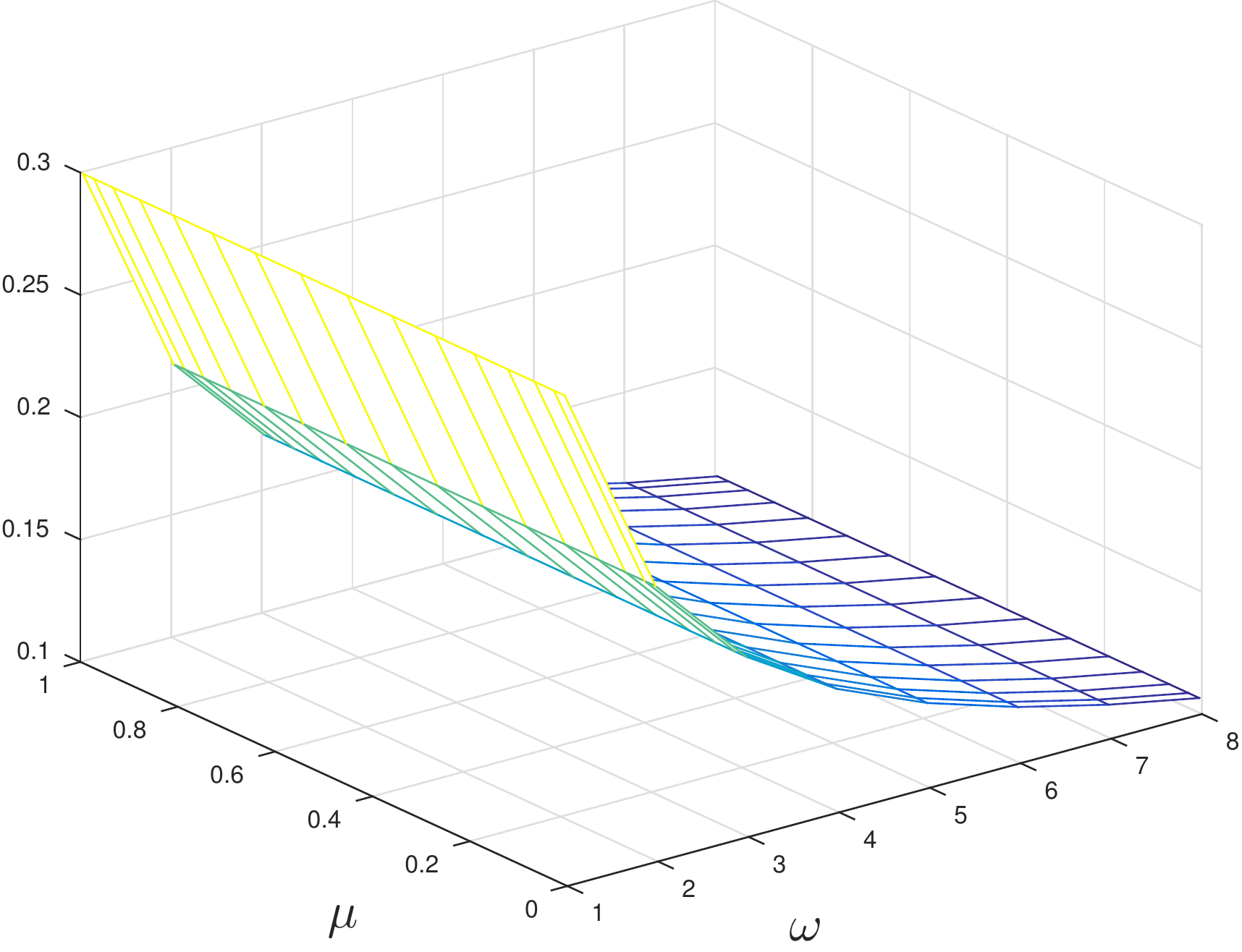}
\includegraphics[height = 0.18\textheight,width = 0.23\textwidth]{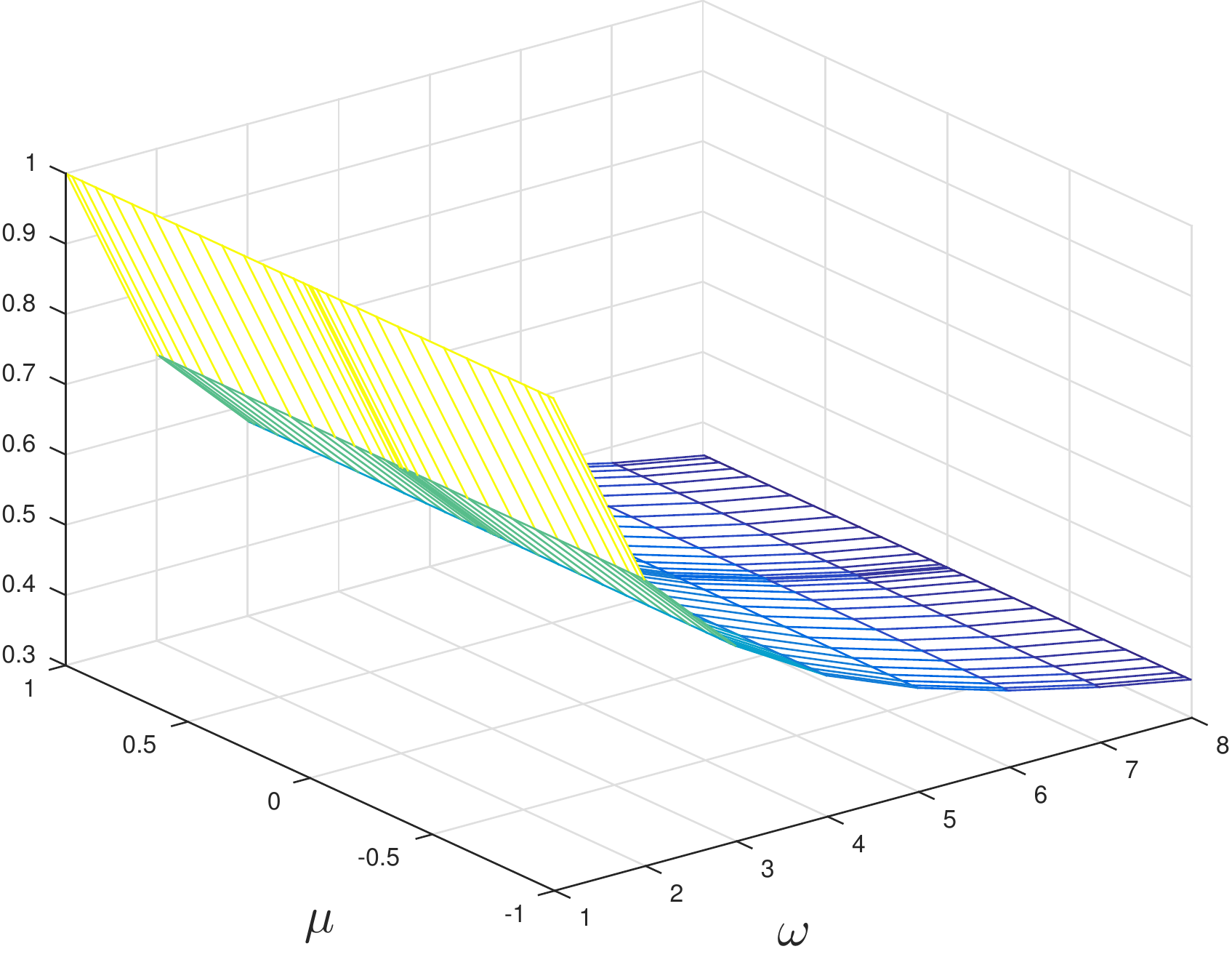}
\includegraphics[height = 0.18\textheight,width = 0.23\textwidth]{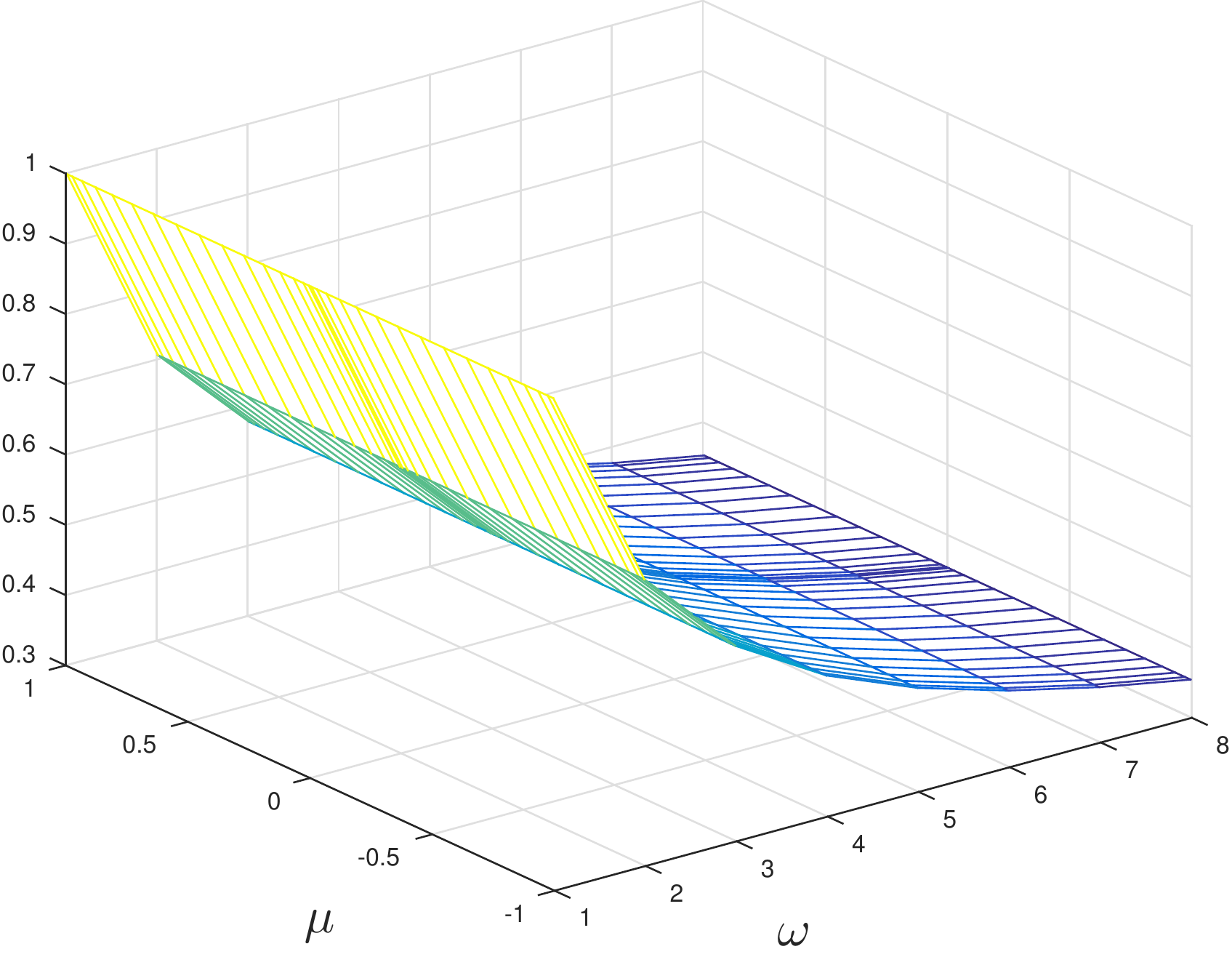}
\includegraphics[height = 0.18\textheight,width = 0.23\textwidth]{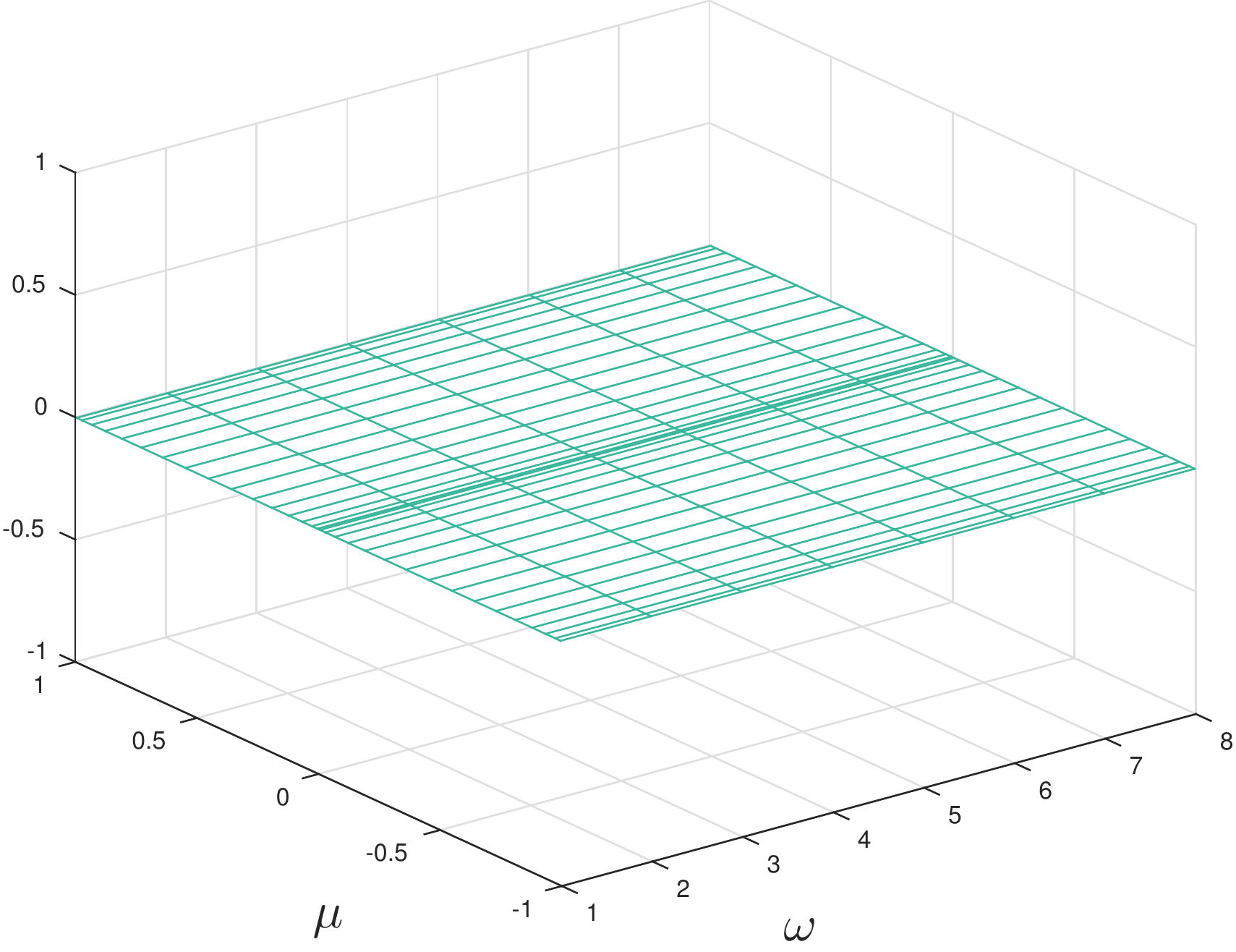}
\caption{(Example 4.3.2) In the second row we use
  $h = X_0 - \CalK(X_0 \big|_{\mu<0})$ as the incoming Dirichlet
  data. The four plots show $h$, $X_0$, numerical result $f_h$ and the
  recovery difference $f_h-X_0$. The accommodation coefficients are
  set as $(\alpha_d, \alpha_s) = (0.3,0.4)$. We recover the exact
  solution $X_0$ as expected. In this example, we use $31$ basis
  functions in $\mu$ direction and sample $8$ grid points along
  $\omega$. Note that the domain size for $h$ is only half of that for
  $X_0$.}\label{fig:MF_Maxwell_recover}
\end{figure}

\bibliographystyle{amsxport}
\bibliography{MaxwellBC}

\end{document}